\documentclass[10pt,reqno]{amsart}
\usepackage[colorlinks=true,
pdfstartview=FitV, linkcolor=cyan, citecolor=magenta,
urlcolor=blue]{hyperref}
\usepackage{amsmath,amsfonts,latexsym,amssymb,dsfont}
\usepackage[utf8]{inputenc}
\usepackage{graphicx,subfig,accents}
\usepackage{pdfsync}
\usepackage{color}
\newtheorem{theorem}{Theorem}[subsection]
\newtheorem{lemma}[theorem]{Lemma}
\newtheorem{definition}[theorem]{Definition}
\newtheorem{proposition}[theorem]{Proposition}
\newtheorem{corollary}[theorem]{Corollary}

\newcommand{\rmks}{{\sc Remarks: }}

\newcommand{\ii}{{\iota}}

\renewcommand{\gg}{{\ms g}}
\newcommand{\hh}{{\ms h}}

\newcommand{\dual}{{\circ}}

\newcommand{\ww}{{\ms W}}
\newcommand{\tw}{{\ms T}}

\newcommand{\tr}{\operatorname{tr}}
\newcommand{\BB}{\mathds{B}}
\newcommand{\bb}{\mathbf b}
\newcommand{\pp}{\ms p}
\newcommand{\PP}{\ms P}
\newcommand{\DD}{\ms D}
\renewcommand{\leq}{\leqslant}
\renewcommand{\geq}{\geqslant}
\newcommand{\cA}{\mathcal A}
\newcommand{\cB}{\mathcal B}

\newcommand{\cL}{\mathcal L}

\newcommand{\cQ}{\mathcal Q}

\newcommand{\cF}{\mathcal F}

\newcommand{\Eig}[1]{[#1]}
\newcommand{\sln}{\ms{PSL}_n(\mathbb R)}
\newcommand{\sld}{\ms{PSL}_2(\mathbb R)}
\newcommand{\End}{\operatorname{End}}
\newcommand{\Hn}{{\ms H}(n,S)}

\newcommand{\otw}{\overline{\ww}_p}
\newcommand{\Hom}{\operatorname{Hom}}
\newcommand{\gir}{\operatorname{gh}}
\newcommand{\pn}{\mathbb P(\mathbb R^n)}
\newcommand{\dd}{{\rm d}}
\newcommand{\ms}{\mathsf}
\newcommand{\mk}{\mathfrak}
\newcommand{\bu}{\accentset{\bullet}}
\newcommand{\logd}[2]{\frac{\{#1,#2\}_S}{#1 #2}}
\newcommand{\defeq}{{}\mathrel{\mathop:}={}} 
\newcommand{\eqdef}{{}=\mathrel{\mathop:}{}}

\newcommand{\auteur}{
\vskip 2truecm
\centerline{Université Côte d’Azur, CNRS, LJAD, France}}
\newcommand{\seq}[1]{
\{{#1}_m\}_{m\in\mathbb N}}

\newcommand{\grf}{\pi_1(S)}
\newcommand{\wh}{\widehat}

\newcommand{\bgrf}{\partial_\infty\pi_1(S)}
\renewcommand{\d}{\bu}
\renewcommand{\epsilon}{\varepsilon}
\newcommand{\jl}{\divideontimes}
\newcommand{\sig}{\operatorname{Sign}}

\newcommand{\id}{\operatorname{Id}}

\title{Goldman Algebra, Opers \\ and the Swapping Algebra}
\author{François LABOURIE}
\thanks{Université Côte d'Azur, CNRS, LJAD, France. The research leading to these results has received funding from the European Research Council under the {\em European Community}'s seventh Framework Programme (FP7/2007-2013)/ERC {\em grant agreement} n$^o$FP7-246918, as well as well as from the ANR program ETTT (ANR-09-BLAN-0116-01) }
\begin{document}

\maketitle
\newtheorem{theotheo}[foo]{Theorem}
\newtheorem{propro}[foo]{Proposition}

\begin{abstract} We define a Poisson Algebra called the {\em swapping algebra} using the intersection of curves in the disk. We interpret a subalgebra of the fraction algebra of the swapping algebra -- called the {\em algebra of multifractions} --  as an algebra of functions on the space of cross ratios and thus as an algebra of functions on the Hitchin component as well as on the space of $\ms{SL}_n(\mathbb R)$-opers with trivial holonomy. We relate this Poisson algebra to the Atiyah--Bott--Goldman symplectic structure and to the Drinfel'd--Sokolov reduction. We also prove an extension of Wolpert formula.
\end{abstract}

\section{Introduction} The purpose of this article is threefold. We first introduce the {\em swapping algebra} which is a Poisson algebra generated -- as a commutative algebra -- by pairs of points on the circle. Then we relate this construction to two well known Poisson structures:
\begin{itemize}
\item the Poisson structure of the character variety of representations of a surface group in $\sln$ discovered by Atiyah, Bott and Goldman
 \cite{Atiyah:1983,Goldman:1984,Goldman:1986}
\item the Poisson structure of the space of $\sln$-opers introduced by Dickey, Gel'fand and Magri and described in a geometrical way by Drinfel'd and Sokolov \cite{Magri:1978gh,Drinfelcprimed:1981ua,Dickey:1997un} \end{itemize}
One way to interpret heuristically these relations is to say that the swapping algebra embodies the notion of a ``Poisson structure" for the space of all cross ratios, space that contains both the space of opers and the ``universal (in genus) Hitchin component". As a byproduct of the methods of this paper, we also produce a generalisation of the Wolpert formula which computes the brackets of length functions for the Hitchin component. 

The results of this article were announced in \cite{Labourie:2010ev} The relation -- at a topological level -- between the character variety and opers was already noted by the author in \cite{Labourie:2006}, by Fock and Goncharov in \cite{Fock:2006a} and foreseen by Witten in \cite{Witten:to} (see also \cite{Govindarajan:1995cr,Govindarajan:1995im}).
I thank Martin Bridson, Sergei Fomin, Louis Funar and Bill Goldman  for their interest and help.

We now explain more precisely the content of this article.

\subsection{The swapping algebra}
Our first result is the construction of the {\em swapping algebra}. To avoid cumbersome expressions, we shall denote most of the time the ordered pair $(X,x)$ of points of the circle by the concatenated symbol $Xx$. We recall in Paragraph \ref{sec:link} the definition and properties of the linking number $[Xx,Yy]$ of the two pairs $(X,x)$ and $(Y,y)$. If $\PP$ is a subset of the circle, we denote by $\cL(\PP)$  the commutative associative algebra generated by pairs of points of $\PP$ with the relations $XX=0$, for all $X$ in $\PP$. Our starting result is the following
\begin{theotheo}{\sc [Swapping Bracket]}\label{theo:A}
For every complex number $\alpha$, there exists a unique Poisson bracket on $\cL(\PP)$ such that the bracket of two generators is
$$
\{X x,Yy\}_\alpha\defeq [Xx,Yy] (X y.Y x +\alpha. X x.Yy).
$$
\end{theotheo}
The {\em swapping algebra} is the algebra $\cL_\alpha(\PP)$ endowed with the Poisson bracket $\{ \cdotp, \cdotp\}_\alpha$. This theorem is proved in Section \ref{sec:swap}. The goal of this paper is to relate this swapping algebra to other Poisson algebras. 

One should note that this bracket can be used to express very simply some results of Wolpert and in particular the variation of the length of curve transverse to a shear \cite{Wolpert:1983td,Wolpert:1982eo}.

\subsection{Cross ratios and the multifraction algebra} We shall concentrate on the interpretation of an offshoot of the swapping algebra. We denote by $\mathcal Q_\alpha(\PP)$ the algebra of fractions of $\cL_\alpha(\PP)$ equipped with the induced Poisson structure. The {\em multifraction algebra} $\mathcal B(\PP)$ is the vector subspace of $\mathcal Q_\alpha(\PP)$ generated by
the {\em elementary multifractions}: $$
[{\rm X},{\rm x};\sigma]\defeq \frac{\prod_{i=1}^{i=n}X_ix_{\sigma(i)}}{ \prod_{i=1}^{i=n}X_ix_{i}},
$$
where ${\rm X}=(X_1,\ldots,X_n)$ and ${\rm x}=(x_1,\ldots,x_n)$ are $n$ tuples of points of $\PP$ and $\sigma$ is a permutation of $\{1,\ldots,n\}$. Then we have the easy proposition
\begin{propro}
The multifraction algebra is a Poisson subalgebra of $\mathcal Q_\alpha(\PP)$. The induced Poisson structure does not depend on $\alpha$. Finally $\mathcal B(\PP)$ is generated as a commutative algebra by the {\em cross fractions}:
$$
[X,Y,x,y]\defeq \frac{Xx.Yy}{Yx.Xy}.
$$
\end{propro}
In particular, it follows that the multifraction algebra is naturally mapped to the commutative algebra of functions on cross ratios (See Section \ref{sec:crossratio}). Thus the existence of a Poisson structure on the multifraction algebra can be interpreted as that of a Poisson structure on the space of cross ratios.
\subsection{The multifraction algebra as a ``universal" Goldman algebra}
We then relate the multifraction algebra to the Goldman algebra. Let $\Gamma$ be the fundamental group of a surface $S$, $\partial_\infty\Gamma$ the boundary at infinity of $\Gamma$, and $\PP$ be the subset of $\partial_\infty \Gamma$ consisting of fixed points of elements of $\Gamma$. The Hitchin component $\Hn$ of the character variety of representations of $\Gamma$ in $\sln$ was interpreted in \cite{Labourie:2005} as a space of cross ratios. Thus every multifraction in $\mathcal B({\ms P})$ gives a smooth function on the Hitchin component (see Proposition \ref{prop:smoothmulti} for details). Thus we have a restriction
$$
{\ms I}_S: \cB(\PP)\to C^\infty(\Hn).
$$
This mapping is not a Poisson morphism, nevertheless it becomes so when we take sequences of well chosen finite index subgroups. More precisely,
we define and prove, as an immediate consequence of one of the main result of Niblo in \cite{Niblo:1992uy}, the existence of {\em vanishing sequences} of finite index subgroups $\{\Gamma_n\}_{n\in\mathbb N}$ of $\Gamma$; these sequences are essentially such that every geodesic becomes eventually simple and for which the intersection of two geodesics becomes eventually minimal (See Paragraph \ref{par:vaseq} and Appendix \ref{sec:vanishexist} for precisions). 

Then denoting by $\{ \cdotp, \cdotp\}_W$ the swapping bracket and $\{ \cdotp, \cdotp\}_{S_n}$ the Goldman bracket for $S_n\defeq \tilde S/\Gamma_n$ coming from the Atiyah--Bott--Goldman symplectic form on the character variety, we prove in Section \ref{sec:gold}
\begin{theotheo}\label{theo:C}{\sc [Goldman bracket for vanishing sequences ]}

Let $\seq{\Gamma}$ be a vanishing sequence of subgroups of $\pi_1(S)$. Let $\PP\subset\bgrf$ be the set of end points of geodesics. Let $b_0$ and $b_1$ be two multifractions in $\mathcal B(\PP)$. Then we have, 
\begin{eqnarray}
\lim_{n\rightarrow\infty}\{b_0,b_1\}_{{S}_n}=\{b_0,b_1\}_{W}.\label{theo:swapuniv0}
\end{eqnarray}
\end{theotheo}
The statement of this theorem actually requires some preliminaries in defining properly the meaning of Assertion \eqref{theo:swapuniv0}. In a way, this result tells us that the swapping bracket is the Goldman bracket on the universal solenoid.

The proof relies on the description of special multifractions called {\em elementary functions} (see Paragraph \ref{sec:elemfunct}) as limits of the well studied functions on the character variety known as {\em Wilson loops}. 

Another result is a precise asymptotic formula, on a fixed surface this time, relating the Goldman and the swapping brackets. Let $\Gamma$ be as above. Let $\gamma\in\Gamma$. 
Let finally $y\in \PP$, and $\gamma^+$, $\gamma^-$ be respectively the attractive and repulsive fixed points of $\gamma$ in $\partial_\infty(\Gamma)$. Let us define the following formal series of cross fractions, reverting to the notation $(X,x)$ for pairs,

$$
\hat\ell_\gamma(y)=\frac{1}{2}
\log\left(\frac{
(\gamma^+,\gamma(y))\cdot(\gamma^-,\gamma^{-1}(y))}
{(\gamma^+,\gamma^{-1}(y))\cdot(\gamma^-,\gamma(y))}\right)
$$

In \cite{Labourie:2006} we show that the {\em period function} $\ell_\gamma\defeq {\ms I}_S(\hat\ell_\gamma(y)))$ -- seen as  a function on the character variety -- is independent on $y$ and is a function of the eigenvalues of the monodromy of $\gamma$. These period functions coincide with the length functions for the classical Teichmüller Theory -- that is $n=2$. 

We now have 
\begin{theotheo}\label{theo:D}{\sc[Bracket of length functions]}
Let $\gamma$ and $\eta$ be homotopy classes of curves which as simple curves have at most one intersection point, then we have 
$$
\lim_{n\to\infty}{\rm I}_S(\{\hat\ell_{\gamma^n}(y),\hat\ell_{\eta^n}(y)\}_W)=\frac{1}{4}\{\ell_{\gamma},\ell_{\eta}\}_S
$$
\end{theotheo}
As a tool of the proof of this result we prove the following extension of the Wolpert formula \cite{Wolpert:1983td,Wolpert:1982eo}
\begin{theotheo}\label{theo:DW}{\sc [Generalised Wolpert Formula]}
Let $\gamma$ and $\eta$ be two homotopy classes of curves which as simple curves have exactly one intersection point. Then the Goldman bracket of the two length functions $\ell_\gamma$ and $\ell_\eta$ is 
\begin{equation}
\{\ell_\gamma,\ell_\eta\}_S({\bf b})=\ii(\gamma,\eta)\sum_{\epsilon,\epsilon^\prime\in\{-1,1\}}\epsilon\epsilon^\prime. {\bf b}\left(\gamma^\epsilon,\eta^{\epsilon'},\gamma^{-\epsilon},\eta^{-\epsilon'}\right).
\end{equation}
\end{theotheo}
This Formula has recently been extended using different methods par Bridgeman in \cite{Bridgeman:2015wn}.

\subsection{The multifraction algebra and $\sln$-opers} We finally relate the multifraction algebra to opers. We recall in Section \ref{sec:oper} the definition of real opers and their interpretation as maps to the projective space $\mathbb P(\mathbb R^n)$ and its dual. In particular, opers with trivial holonomy can be embedded in the space of smooth cross ratios. The Drinfel'd-Sokolov reduction allows us to define the Poisson bracket of pairs of {\em acceptable observables}, a subclass of functions on the spaces of opers. We then show that this Poisson bracket coincide with the swapping bracket 

\begin{theotheo}{\sc[Swapping Bracket and opers]}\label{theo:E}

Let $(X_0,x_0,Y_0,y_0,X_1,x_1,Y_1,y_1)$ be pairwise distinct points on the circle $\mathbb T$. Then the cross fractions $[X_0,x_0,Y_0,y_0]$ and $[X_1,x_1,Y_1,y_1]$ defines a pair of acceptable observables whose Poisson bracket with respect to the Drinfel'd-Sokolov reduction coincides with their Poisson bracket in the multifraction algebra.
\end{theotheo}

\tableofcontents
\section{The swapping bracket}\label{sec:swap}
In this section, we first recall the properties and definition of the linking number of two ordered pairs of points. We then construct the swapping algebra and prove Theorem \ref{theo:A} which relies on an identity involving the linking numbers of six points.
\subsection{Linking number for pairs of points}\label{sec:link}
We recall that if $(X,x,Y,y)$ is a quadruple of points on the real line the {\em linking number} of $(X,x)$ and $(Y,y)$ is
\begin{eqnarray}
[Xx,Yy]&\defeq &\frac{1}{2}\big(\sig(X-x)\sig(X-y)\sig(y-x)\cr & & -\sig(X-x)\sig(X-Y)\sig(Y-x)\big),\label{In0}
\end{eqnarray}
where $\sig(x)=-1,0,1$ whenever $x<0$, $x=0$ and $x>0$ respectively. By definition, the linking number is invariant by orientation preserving homeomorphisms of the real line.
\begin{enumerate}
\item
When the four points are pairwise distinct, this linking number is also the total linking number of the curve joining $X$ to $x$ with the curve joining $Y$ to $y$ in the upper half plane.

\item The equality cases are as follows:
\begin{enumerate}
\item
For all points $(X,Y,y)$ on the circle
\begin{equation}
[XX,Yy]=0=[Xy,Xy]\, . \label{XX}
\end{equation}
\item If, up to cyclic permutation, $(X,Y,x)$ are pairwise distinct  points and oriented, then 
\begin{equation}[Xx,Yx]=1/2\, .\label{1/2}
\end{equation}
\end{enumerate}
\end{enumerate} 
The first observation shows that we can define the linking number of a quadruple of points on the oriented circle $S^1$ by choosing a point $x_0$ disjoint from the quadruple and defining the linking number as the linking number of the quadruple in $S^1\setminus\{x_0\}\sim\mathbb R$. The linking number so defined does not depend on the choice of $x_0$ and is invariant under orientation preserving homeomorphisms.

\subsubsection{Properties of the linking number} We abstract the useful property (for us) of the linking number of pairs of points in the following definition.
Let ${\PP}$ be any set.

\begin{definition} A {\em linking number} on pair of points of ${\PP}$ is a map from ${\PP}^4$ to a commutative ring
$$
(X,x,Y,y)\to [Xx,Yy],
$$
so that for all points $X,x,Y,y,Z,z$ 
\begin{align}
[Xx,Yy]+[Yy,Xx]&=0\label{In1}\ ,&\hbox{ \sc First antisymmetry,}&\\
\lbrack Xx,Yy]+[Xx,yY] &=0\label{In2}\ ,&\hbox{ \sc Second antisymmetry,}&\\
\lbrack zy,XY]+[zy,YZ]+[zy,ZX]&=0\label{In2b}\ ,&\hbox{ \sc Cocycle identity,}&
\end{align}
and moreover if $(X,x,Y,y)$ are all pairwise distinct then
\begin{align}
\ \ \ \ \ \ \ \ \ \ \ \ [Xx,Yy].[Xy,Yx]&=0\ , \label{prodIn}&\hbox{\sc Linking number alternative.}&
\end{align}
\end{definition}
We illustrate the cocycle identity and the alternative for the standard linking number in Figure \eqref{fig:1}

\begin{figure}[h]
 \centering
 \subfloat[Cocycle identity]{\label{fig:1a}\includegraphics[width=0.5\textwidth]{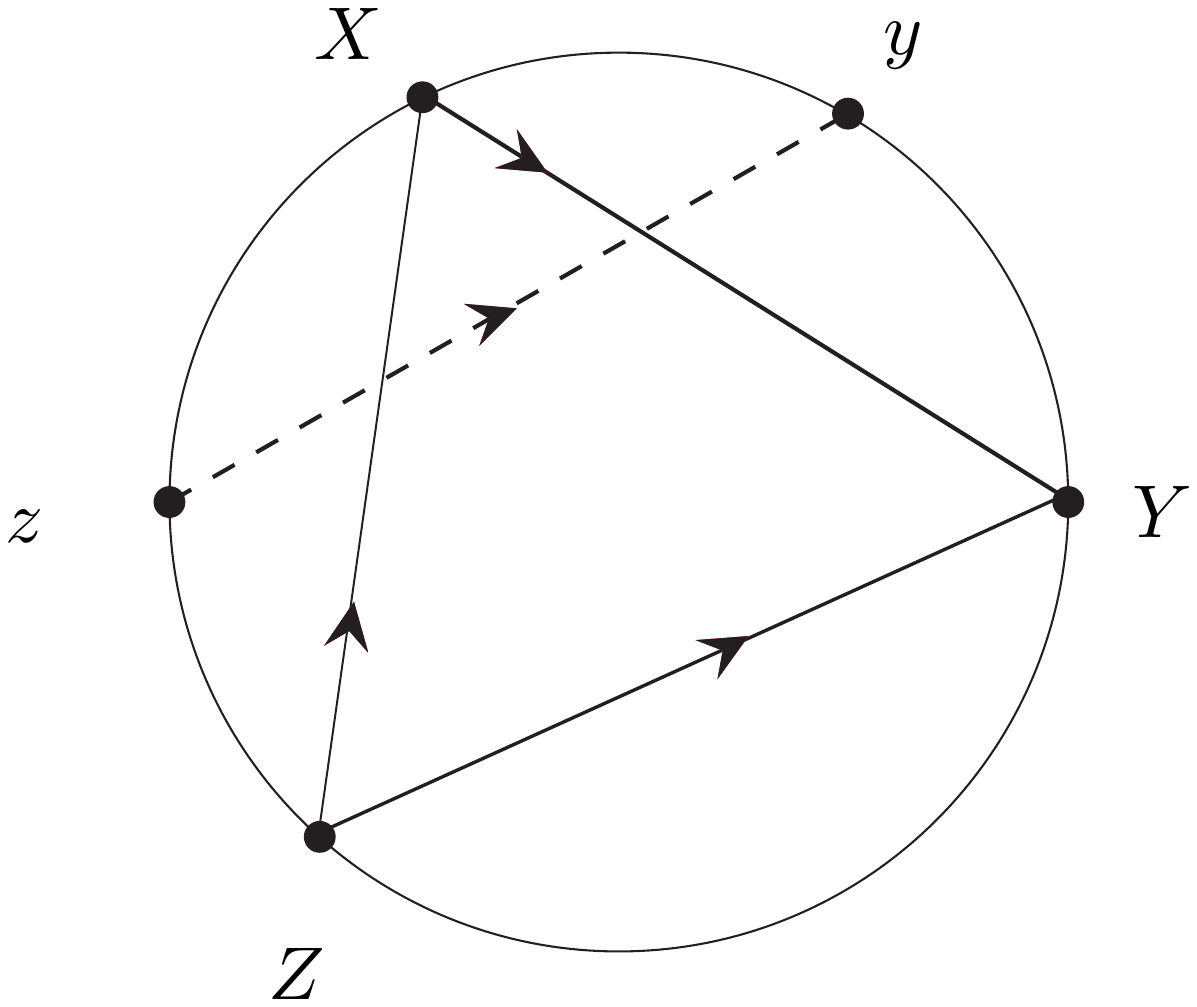} }        
 \subfloat[ Linking number alternative]{\label{fig:1b}\includegraphics[width=0.5\textwidth]{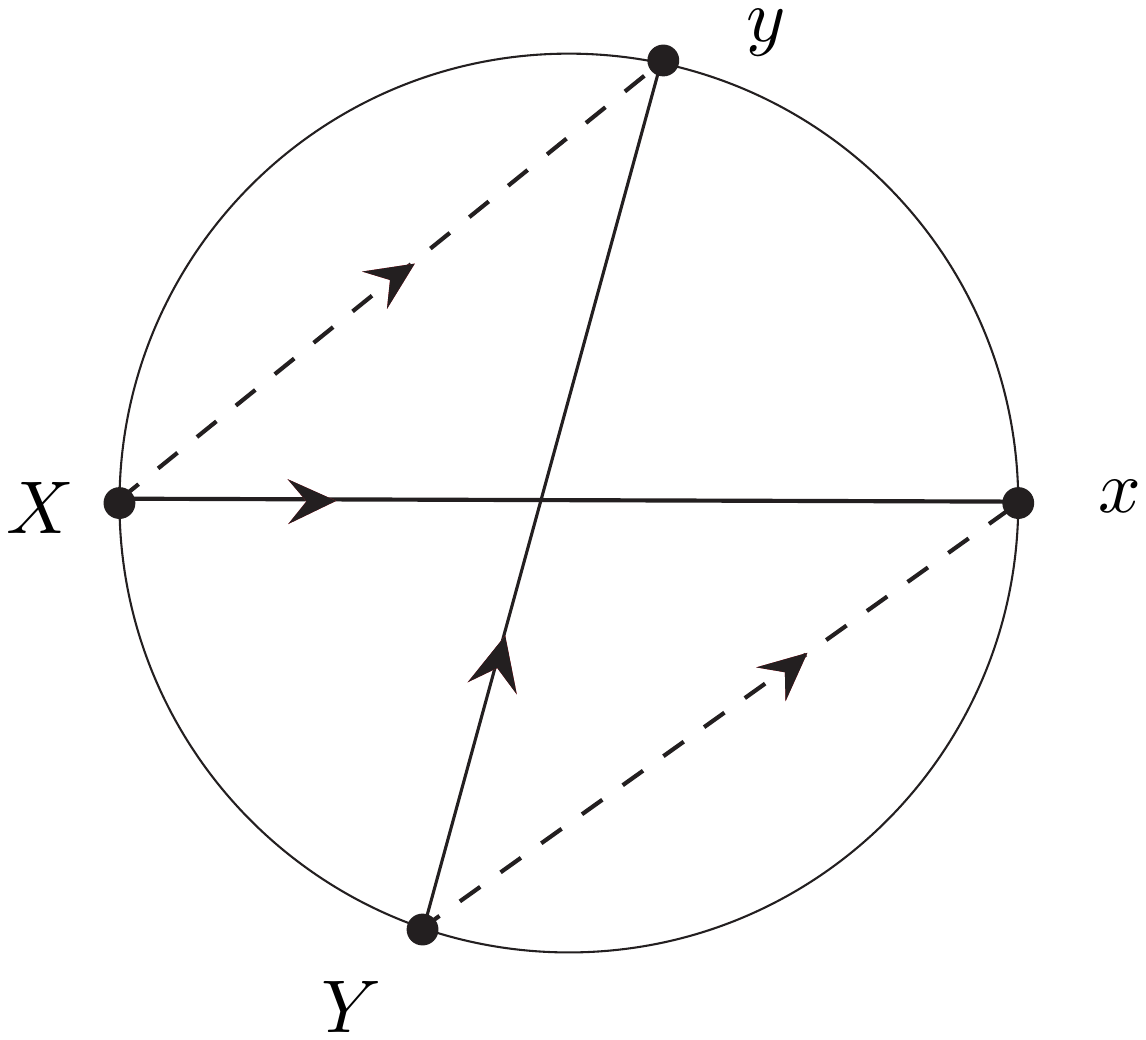}}
 \caption{linking number for pairs of points on the circle}
 \label{fig:1}
\end{figure}
Then we prove
\begin{proposition}
The canonical linking number for pairs of points of the circle is a linking number in the sense of the previous definition. 
\end{proposition}

\begin{proof}
The first two symmetries are checked from the definition. When $\{x,y\}\cap \{X,Y,Z\}=\emptyset$, then Equation \eqref{In2b} follows from the geometric definition of the linking number. It remains to check different cases of equality. We can assume that $(X,Y,Z)$ are pairwise distinct: otherwise the equality follows from the two previous ones and Assertion \eqref{XX}.
\begin{itemize}
\item if $x=y$, the equation is true by Assertion \eqref{XX}.
\item Assume that up to cyclic permutations of $(X,Y,Z)$ we have $x=X$ and $y\not\in \{X,Y,Z\}$, then the equality follows from the following remark. Let $z,t$ be points close enough to $x$ so that $(z,x,t)$ is oriented then when $A=Y$ or $A=Z$, we have
$$
[xy,xA]=\frac{1}{2}\left([zy,xA]+[ty,xA]\right).
$$
\item Assume finally that $(x,y)=(X,Y)$, the the equality reduces to
$$
[XY,ZX]+[XY,YZ]=0,
$$
which is true, by Equation \eqref{1/2} and the fact that $(X,Y,Z)$ has the opposite orientation of $(Y,X,Z)$.
\end{itemize}
Equation \eqref{prodIn} follows from the geometric definition of linking number.
\end{proof}
A linking number satisfies more complicated relations. Namely
\begin{proposition}\label{prelimF}
Let $(X,x,Z,z,Y,y)$ be 6 points on the set ${\PP}$ equipped with a linking number $[ \cdotp, \cdotp]$, then
\begin{align}
[Xy,Zz]+[Yx,Zz]=[Xx,Zz]+[Yy,Zz].\label{In5}
\end{align}
Moreover, if $$
\{X,x\}\cap\{Y,y\}\cap\{Z,z\}=\emptyset,
$$
then
\begin{align}
[Xx,Yy][Xy,Zz]+[Zz,Xx][Zx,Yy]+[Yy,Zz][Yz,Xx]&=0\ ,\label{In3}\\ 
\lbrack Xx,Yy][Yx,Zz]+[Zz,Xx][Xz,Yy]+[Yy,Zz][Zy,Xx]&=0\label{In4}\ .
\end{align}
\end{proposition}

\rmks\begin{enumerate}
\item We remark that the hypothesis on the configuration of points is necessary: if $X,x,Y,Z$ are pairwise distinct, then for $(X,x,Y,y,Z,z)=(X,x,Y,x,Z,x)$ the left hand side in Equation \eqref{In3} is non-zero in the case of the standard linking number of pairs of points on the circle.
\item A simple way to prove this proposition is to use a mathematical computing software, we give below a mathematical proof. 

\end{enumerate}

\begin{proof} 
Formula \eqref{In5} follows at once from the Cocycle Identity \eqref{In2b}.
We now prove Formulas \eqref{In3} and \eqref{In4}. Let us define 
\begin{align*}
F(X,x,Y,y,Z,z)&\defeq [Xx,Yy][Xy,Zz]+[Zz,Xx][Zx,Yy]+[Yy,Zz][Yz,Xx]\ ,&\\ 
G(X,x,Y,y,Z,z)&\defeq \lbrack Xx,Yy][Yx,Zz]+[Zz,Xx][Xz,Yy]+[Yy,Zz][Zy,Xx]\ .&
\end{align*}
We first prove some symmetries of $F$ and $G$.

Our first observation is that, using the Antisymmetry \eqref{In1}, we get that
\begin{align}
F(X,x,Y,y,Z,z)=-G(Y,y,X,x,Z,z).\label{F=G}
\end{align}
Thus we only need to prove that $F=0$.
\vskip 0.2 truecm
\noindent{{\sc STEP 1: }{\em The expression $F$ is invariant under all permutations of the pairs $(X,x)$, $(Y,y)$ and $(Z,z)$}}
\vskip 0.2 truecm

Using Equations \eqref{In5} and \eqref{In1}, we obtain that
\begin{align*}
F(X,x,Y,y,Z,z)\ +&\ G(X,x,Y,y,Z,z)&\cr=& \ [Xx,Yy][Yy,Zz]+[Xx,Yy][Xx,Zz]&\cr
&+[Zz,Xx][Zz,Yy]+[Zz,Xx][Xx,Yy]&\cr
&+[Yy,Zz][Yy,Xx]+[Yy,Zz][Zz,Xx]&\cr
&=0\ .&
\end{align*}
Hence, by Equation \eqref{F=G}.
\begin{align}
F(X,x,Y,y,Z,z)=F(Y,y,X,x,Z,z)
\end{align}
By construction $F$ is invariant by cyclic permutations and thus from the previous equation $F$ is invariant by all permutations of the pairs $(X,x)$,$(Y,y)$ and $(Z,z)$.

\vskip 0.2 truecm
\noindent
{\sc STEP 2: }{\em The expression $F$ satisfies a cocycle equation
\begin{align}
F(X,x,Y,y,Z,z) +F(x,t,Y,y,Z,z)= F(X,t,Y,y,Z,z).\label{cocF}
\end{align}
We also have the symmetries
\begin{align}
F(X,x,Y,y,Z,z)&=-F(x,X,Y,y,Z,z)\cr
&=-F(X,x,y,Y,Z,z)\cr&=-F(X,x,Y,y,z,Z).\label{Antisym}
\end{align}}
\vskip 0.2 truecm
The Symmetries \eqref{Antisym} follow at once from the Cocycle \eqref{cocF} and the fact that $F(X,X,Y,y,Z,z)=0$.

Let us prove a cocycle equation for $F$. We shall only use the cocycle Equation \eqref{In2b} and the previous symmetries for the linking number. By definition,
\begin{align*}
F(X,x,Y,y,Z,z)\ +&\ F(x,t,Y,y,Z,z)&\cr
=&\ \ [Xx,Yy][Xy,Zz]+[xt,Yy][xy,Zz]&\cr
&+ [Zz,Xx][Zx,Yy]+[Zz,xt][Zt,Yy]&\cr
&+ [Yy,Zz][Yz,Xx]+[Yy,Zz][Yz,xt].&
\end{align*}
Using the cocycle Equation \eqref{In2b} to expand the first and regrouping the fifth and sixth terms of the right hand side, we get 
\begin{align*}
F(X,x,Y,y,Z,z)\ +&\ F(x,t,Y,y,Z,z)&\cr
=&\ \ [Xt,Yy][Xy,Zz]+[tx,Yy][Xy,Zz]+[xt,Yy][xy,Zz] &\cr
&+[Zz,Xx][Zx,Yy]+[Zz,xt][Zt,Yy]&\cr
&+[Yy,Zz][Yz,Xt].&
\end{align*}
Using the cocycle Equation \eqref{In2b} for regrouping the second and third term of the right hand side and rearranging, we get 
\begin{align*}
F(X,x,Y,y,Z,z)\ +&\ F(x,t,Y,y,Z,z)&\cr
=&\ \ [Xt,Yy][Xy,Zz] &\cr
&+[Zz,Xx][Zx,Yy]+[Zz,Xx][xt,Yy]+[Zz,xt][Zt,Yy]&\cr
&+[Yy,Zz][Yz,Xt]&\cr
=&\ F(X,t,Y,y,Z,z).&
\end{align*}
Using the cocycle Equation \eqref{In2b} to regroup the second and third term, then the fourth, of the right hand side, we finally get
\begin{align*}
F(X,x,Y,y,Z,z)\ +&\ F(x,t,Y,y,Z,z)&\cr
=&\ \ [Xt,Yy][Xy,Zz] &\cr
&+[Zz,Xt][Zt,Yy]&\cr
&+[Yy,Zz][Yz,Xt]&\cr
=&\ F(X,t,Y,y,Z,z).&
\end{align*}

\vskip 0.2 truecm
\noindent{\sc STEP 3: }{\em If $(X,x,Y,y)$ are pairwise distinct, then
\begin{align}
F(X,x,Y,y,Y,x)&=0\, ,&\label{F0-1}\\
F(X,x,X,x,Y,y)&=0\, .&\label{F0-2}
\end{align}}
\vskip 0.2 truecm
Let us prove first Equation \eqref{F0-1}. It follows from the Alternative \eqref{prodIn} and the cocycle Equation \eqref{In2b} that
\begin{align*}
F(X,x,Y,&y,Y,x)&\cr
&=[Xx,Yy][Xy,Yx]+([Yx,Xx][Yx,Yy]+[Yy,Yx][Yx,Xx])&\cr
&=0\ .&
\end{align*}
This proves Formula \eqref{F0-1}. Similarly, using the cocycle Formula \eqref{cocF} for $F$ for the first equality, symmetries for the second and our previous Formula \eqref{F0-1} (for $(X,x,y,Y)$ ) for the last, we get 
\begin{align*}
F(X,x,X,&x,Y,y)&\cr
&=F(X,x,X,y,Y,y)+F(X,x,y,x,Y,y)&\cr
&=F(X,x,y,x,y,Y)=F(X,x,y,Y,y,x)\cr
&=0.
\end{align*}
\vskip 0.2 truecm
\noindent{\sc STEP 4: }{\em If $(X,x,Y,y,Z)$ are pairwise distinct, then
\begin{align}
F(X,x,Y,&y,Z,x)=0\, .&\label{F0-3}
\end{align}}
\vskip 0.2 truecm
Using the cocycle  Formula \eqref{cocF} for $F$ and the previous step, we get
\begin{align*}
F(X,x,Y,&y,Z,x)&\cr
&=F(X,x,Y,Z,Z,x)+F(X,x,Z,y,Z,x)&\cr
&=-F(X,x,Z,Y,Z,x)=0\ .&
\end{align*}
\noindent{\sc FINAL STEP: }{\em If $(X,x,Y,y,Z,z)$ are pairwise distinct, then
\begin{align}
F(X,x,Y,&y,Z,z)=0\, .&\label{F0-4}
\end{align}}
Indeed, using the cocycle Formula \eqref{cocF} for $F$ for the first equality, symmetries for the second, and the previous step for the last equality, we get
\begin{align*}
F(X,x,Y,&y,Z,z)&\cr
&=F(X,x,Y,y,Z,Y)+F(X,x,Y,y,Y,z)&\cr
&=F(y,Y,x,X,Z,Y)-F(y,Y,x,X,y,z,Y)&\cr
&=0\ .&
\end{align*} 

This concludes the proof.
\end{proof}
\subsection{The swapping algebra}
Let ${\PP}$ be a set and $[ \cdotp, \cdotp]$ be a linking number with values in an integral domain $A$. We represent a pair $(X,x)$ of points of ${\PP}$ by the expression $Xx$.
We consider the free associative commutative algebra $\cL({\PP})$ generated over $A$ by pair of points on ${\PP}$, together with the relation $XX=0$ for all $X\in\mathsf P$.

Let $\alpha$ be any element in $A$. We define the {\em swapping bracket} of two pairs of points as the following element of $\cL({\PP})$
\begin{align}
\{Xx,Yy\}_{\alpha}\defeq [Xx,Yy](\alpha. Xx.Yy +Xy.Yx).
\end{align}

We extend the swapping bracket to the whole algebra $\cL({\PP})$ using the Leibniz Rule and call the resulting algebra $\cL_\alpha({\PP})$ the {\em swapping algebra}.

\begin{theorem}
The swapping bracket satisfies the Jacobi identity. Hence, 
the swapping algebra $\cL_\alpha({\PP})$ is a Poisson algebra. 
\end{theorem}
\begin{proof}
All we need to check is the Jacobi identity 
$$
\{\{Xx,Yy\}_\alpha,Zz\}_\alpha+\{\{Yy,Zz\}_\alpha,Xx\}_\alpha+\{\{Zz,Xx\}_\alpha,Yy\}_\alpha=0,
$$
for the generators of the algebra. 

We make preliminary computations, omitting the subscript $\alpha$ in the bracket. The triple bracket $\{\{A,B\},C\}$ is a polynomial of degree 2 in $\alpha$ and we wish to compute its coefficients.
By definition, using the Leibniz rule for the second equality, we have
\begin{align}
\{\{Xx,Yy\},Zz\}
=& [Xx,Yy]\left(\alpha\{Xx.Yy,Zz\}+\{Xy.Yx,Zz\}\right)&\cr
=&\ \alpha[Xx,Yy]\left(\{Xx,Zz\}.Yy+\{Yy,Zz\}.Xx\right)&\cr
&+[Xx,Yy]\left(\{Xy,Zz\}.Yx+\{Yx,Zz\}.Xy\right).&\label{Jac1}
\end{align}
Now we compute two expressions appearing in the right hand side of the previous equation. We have
\begin{align}
\{Xx,Zz\}.Yy\ +\ &\{Yy,Zz\}.Xx&\cr
=\ &\alpha\left([Xx,Zz]+[Yy,Zz]\right)Xx.Yy.Zz&\cr
&+\left([Xx,Zz]Xz.Yy.Zx+[Yy,Zz] Xx.Yz.Zy\right).&\label{Jac2}
\end{align}
Similarly
\begin{align}
\{Xy,Zz\}.Yx+\ &\{Yx,Zz\}.Xy&\cr
=&\alpha\left([Xy,Zz]+[Yx,Zz]\right)Xy.Yx.Zz&\cr
&+[Xy,Zz] Xz.Yx.Zy +[Yx,Zz]Xy.Yz.Zx\, .\label{Jac3}&
\end{align}
It follows from Equations \eqref{Jac2} and \eqref{Jac3} that the coefficient of $\alpha^2$ in the triple bracket \eqref{Jac1} is
\begin{equation}
P_2\defeq \left([Xx,Yy][Xx,Zz] +[Xx,Yy][Yy,Zz]\right)Xx.Yy.Zz\label{Jac-a2}.\end{equation}
The coefficient of $\alpha$ in the triple bracket \eqref{Jac1} is
\begin{align}
P_1\defeq \ \ 
=\ \ &[Xx,Yy][Xx,Zz]Xz.Yy.Zx+[Xx,Yy][Yy,Zz] Xx.Yz.Zy&\cr
&+([Xx,Yy][Xy,Zz]+[Xx,Yy][Yx,Zz])Xy.Yx.Zz
\label{Jac-a1}.&
\end{align}
Finally the constant coefficient is
\begin{align}
P_0&\defeq [Xx,Yy][Xy,Zz] Xz.Yx.Zy +[Xx,Yy][Yx,Zz]Xy.Yz.Zx\label{Jac-a0},&
\end{align}
so that
\begin{align}
\{\{Xx,Yy\},Zz\}
=&\alpha^2 P_2+\alpha. P_1+P_0.&
\end{align}

In order to check the Jacobi identity, we have to consider the sum $S_2$, $S_1$ and $S_0$ over cyclic permutations of $(Xx,Yy,Zz)$ of the three terms $P_2$, $P_1$ and $P_0$. We consider successively these three coefficients.
\vskip 0,2 truecm

\noindent{\sc Term of degree 0}: We first have
\begin{align}
S_0&=F(X,x,Y,y,Z,z)(Xz.Yx.Zy -Xy.Yz.Zx)\ ,&\label{S0}
\end{align}
Indeed, we have
\begin{align*}
S_0&=A.Xz.Yx.Zy + B.Xy.Yz.Zx\ ,&
\end{align*}
 where
\begin{align*}
A&=[Xx,Yy][Xy,Zz]+[Zz,Xx][Zx,Yy]+[Yy,Zz][Yz,Xx]=F(X,x,Y,y,Z,z)\ ,& \cr
 B&=[Xx,Yy][Yx,Zz]+[Zz,Xx][Xz,Yy]+[Yy,Zz][Zy,Xx]=G(X,x,Y,y,Z,z)\ .&
\end{align*}
Now Equation \eqref{S0} follows from Equation \eqref{F=G}.
\vskip 0.1 truecm
We now prove that $S_0=0$.
It follows from Proposition \ref{prelimF} that if $$
\{X,x\}\cap\{Y,y\}\cap\{Z,z\}=\emptyset,
$$
then $F=0$, hence $S_0=0$. 

Up to cyclic permutations, we just have to consider two cases
\begin{enumerate}
\item 
If $x=y=z$ or $X=Y=Z$ then $$Xz.Yx.Zy -Xy.Yz.Zx=0,$$ hence $S_0=0$.
\item If $x=y=Z$ or $X=Y=z$ or the other cases obtained by cyclic permutations, since $aa=0$, we have
$$
Xz.Yx.Zy=Xy.Yz.Zx=0
$$
Thus $S_0=0$.
\end{enumerate}
We have completed the proof that $S_0=0$.
\vskip 0,2 truecm
\noindent{\sc Term of degree 1.} Next, we write
\begin{align*}
P_1=&\ A_1(X,x,Y,y,Z,z)Xx.Yz.Zy&\cr
 &+A_2(X,x,Y,y,Z,z) Xz.Yy.Zx&\cr
 &+A_3(X,x,Y,y,Z,z) Xy.Yx.Zz\ .&
\end{align*}
Thus
\begin{align*}
S_1&=A_x. Xx.Yz.Zx + A_y. Xz.Yy.Zx. +A_z. Xy.Yx.Zz\ ,&
\end{align*}
where
\begin{align*}
A_z=&A_3(X,x,Y,y,Z,z)+A_2(Y,y,Z,z,X,x)+A_1(Z,z,X,x,Y,y)&\cr
=&\ [Xx,Yy][Xy,Zz]+[Xx,Yy][Yx,Zz]&\cr &+[Yy,Zz][Yy,Xx]+[Zz,Xx][Xx,Yy]&\cr
=&\ [Xx,Yy]\left([Xy,Zz]+[Yx,Zz]-[Yy,Zz]-[Xx,Zz]\right)\ .&
\end{align*}
By Equation \eqref{In5}, $A_z=0$. Therefore, $A_y=A_z=A_x=0$ by cyclic permutations. We have completed the proof that $S_1=0$.
\vskip 0,2 truecm
\noindent{\sc Term of degree 2.} Finally, $S_2=C.Xx.Yy.Zz$, where
\begin{align*}
C=&\ [Xx,Yy][Xx,Zz] +[Xx,Yy][Yy,Zz]&\cr
&+[Yy,Zz][Yy,Xx] +[Yy,Zz][Zz,Xx]&\cr
&+[Zz,Xx][Zz,Yy] +[Zz,Xx][Xx,Yy]\ .&
\end{align*}
Then $C=0$ by the antisymmetry of the linking number. Thus $S_2=0$.
\vskip 0.2 truecm

This concludes the proof of the Jacobi identity. Indeed
\begin{align*}
\{\{Xx,Yy\}_\alpha,Zz\}_\alpha+\{\{Yy,Zz\}_\alpha,Xx\}_\alpha+&\{\{Zz,Xx\}_\alpha,Yy\}_\alpha
&\cr&=\alpha^2 S_2+\alpha S_1+S_0&\cr
&=0.&
\end{align*}

\end{proof}
\subsection{The multifraction algebra}
The swapping algebra is very easy to define. However, in the sequel we shall need to consider other Poisson algebras built out of the swapping algebra: these algebras  will be more precisely subalgebras of the fraction algebra of $\cL(\PP)$. We introduce in this paragraph {\em cross fractions}, {\em multifractions} and the {\em multifraction algebra}.
\subsubsection{Cross fractions and multifractions}
Since $\cL_\alpha(\PP)$ is an integral domain (with respect to the commutative product) we can consider its algebra of fractions $\cQ_\alpha(\PP)$. 

Let $(X,Y,x,y)=:Q$ be a quadruple of points of $\PP$ so that $x\not=Y$ and $y\not=X$. The {\em cross fraction} determined by $Q$ is the element of $\cQ_\alpha(\PP)$ defined by
$$
[X;Y;x;y]\defeq \frac{Xx.Yy}{Xy.Yx}.
$$
More generally, let ${\rm X}\defeq (X_1,\ldots , X_n)$ and ${\rm x}\defeq (x_1,\ldots x_n)$ are two tuples of elements of $\PP$ so that $x_i\not=X_i$ for all $i$, let $\sigma$ be a permutation of 
$\{1,\ldots,n\}$ then the {\em elementary multifraction -- defined over $\PP$ --} defined by this data is
$$
[{\rm X},{\rm x};\sigma]\defeq \frac{\prod_{i=1}^{i=n}X_ix_{\sigma(i)}}{ \prod_{i=1}^{i=n}X_ix_{i}}.
$$
\subsubsection{The multifraction algebra}
Let now $\mathcal B(\PP)$ be the vector space generated by elementary multifractions and let us call any element of $\mathcal B(\PP)$ a {\em multifraction}. We have the following proposition
\begin{proposition}
The vector space $\mathcal B(\PP)$ is a Poisson subalgebra of $\cQ_\alpha(\PP)$. Moreover it is generated as a Poisson algebra by cross fractions. Finally the swapping bracket $\{ \cdotp, \cdotp\}_\alpha$ when restricted to $\mathcal B(\PP)$ does not depend on $\alpha$.
\end{proposition}
From now on, we call the Poisson algebra $\cB(\PP)$ the {\em algebra of multifractions}.
\begin{proof} The proposition follows from the following immediate observations:
\begin{itemize}
\item every elementary multifraction is a product of cross fractions,
\item if $A$ and $B$ are two cross fractions then $\{A,B\}_\alpha$ is a multifraction and does not depend on $\alpha$.
\end{itemize}
\end{proof}

\section{Cross ratios and cross fractions}\label{sec:crossratio}
In this section, we interpret cross fractions, and in general multifractions, as functions on the space of cross ratios. 
\subsection{Cross ratios}\label{def:cr}
Recall that a cross ratio on a set ${\PP}$ is a map $\bb$ from
$$
{\PP}^{4*}\defeq \{(X,Y,x,y)\in {\PP}\mid y\not=X, x\not=Y\}
$$
to a field $\mathbb K$ 
which satisfies some algebraic rules. 
These rules encode two conditions which constitute a normalisation, 
and two multiplicative cocycle identities which hold for different sets of variables:
\begin{eqnarray*}
{\hbox{\sc Normalisation:}}\ \ \bb(X,Y,x,y)&=&0\ \ \Leftrightarrow x=X \,,\hbox{ or } Y=y,\\ 
{\hbox{\sc Normalisation:}}\ \ \bb(X,Y,x,y)&=&1 \ \ \Leftrightarrow x=y \,,\hbox{ or } X=Y,\\
{\hbox{\sc Cocycle identity:}}\ \ \bb(X,Y,x,y)&=&{\bb(X,Y,x,z)}{\bb(X,Y,z,y)},\\
{\hbox{\sc Cocycle identity:}}\ \ \bb(X,Y,x,y)&=&\bb(X,Z,x,y)\bb(Z,Y,x,y).
\end{eqnarray*}
Assume $\Gamma$ acts on ${\PP}$, way say the cross ratio $\bb$ is $\Gamma$-invariant, if it is invariant under the diagonal action.

\rmks
\begin{itemize}
\item We change our convention from our previous articles \cite{Labourie:2006,Labourie:2005} in order to be coherent with the formula for the classical projective cross ratio: let $\bb$ be a cross ratio with respect to the definition above, let $b(X,x,Y,y)\defeq \bb(X,Y,x,y)$, then $b$ is a cross ratio using our older convention. Observe that the second normalisation together with the cocycle identities imply the following symmetries: $$\bb(X,Y,x,y)=\bb(Y,X,x,y)^{-1}=\bb(Y,X,y,x)=\bb(X,Y,y,x)^{-1}.$$
\item Assume $\Gamma$ acts on ${\PP}$. Let $\bb$ be a $\Gamma$-invariant cross ratio. Let $\gamma\in\Gamma$ and $\gamma^+$ and $\gamma^-$ be two $\gamma$-fixed points in ${\PP}$, then the following quantity
\begin{eqnarray*}
 \bb(\gamma^+,\gamma^-,\gamma y,y)
 \end{eqnarray*}
 does not depends on the choice of $y$. In particular, let $S$ be a closed connected oriented surface of genus greater than 2, let ${\PP}$ be $\bgrf$ equipped with the action of $\grf$. Let $\gamma^+$ and $\gamma^-$ be respectively the attractive and repulsive fixed point of a non-trivial element $\gamma$ of $\pi_1(S)$ and $\bb$ a $\grf$-invariant cross ratio, then 
 $$
 \ell_\bb(\gamma)\defeq \left\vert\log\left\vert \bb(\gamma^+,\gamma^-,\gamma (y),y)\right\vert\right\vert,
 $$
 is called the {\em period} of $\gamma$.\end{itemize}

 We finally denote by $\BB(\PP)$ the set of cross ratios on $\PP$.
 
 These definitions are closely related to those given by Otal in \cite{Otal:1990th,Otal:1992},
 discussions from various perspectives by Ledrappier in \cite{Ledrappier:1995} 
and work of Bourdon in \cite{Bourdon:1996} in the context of ${\rm CAT}(-1)$-spaces.

\subsection{Multifractions as functions} 
To every cross fraction $[X;Y;x;y]$, we associate a function, denoted by $\overline{[X;Y;x;y]}$, on $\BB(\PP)$ by the following formula
$$
\overline{[X;Y;x;y]}(\bb)\defeq\bb(X,Y,x,y).
$$
The following proposition follows at once from the definition of cross ratio
\begin{proposition}
The map ${[X;Y;x;y]}\to\overline{[X;Y;x;y]}$ extends uniquely to a morphism of commutative associative algebras from $\cB(\PP)$ to the algebra of functions on $\BB(\PP)$.
\end{proposition}
 
In the sequel, we shall use an identical notation for a multifraction and its image in 
 the space of functions on $\BB(\PP)$. So far also, we did not (and will not) consider any topological structure on $\BB(\PP)$ nor on $\mathsf P$.

\subsection{Multifractions and Hitchin components}
In \cite{Labourie:2006}, we identified the Hitchin component with a space of cross ratios satisfying certain identities. Let us recall notation and definitions

\subsubsection{Hitchin component} Let $S$ be a closed oriented connected surface with genus at least two.

\begin{definition}{\sc[Fuchsian and Hitchin homomorphisms]} An {\em $n$-Fuchsian} homomorphism from $\grf$ to $\sln$ is a
homomorphism $\rho$ which factorises as
$\rho={\ii}\circ\rho_{0}$, where $\rho_{0}$ is a discrete faithful 
homomorphism with values in $\sld$ and ${\ii}$
is an irreducible homomorphism from $\sld$ to 
$\sln$. 

An {\em $n$-Hitchin homomorphism} from $\grf$ to $\sln$ is a homomorphism which may be deformed into an $n$-Fuchsian homomorphism
\end{definition}
 
The {\em Hitchin component} ${\Hn}$ is the space of Hitchin homomorphisms up to conjugacy by an exterior automorphism of $\sln$. All these representations lift to $\mathsf{SL}(n,\mathbb R)$. By construction ${\Hn}$ is identified with a connected component of the character variety. It is a result by Hitchin \cite{Hitchin:1992es} that ${\Hn}$ is homeomorphic to the interior of a ball of dimension $(2g-2)(n^2-1)$. 

As a corollary of the main result of \cite{Labourie:2006}, we have
\begin{theorem}
If $\rho$ is Hitchin, if $\gamma$ is a non-trivial element of $\grf$ then $\rho(\gamma)$ has $n$ distinct positive real eigenvalues.
\end{theorem}

By convention, we write these eigenvalues as $\lambda_1(\rho(\gamma)),\ldots,\lambda_n(\rho(\gamma))$ where
$$
\lambda_1(\rho(\gamma))>\ldots>\lambda_n(\rho(\gamma))>0.
$$
This allows us to introduce the following
\begin{definition}\label{def:girth}{\sc[Girth and width]}
The {\em width} of a non-trivial element $\gamma$ of $\grf$ with respect to a Hitchin representation $\rho$ is
$$
\operatorname{width}_\rho(\gamma)\defeq \log\left(\left\vert\frac{\lambda_1(\rho(\gamma))}{\lambda_n(\rho(\gamma))}\right\vert\right).
$$
The {\em girth} of $\rho$ is
\begin{eqnarray}
\gir(\rho)\defeq \sup\left\{\left\vert\frac{\lambda_2(\rho(\gamma))}{\lambda_1(\rho(\gamma))}\right\vert \mid \gamma\in\grf\setminus\{{\rm Id}\}\right\}.
\end{eqnarray}
\end{definition}

The following proposition will be used several times
\begin{proposition}\label{girth}
Let $C$ be a compact subset of $\Hn$ then
\begin{enumerate}
	\item For any positive $A$, the following subset of $\pi_1(S)$ defined by 
	$$S_A=\left\{\gamma\in\pi_1(S)\mid \exists \rho\in C, \, \, \left\vert \frac{\lambda_2(\rho(\gamma))}{\lambda_1(\rho(\gamma))}\right\vert> A\right\}$$
	contains only finitely many conjugacy classes.
	\item Moreover
$$
\sup\{\gir(\rho)\mid\rho\in C\}<1.
$$
\end{enumerate}
\end{proposition}

For the proof of this proposition,
we first need \begin{lemma}\label{L12Anos}
Let $S$ be a hyperbolic surface with unit tangent bundle $\mathsf U S$ equipped with the geodesic flow $\{\phi_t\}_{t\in\mathbb R}$. Let $\rho_0$ be a Hitchin representation in $\Hn$. Then there exists a neighbourhood $W$ of $\rho_0$ in $\Hn$, such that for every $\rho$ in $W$, there exists a function $f_\rho: \mathsf U S\to\mathbb R$ verifying
\begin{itemize}
	\item for every closed orbit $\gamma$ and $x\in\gamma$,
	 $$\int_0^{\ell(\gamma)} f_\rho\circ\phi_s(x)\,{\rm d}s\,=\log\left\vert\frac{\lambda_1(\rho(\gamma))}{\lambda_2(\rho(\gamma))}\right\vert,$$
	 where $\ell(\gamma)$ is the hyperbolic length of $\gamma$;
	\item the function $\rho\mapsto f_\rho$ is continuous from $W$ to $C^0(\mathsf U S,\mathbb R)$ and moreover there exists a positive constant $\epsilon_0$ so that for all $\rho$, $f_\rho>\epsilon_0$ .
\end{itemize}.
	\end{lemma}
\vskip 0.2 truecm	
\noindent{\em Proof of Lemma \ref{L12Anos}:}
This follows from the Anosov property of Hitchin representations and results in \cite{Bridgeman:2015ba}. One could also use results by Guichard--Wienhard \cite{Guichard:2012eg} or combine results of Sambarino \cite{Sambarino:2014jv,Sambarino:2014kc}. Since by Theorem 6.1 of \cite{Bridgeman:2015ba}, the limit maps of a Hitchin representation depend in an analytic way of the representation, we can find
\begin{itemize}
\item a neighbourhood $D$ of $\rho_0$ in $\Hn$,
\item a vector bundle $E$ over $M\defeq D\times \mathsf U S$ smooth along $\mathsf U S$,
\item a splitting of $
E=L_1\oplus\ldots L_n$ into line bundles smooth along the geodesic flow,
\item a continuous lift $\{\Phi_t\}_{t\in\mathbb R}$ on $E$ of the geodesic flow $\{\phi_t\}_{t\in\mathbb R}$ on $M$ preserving this decomposition and smooth along $\mathsf U S$,
\end{itemize}
so that if $\gamma$ is a closed geodesic of hyperbolic length $\ell(\gamma)$ and $u\in L_i\vert_{\{\rho\}\times\gamma}$, then
$$
\Phi_{\ell(\gamma)}(u)= \lambda_i(\rho(\gamma)) \cdotp u.
$$ 
In this last equation, we identify the closed geodesic with the corresponding conjugacy class in $\grf$.
We now construct metrics $\omega_i$ on $L_i$ smooth along the geodesic flow. Let us consider the functions $g_i$ on $M$ so that $g_i \cdotp\omega_i=\left.\frac{\rm d}{{\rm d} t}\right\vert_{t=0}\Phi_t^*\omega_i$. In particular, we have \begin{eqnarray}
\Phi_t^*\omega_i(x)=\exp\left(\int_0^t g_i\circ\phi_s(x)\,{\rm d}s\right) \omega_i(x).
\end{eqnarray}
Then by construction for $x\in \{\rho\}\times\gamma$, we have 
$$
-\log\left(\lambda_i(\rho(\gamma)\right)=\int_0^{\ell(\gamma)}g_i\circ\phi_s(x)\,{\rm d}s.
$$
Let now $g=g_2-g_1$, then
$$
\log\left(\frac{\lambda_1(\rho(\gamma))}{\lambda_2(\rho(\gamma))}\right)=\int_0^{\ell(\gamma)}g\circ\phi_s(x)\,{\rm d}s.
$$
By the Anosov property, there exists some $T>0$, so that the flow $\Phi_T$ contracts uniformly on $\operatorname{Hom}(L_1,L_2)$ along $\{\rho_0\}\times \mathsf U S$. In a more precise way, if we denote by $\omega_1^*$ the dual metric on $L_1^*$ to $\omega_1$, then there exists some $T$ so that along $\{\rho_0\}\times \mathsf U S$, we have
$$
\Phi_T^*(\omega_1^*\otimes \omega_2)= H \cdotp \omega_1^*\otimes \omega_2,
$$ 
where $H$ is a continuous function on $M$ so that along $\{\rho_0\}\times \mathsf U S$, 
$$H<\frac{1}{2}.$$ 
By the continuity of $H$, the previous inequality extends to $M$ after possibly restricting $D$. As a consequence, we have for $x\in M$,
$$
\int_0^T g\circ\phi_s(x)\,{\rm d}s =-\log(H(x))>\log(2).
$$
 Let now 
$$
f(x)\defeq\frac{1}{T}\int_0^T g\circ\phi_s(x)\,{\rm d}s.
$$
Then by construction $$f(x)> \frac{1}{T}\log(2)\eqdef\epsilon_0,$$ and, moreover, for
 $x\in\{\rho\}\times \mathsf U S$,
$$
\int_0^{\ell(\gamma)}f\circ\phi_s(x)\,{\rm d}s=\int_0^{\ell(\gamma)}g\circ\phi_s(x)\,{\rm d}s=\log\left(\frac{\lambda_1(\rho(\gamma))}{\lambda_2(\rho(\gamma))}\right).$$ \qed

\vskip 0.2 truecm	
\noindent{\em Proof of Proposition \ref{girth}:} By compactness, it is enough to prove that every $\rho$ in $\Hn$ possesses a neighbourhood $W$ so that the properties of the proposition hold when $C$ is replaced by $W$. We choose the neighbourhood $W$ obtained in the previous lemma. Let then $f_\rho$ be as in the conclusion of this lemma. Since $f_\rho$ is bounded away from zero by a positive constant $\epsilon_0$, it follows that
$$
\log\left(\frac{\lambda_1(\rho(\gamma))}{\lambda_2(\rho(\gamma))}\right)>\epsilon_0\cdot\,\ell(\gamma).
$$
The first result immediately follows. Then for the second result we use the fact that $S_{\frac{1}{2}}$ contains only finitely many conjugacy classes and that given $\gamma$ the function 
$$
\rho\mapsto \frac{\lambda_2(\rho(\gamma))}{\lambda_1(\rho(\gamma))},
$$
is continuous and with values less than 1.
\qed

\subsubsection{Rank $n$ cross ratios}
For every integer $p$, let $\bgrf^p_*$ be the set of pairs $$
({\rm X},{\rm x})=\left((X_0,X_1,\ldots,X_p),(x_0,x_1,\ldots,x_p)\right),$$ 
of $(p+1)$-tuples of points in $\bgrf$ such that 
$
X_j\not= X_i\not= x_0$ and $x_j\not=x_i\not=X_0$,
whenever $j>i>0$. Let $\chi^p({\rm X},{\rm x})$ be the multifraction defined by
$$\chi^p({\rm X},{\rm x})\defeq \det_{i,j>0}\left([X_i;X_0;x_j;x_0]\right).$$
A cross ratio $\bb$ has {\em rank $n$} if 
\begin{itemize}
\item $\chi^n({\rm X},{\rm x})(\bb)\not=0$, for all $({\rm X},{\rm x})$ in $\bgrf^n_* $,
\item $\chi^{n+1}({\rm X},{\rm x})(\bb)=0$, for all $({\rm X},{\rm x})$ in $\bgrf^{n+1}_* $.
\end{itemize}
The main result of \cite{Labourie:2008tv}
-- which used a result by Guichard \cite{Guichard:2008tu} -- is the following. 
\begin{theorem}\label{hcintro}
There exists a bijection $\phi$ from the set of $n$-Hitchin representations to the set of $\grf$-invariant rank $n$ cross ratios, such that if $\bb=\phi(\rho)$ then
\begin{enumerate}
\item for any non-trivial element $\gamma$ of $\grf$
$$
\ell_{\bb}(\gamma)=\operatorname{width}_\rho(\gamma),
$$
where $\ell_{\bb}(\gamma)$ is the period of $\gamma$ given with respect to $\bb=\phi(\rho)$, and $\operatorname{width}_\rho(\gamma)$ is the width of $\gamma$ with respect to $\rho$.
\item Moreover, if $\gamma_1$ and $\gamma_2$ are two non-trivial elements of $\grf$, if $e_i$, (respectively $E_i$) is an eigenvector of $\rho(\gamma_i)$ of maximal eigenvalue (respectively eigenvector of $\rho^*(\gamma_i)$ of minimum eigenvalue) then
\begin{eqnarray}
\bb(\gamma_1^+,\gamma_2^+,\gamma_2^-,\gamma_1^-)&=&\frac{\langle E_2,e_1\rangle\langle E_1,e_2\rangle}{\langle E_1,e_1\rangle\langle E_2,e_2\rangle}.\label{crhitchin}
\end{eqnarray}
\end{enumerate}
\end{theorem}

In particular, every multifraction defines a function on the Hitchin component.
\section{Wilson loops, multifractions and length functions}
In this section, we shall relate Wilson loops -- which are regular functions on the character variety -- to multifractions. We will also introduce {\em elementary functions} which are limits of Wilson loops, prove that they generate the multifraction algebra and that they are smooth functions on the Hitchin component.
We finally introduce {\em length functions} in Paragraph \ref{def:lf}.

\subsection{Wilson loops} Let $\gamma$ be an element of $\grf$ and $\rho$ be an element of $\Hn$. The {\em Wilson loop} associated to $\gamma$ is the function $\ww(\gamma)$ on $\Hn$ defined by
$$
\ww(\gamma)(\rho)\defeq\tr (\rho(\gamma)).
$$
Wilson loops only depends on conjugacy classes.
Let us introduce the following definition.

\begin{definition}{\sc [Class of an element]}
Let $\gamma$ be a non-trivial element of $\grf$, the {\em class} $[\gamma]$ of $\gamma$ is the oriented pair $(\gamma^+,\gamma^-)$ of points of $\bgrf$ where $\gamma^+$ and $\gamma^-$ are respectively the attractive and repulsive fixed points of $\gamma$.
 \end{definition} Recall that $[\gamma]=[\eta]$ if and only if there exist positive integers $m$ and $n$ so that $\gamma^m=\eta^n$.
\subsubsection{Asymptotics of Wilson loops}
Let $\rho$ be a Hitchin representation. Recall that for any $\gamma$ in $\grf$ we can write
$$
\rho(\gamma)=\sum_{1\leq i\leq n}\lambda_i(\rho(\gamma)) p_i(\gamma),
$$
where $p_i(\gamma)$ is a projector of trace 1, and $\lambda_i(\rho(\gamma))$ are real numbers so that
$$
0<\vert\lambda_n(\rho(\gamma))\vert<\ldots<\vert\lambda_1(\rho(\gamma))\vert.
$$
Let us denote $\bu\rho(\gamma)=p_1(\gamma)$. We denote by 
$\Eig{\mathsf A}$ the set of eigenvectors of a purely loxodromic matrix $\mathsf A$, and observe that $\Eig{\mathsf A^n}=\Eig{\mathsf A}$. We choose an auxiliary norm, denoted $\Vert \cdotp\Vert$ on $\mathbb R^n$. Then we have
\begin{proposition}\label{Asym1}
For any $\gamma$ in $\grf$, and $p\in\mathbb N$ we have
\begin{eqnarray}
\left\Vert \frac{\rho(\gamma^p)}{\ww(\gamma^p)(\rho)}- \bu\rho(\gamma)\right\Vert
&\leq& \gir(\rho)^pK\left(\Eig{\rho(\gamma)}\right),
\end{eqnarray}
where $K$ is a continuous function on the set of $n$ lines in general position. 
\end{proposition}
\begin{proof} Let $\mathsf A=\rho(\gamma)$. Since   ${\ms A}$ is a real diagonalisable matrix then $$
{\ms A}=\sum_{i=1}^{i=n}\lambda_i {\ms p}_i,
$$
where ${\ms p}_i$ are projectors and the eigenvalues $\lambda_i$ satisfy $\lambda_1> \cdots>\lambda_n>0$. Thus 
\begin{eqnarray}
\left\Vert \frac{{\ms A}^p}{\tr( {\ms A}^p)}-{\ms p}_1\right\Vert&\leq&\frac{1}{\sum_{i=1}^{i=n}\lambda_i^p}\left\Vert
\sum_{i=2}^{i=n}\lambda_i^p{\ms p}_i-\left(\sum_{i=2}^{i=n}\lambda_i^p\right){\ms p}_1\right\Vert\cr
&\leq&\left(\frac{\lambda_2}{\lambda_1}\right)^p\left(n\Vert {\ms p}_1\Vert +\sum_{i=2}^{i=n}\Vert {\ms p}_2\Vert\right).
\end{eqnarray}
Thus the inequality follows by taking 
$$
K(\Eig{{\ms A}})=n\Vert {\ms p}_1\Vert +\sum_{i=2}^{i=n}\Vert {\ms p}_2\Vert.
$$
\end{proof}

As a corollary, we get
\begin{corollary}\label{Asym1-coro}
	Let $\gamma_1$, $\gamma_2$,$\ldots$, $\gamma_q$ be $p$ coprime elements of $\Gamma$, let $m_1$, $m_2$,$\ldots$, $m_q$ be positive numbers then 
	\begin{eqnarray}
\left\Vert \frac{\prod_{i=1}^q\rho(\gamma_i^{m_i})}{\ww\left(\prod_{i=1}^q\gamma_i^{m_i}\right)(\rho)}- \frac{\bu\rho(\gamma_1)\bu\rho(\gamma_q)}{\tr\left(\bu\rho(\gamma_1)\bu\rho(\gamma_q)\right)}\right\Vert
&\leq& \gir(\rho)^m K,
\end{eqnarray}
where $m=\inf(m_i)$ and  $K$ depends continuously on the eigenvectors of $\rho(\gamma_i)$ and their relative configurations. 
	\end{corollary}
	\begin{proof}
		We restate the previous Proposition by saying that
		
		\begin{eqnarray}
\rho(\gamma^p)&=&\ww(\gamma^p)(\rho)\cdotp\left(\bu\rho(\gamma)
+ \gir(\rho)^p\cdotp K(\gamma,\rho)\right),
		\end{eqnarray}
		where $K(\gamma,\rho)$ is continuous in $\rho$ and  only depends on the eigenvectors of $\rho(\gamma)$.
		Thus
				\begin{eqnarray}
\prod_{i=1}^q\rho(\gamma_i^{m_i})&=&\prod_{i=1}^q\ww(\gamma_i^{m_i})(\rho)\cdotp\prod_{i=1}^q\left(\bu\rho(\gamma_i)
+ \gir(\rho)^{m_i}\cdotp K(\gamma_i,\rho)\right)\cr
&=&\prod_{i=1}^q\ww(\gamma_i^{m_i})(\rho)\cdotp\left(\prod_{i=1}^q\bu\rho(\gamma_i) 
+ \gir(\rho)^m\cdotp K_0(\gamma_1,\ldots,\gamma_p;\rho)\right),\label{eq:corAsym11}
		\end{eqnarray}
where $K_0(\gamma_1,\ldots,\gamma_p;\rho)$ is continuous in $\rho$ and  only depends on the eigenvectors of 
$\rho(\gamma_i)$.
Thus 
				\begin{eqnarray}
\frac{\ww\left(\prod_{i=1}^q\gamma_i^{m_i}\right)(\rho)}{\prod_{i=1}^q\ww(\gamma_i^{m_i})(\rho)}&=&\tr\left(\prod_{i=1}^q\bu\rho(\gamma_i)\right)
+ \gir(\rho)^m\cdotp K_1(\gamma_1,\ldots,\gamma_p;\rho),\label{eq:corAsym12}
		\end{eqnarray}
		where $K_1(\gamma_1,\ldots,\gamma_p;\rho)$ is continuous in $\rho$ and  only depends on the eigenvectors of 
$\rho(\gamma_i)$.
Combining equations \eqref{eq:corAsym11} and \eqref{eq:corAsym12}, we obtain that
				\begin{eqnarray}
\frac{\prod_{i=1}^q\rho(\gamma_i^{m_i})}{\ww\left(\prod_{i=1}^q\gamma_i^{m_i}\right)(\rho)}&=&\frac{\prod_{i=1}^q\bu\rho(\gamma_i)}{\tr\left(\prod_{i=1}^q\bu\rho(\gamma_i)\right)}
+ \gir(\rho)^m\cdotp K_2(\gamma_1,\ldots,\gamma_p;\rho),\label{eq:corAsym13}
		\end{eqnarray}
where $K_1(\gamma_1,\ldots,\gamma_p;\rho)$ is continuous in $\rho$ and  only depends on the eigenvectors of 
$\rho(\gamma_i)$ and their relative positions. To conclude the proof of the corollary, we remark that is $A$ is an endormorphism $\rm p$ and $\rm q$ are projectors so that
$\tr({\rm p} A {\rm q})\not=0\not=\tr({\rm p} A {\rm q})$, then 
$$
\frac{{\rm p} A {\rm q}}{\tr({\rm p} A {\rm q})}=\frac{{\rm p}  {\rm q}}{\tr({\rm p} {\rm q})}\ .
$$
Using this remark, we get that
\begin{eqnarray*}
\frac{\prod_{i=1}^q\bu\rho(\gamma_i)}{\tr\left(\prod_{i=1}^q\bu\rho(\gamma_i)\right)}
&=&\frac{\bu\rho(\gamma_1)\bu\rho(\gamma_q)}{\tr\left(\bu\rho(\gamma_1)\bu\rho(\gamma_q)\right)}
\, .
		\end{eqnarray*}
Combining this last equality with Equality \eqref{eq:corAsym13} yields the statement of the corollary.	
	\end{proof}

We begin with the following proposition where we consider multifractions as functions on $\Hn$.

\begin{proposition}\label{Asym2}
Let $\gamma_1,\ldots, \gamma_k$ be non-trivial elements of $\pi_1(S)$.
Then the sequence
$$
\left\{\frac{\ww(\gamma_1^p\ldots \gamma_k^p)}{\ww(\gamma_1^p)\ldots \ww(\gamma_k^p)}\right\}_{p\in \mathbb N}
$$
converges uniformly on every compact of $\Hn$ to a multifraction when $p$ goes to infinity. More precisely,
$$
\lim_{p\rightarrow\infty}\left(\frac{\ww(\gamma_1^p\ldots \gamma_k^p)}{\ww(\gamma_1^p)\ldots \ww(\gamma_k^p)}\right)=\frac{\prod_{i=1}^{i=k} \gamma_{i+1}^+\gamma_i^-}{\prod_{i=1}^{i=k} \gamma_{i}^+\gamma_i^-}
=[{\rm G}^+,{\rm G}^-;\tau],$$
where ${\rm G}^\pm=(\gamma_1^\pm,\ldots,\gamma_k^\pm)$ and $\tau(i)=i-1$, using the convention that $k+1=1$.
\end{proposition}
\begin{proof} We first observe that if $e_i$, (respectively $E_i$) is an eigenvector of $\rho(\gamma_i)$ of maximal eigenvalue (respectively eigenvector of $\rho^*(\gamma_i)$ of minimum eigenvalue) so that $\langle E_i,e_i \rangle=1$ then
$$
\tr (\bu\rho(\gamma_1)\ldots\bu\rho(\gamma_k))=\prod_{i}\langle E_i,e_{i+1}\rangle.
$$
By Equation \eqref{crhitchin}, 
$$
\prod_{i}\langle E_{i},e_{i+1}\rangle=\left(\frac{\prod_{i=1}^{i=p} \gamma_{i+1}^+\gamma_i^-}{\prod_{i=1}^{i=k} \gamma_{i}^+\gamma_i^-}\right)(\rho).
$$
It thus follows that
$$
\tr (\bu\rho(\gamma_1)\ldots\bu\rho(\gamma_k))=[{\rm G}^+,{\rm G}^-;\tau].
$$
Then the result follows at once from Propositions \ref{Asym1} and \ref{girth}.\end{proof}
\subsection{Elementary functions}\label{sec:elemfunct}

Proposition \ref{Asym2} leads us to the following definition. 

\begin{definition}  
The multifraction
\begin{eqnarray}
\tw(\gamma_1,\ldots,\gamma_p)\defeq \frac{\prod_{i=1}^{i=p} \gamma_{i+1}^+\gamma_i^-}{\prod_{i=1}^{i=p} \gamma_{i}^+\gamma_i^-}\label{defelem}
\end{eqnarray}
is an {\em elementary function of order $p$}
\end{definition}
By the previous proposition and its proof, we have the following equalities
\begin{eqnarray}
\tw(\gamma_1,\ldots,\gamma_p) &=& \lim_{n\rightarrow\infty}\frac{\ww(\gamma_1^n\ldots \gamma_p^n)}{\ww(\gamma_1^n)\ldots \ww(\gamma_p^n)}\label{elem-limi}\\
\tw(\gamma_1,\ldots,\gamma_p) &=&\tr(\bu\rho(\gamma_1)\ldots\bu\rho(\gamma_p)).\label{elem-proj}
\end{eqnarray}

The following formal properties of elementary functions are then easily checked.\begin{proposition}
The following properties of elementary functions hold
\begin{enumerate}
\item {\sc Cyclic invariance:} for every cyclic permutation $\sigma$ of $\{1,\ldots,p\}$ we have 
$$
\tw(\gamma_1,\ldots,\gamma_p)=\tw(\gamma_{\sigma(1)},\ldots,\gamma_{\sigma(p)})
$$ 
\item {\sc Class invariance:} if $[\eta_i]=[\gamma_i]$ then 
$$
\tw(\gamma_1,\ldots\gamma_p)=\tw(\eta_1,\ldots,\eta_p).
$$
\item if $[\gamma_p]=[\gamma_{p-1}]$ then
$$
\tw(\gamma_1,\ldots\gamma_p)=\tw(\gamma_1,\ldots\gamma_{p-1})
$$
\item if $[\gamma_p]=[\gamma_{p-1}^{-1}]$ then
$$
\tw(\gamma_1,\ldots\gamma_p)=0
$$
\item {\sc Relations} Assume that $[\gamma_i]\not=[\gamma_{i+1}]$ then
$$
\tw(\gamma_1,\ldots,\gamma_p)=\frac{\tw(\gamma_1,\gamma_2)\tw(\gamma_1,\gamma_p)\tw(\gamma_2,\gamma_3,\ldots,\gamma_p)}{\tw(\gamma_p,\gamma_2,\gamma_1)}
$$
\end{enumerate}
\end{proposition}
We deduce from the last statement the following corollary
\begin{corollary}\label{ElemGen}
Let $\PP$ be the set of fixed points in $\bgrf$ of non-trivial elements of $\grf$. Then every restriction of an elementary multifraction over $\PP$ is a quotient of product of elementary functions of order 2 and 3.
\end{corollary}
\begin{proof} Let us consider $a,b,c,d$ be four non-trivial elements of $\grf$, then we have
\begin{eqnarray}
\frac{\tw(a,b,c).\tw(c,d)}{\tw(a,d,c)\tw(c,b)}=[b^+;d^+;a^-;c^-].\label{BirElem}
\end{eqnarray}
The result follows.
\end{proof}

Recall that in this section we choose $\PP$ to be the set of fixed points of non-trivial elements of $\grf$. We now prove, 

\begin{proposition}\label{prop:smoothmulti}
Every multifraction -- defined over $\PP$ -- is a smooth function on $\Hn$.
\end{proposition}
\begin{proof} Let $\operatorname{Hom}(n,S)$ be the space of Hitchin homomorphisms. Let $\pi$ be the submersion $$
\pi:\operatorname{Hom}(n,S)\to\Hn=\operatorname{Hom}(n,S)/\operatorname{Aut}(\sln).$$
For every loxodromic element $A$ in $\sln$, let $p_A$ be the projection on the eigenspace of maximal eigenvalue with respect to the other eigenspaces. The map $A\to p_A$ (from the space of loxodromic elements) is smooth. It follows that for any elements $\gamma_1,\ldots,\gamma_k$ in $\grf$ the map from $\operatorname{Hom}(n,S)$ to $\mathbb R$ defined by
$$
\Psi:\rho\to\tr (p_{\rho(\gamma_1)}.\ldots.p_{\rho(\gamma_k)})
$$ 
is smooth. We conclude by observing that $\Psi$ is $\operatorname{Aut}(\sln)$-invariant and that by Equation \eqref{elem-proj}
$$
\Psi=\tw(\gamma_1,\ldots,\gamma_k)\circ\pi.
$$
Thus every elementary function is smooth and by the previous result every multifraction is smooth.
\end{proof}

\subsection{The swapping bracket of elementary functions}
For the sequel, we shall need to compute the swapping brackets of elementary functions. This is given by the following proposition whose proof follows by an immediate application of the definition. We first say that two non-trivial elements $\gamma$ and $\eta$ in $\grf$ are {\em coprime} if for all non-zero integers $m$ and $n$, $\gamma^n\not=\eta^m$.

\begin{proposition}\label{prop:braelem}
Let $\gamma_0,\ldots,\gamma_p$, respectively $\eta_0,\ldots,\eta_q$ be elements of $\grf\setminus\{1\}$ such that 
$(\gamma_i,\gamma_{i+1})$ as well as $(\eta_j,\eta_{j+1})$ are pairwise coprime. Let 
\begin{eqnarray}
{\rm a}_{i,j}&\defeq &[\gamma_i^+\gamma_i^-,\eta_j^+\eta_j^-]\, ,\cr
{\rm b}_{i,j}&\defeq &[\gamma_{i+1}^+\gamma_i^-,\eta_{j+1}^+\eta_j^-]\, ,\cr
{\rm c}_{i,j}&\defeq &[\gamma_{i}^+\gamma_i^-,\eta_{j+1}^+\eta_j^-]\, ,\cr
{\rm d}_{i,j}&\defeq &[\gamma_{i+1}^+\gamma_i^-,\eta_{j}^+\eta_j^-]\, ,\cr
\tw_\gamma&\defeq &\tw(\gamma_0,\ldots,\gamma_p)\cr
\tw_\eta&\defeq &\tw(\eta_0,\ldots,\eta_q)\, .
\end{eqnarray}
 Then
\begin{eqnarray}
\frac{\{\tw_\gamma,\tw_\eta\}}{\tw_\gamma.\tw_\eta}&=&\mathop{\sum_{0\leq i\leq q,}}_{0\leq j\leq p}\left({\rm a}_{i,j}\tw(\gamma_i,\eta_{j})+{\rm b}_{i,j}\frac{\tw(\eta_{j+1},\eta_{j},\gamma_{i+1},\gamma_{i})}
{\tw(\eta_j,\eta_{j+1})\tw(\gamma_i,\gamma_{i+1})}\right)\cr
&-&\mathop{\sum_{0\leq i\leq q}}_{0\leq j\leq p}\left({\rm c}_{i,j}\frac{\tw(\gamma_i,\eta_{j+1},\eta_{j})}
{\tw(\eta_j,\eta_{j+1})}+{\rm d}_{i,j}\frac{\tw(\eta_j,\gamma_{i+1},\gamma_{i})}
{\tw(\gamma_i,\gamma_{i+1})}\right).
\end{eqnarray}
\end{proposition}
\begin{proof}
Using ``logarithmic derivatives", we have
\begin{align*}
\frac{\{\tw_\gamma,\tw_\eta\}}{\tw_\gamma.\tw_\eta}&=\mathop{\sum_{0\leq i\leq p,}}_{0\leq j\leq q}\Bigg(\left(\frac{\{\gamma_{i+1}^+\gamma_i^-,\eta_{j+1}^+\eta_j^-\}}
{\gamma_{i+1}^+\gamma_i^-.\eta_{j+1}^+\eta_j^-}
 + 
 \frac
 {\{\gamma_{i}^+\gamma_i^-,\eta_{j}^+\eta_j^-\}}
 {\gamma_{i}^+\gamma_i^-.\eta_{j}^+\eta_j^-} \right)\cr&-\left(\frac{\{\gamma_{i}^+\gamma_i^-,\eta_{j+1}^+\eta_j^-\}}
{\gamma_{i}^+\gamma_i^-.\eta_{j+1}^+\eta_j^-}
 + 
 \frac
 {\{\gamma_{i+1}^+\gamma_i^-,\eta_{j}^+\eta_j^-\}}
 {\gamma_{i+1}^+\gamma_i^-.\eta_{j}^+\eta_j^-} \right)\Bigg)\cr
 &=\sum_{\substack{0\leq i\leq q\\ 0\leq j\leq p}}\Bigg({\rm b}_{i,j}\frac{\gamma_{i+1}^+\eta_j^-.\eta_{j+1}^+\gamma_i^-}
{\gamma_{i+1}^+\gamma_i^-.\eta_{j+1}^+\eta_j^-}+{\rm a}_{i,j}\frac{\gamma_{i}^+\eta_j^-.\eta_{j}^+\gamma_i^-}{\gamma_{i}^+\gamma_i^-.\eta_{j}^+\eta_j^-}\cr &-{\rm d}_{i,j}\frac{\gamma_{i+1}^+\eta_j^-.\eta_{j}^+\gamma_i^-}{\gamma_{i+1}^+\gamma_i^-.\eta_{j}^+\eta_j^-}-{\rm c}_{i,j}\frac{\gamma_{i}^+\eta_j^-.\eta_{j+1}^+\gamma_i^-}{\gamma_{i}^+\gamma_i^-.\eta_{j+1}^+\eta_j^-}\Bigg).
\end{align*}

From the definition of elementary functions \eqref{defelem}, we get that
\begin{eqnarray*}
\frac{\tw(\eta_{j+1},\eta_{j},\gamma_{i+1},\gamma_{i})}
{\tw(\eta_j,\eta_{j+1})\tw(\gamma_i,\gamma_{i+1})}&=&\frac{\gamma_{i+1}^+\eta_j^-.\eta_{j+1}^+\gamma_i^-}
{\gamma_{i+1}^+\gamma_i^-.\eta_{j+1}^+\eta_j^-}\, ,\cr
\tw(\gamma_i,\eta_{j})&=&\frac{\gamma_{i}^+\eta_j^-.\eta_{j}^+\gamma_i^-}
{\gamma_{i}^+\gamma_i^-.\eta_{j}^+\eta_j^-}\, ,\\
\frac{\tw(\eta_j,\gamma_{i+1},\gamma_{i})}
{\tw(\gamma_i,\gamma_{i+1})}&=&\frac{\gamma_{i+1}^+\eta_j^-.\eta_{j}^+\gamma_i^-}{\gamma_{i+1}^+\gamma_i^-.\eta_{j}^+\eta_j^-}\, ,\\
\frac{\tw(\gamma_i,\eta_{j+1},\eta_{j})}
{\tw(\eta_j,\eta_{j+1})}&=&\frac{\gamma_{i}^+\eta_j^-.\eta_{j+1}^+\gamma_i^-}
{\gamma_{i}^+\gamma_i^-.\eta_{j+1}^+\eta_j^-}.
\end{eqnarray*}
This concludes the proof of the proposition
\end{proof}

\subsection{Length functions}\label{def:lf}

We introduce in this paragraph length functions.
\subsubsection{Length functions from the point of view of the multifraction algebra}
\par
Recall first that $\grf$ acts on $\bgrf$ and thus on $\mathcal B(\bgrf)$.
For any $y\in\bgrf$ and $\beta$ a non-trivial element in $\grf$, let us introduce the following cross fraction, 
$$
p_\beta(y)=
\frac{(
\beta^+,\beta(y)) \cdotp(\beta^-,\beta^{-1}(y))}
{(\beta^+,\beta^{-1}(y)) \cdotp(\beta^-,\beta(y))},
$$

where for readibility we revert to the classical notation $(X,x)$ for pair of points rather than the concatenated notation $Xx$.
We have, for any $\beta$ in $\grf$
$$
\frac{p_\beta(y)}{p_\beta(z)}=\frac{(\beta^2)*F_{y,z}}{F_{y,z}},
$$
where
$$
F_{y,z}=\frac{(\beta^+,\beta^{-1}(y)) \cdotp(\beta^-,\beta^{-1}(z))}{(\beta^+,\beta^{-1}(z)) \cdotp(\beta^-,\beta^{-1}(y))}
$$
In particular, the restriction of $p_\beta(y)$ to the space of $\grf$-invariant cross ratios is independent on the choice of $y$.

For the sake of simplicity, we introduce the following formal series of multifractions and call it a {\em length function}.
$$
\hat\ell_\beta(y)\defeq \frac{1}{2}\log(p_\beta(y)),
$$
extending the bracket by the ``log derivative" formulas
\begin{eqnarray}
\{\hat\ell_\beta(y),q\}\defeq \frac{\{p_\beta(y),q\}}{2\,p_\beta(y)},\ \ \{\hat\ell_\beta(y),\hat\ell_\gamma(z)\}\defeq \frac{\{p_\beta(y),p_\gamma(z)\}}{4\, p_\beta(y).p_\gamma(z)}.
\end{eqnarray}
Observe that
${\rm I}_S(\hat\ell_{\beta^n}(y))=n.{\rm I}_S(\hat\ell_\beta(y))$. 

\subsubsection{Length functions and the character variety}

We can further relate these objects with the period and length defined in Paragraph \ref{def:cr}.
Let $${\rm I}_S: \mathcal B(\bgrf)\to C^\infty(\Hn),$$ denote the restriction of functions from $\mathbb B(\bgrf)$ to
$\Hn$. 

We have for $\beta\in\grf$ that $$
{\rm I}_S(\hat\ell_\beta(y))=\ell_\beta,$$
where $$\ell_\beta(\rho)\defeq\ell_b(\beta),$$
and $\ell_b$ is the period of $\beta$ which respect to the cross ratio associated to $\rho$ (see Section \ref{def:cr}).

\section{The Goldman algebra}
In this section, we first recall the construction of the Atiyah--Bott--Goldman symplectic form on the character variety. We then explain the construction of the Goldman algebra which allows to compute the bracket of Wilson loops in terms of a Lie bracket on the vector space generated by free homotopy classes of loops. 
\subsection{The Atiyah-Bott-Goldman symplectic form}
In \cite{Atiyah:1983}, Atiyah and Bott introduced a symplectic structure on the character variety of representations of closed surface groups in compact Lie group, generalising Poincaré duality. This was later generalised by Goldman for non-compact groups in \cite{Goldman:1986,Goldman:1984} and connected to the Weil--Petersson Kähler form. 
If we identify the tangent space of $\Hn$ at $\rho$ with $H^1_\rho(\mk g)$, where $\mk g$ is the Lie algebra of $\sln$ then the symplectic form is given 
\begin{eqnarray}
\omega_S\left([A],[B]\right)=\int_S \tr(A\wedge B),\label{defABG}
\end{eqnarray}
where $A$ and $B$ are de Rham representatives of the cohomology classes $[A]$ and $[B]$. We denote by $\{ \cdotp, \cdotp\}_S$ the associated Poisson bracket, called the {\em Atiyah--Bott--Goldman (ABG) Poisson bracket} in the sequel, and $\cA(S)$ the Poisson algebra of smooth functions on $\Hn$.
In the next paragraph, we show how to compute the Atiyah--Bott--Goldman bracket, in the case of $\sln$, for the Wilson loops that we introduced in the previous section.
\subsection{Wilson loops and the Goldman algebra}
We describe in this paragraph the Goldman algebra and how it helps computing the ABG-Poisson bracket. Let $\rm C$ be the set of free homotopy class of closed curves on an oriented surface ${S}$. Let ${\mathbb Q}[\rm C]$ be the vector space generated by $\rm C$ over $\mathbb Q$. We extend linearly Wilson loops so that the map $\gamma\mapsto \ww(\gamma)$ is now a linear map from ${\mathbb Q}[\rm C]$ to $C^\infty(\Hn)$. 

Goldman introduced in \cite{Goldman:1986} a Lie bracket on ${\mathbb Q}[\rm C]$. We define it for two elements $\gamma_1$ and $\gamma_2$ of $\rm C\subset\mathbb Q[\rm C]$ and then extend it to $\mathbb Q[\rm C]$ linearly. We choose two curves representing $\gamma_1$ and $\gamma_2$, which we denote the same way. 

If $\gamma_1$ and $\gamma_2$ are two curves from $S^1$ to $S$, an {\em intersection point} is a pair $(a,b)$ in $S^1\times S^1$ so that $\gamma_1(a)=\gamma_2(b)$. By a slight abuse of language, we usually identify an intersection point $(a,b)$ with its image $x=\gamma_1(a)=\gamma_2(b)$. We further assume that $\gamma_1$ and $\gamma_2$ have transverse intersection points.
 
For every intersection point $x$, let $\ii _x$ be the local intersection number at $x$, let $\gamma_1\jl_x \gamma_2$ be the free homotopy class of the curve obtained by composing $\gamma_1$ and $\gamma_2$ in $\pi_1({S},x)$ and finally let
$$
\ii (\gamma_1,\gamma_2)\defeq \sum_{x\in \gamma_1\cap \gamma_2}\ii _x,
$$
be the global intersection number. 

\begin{definition}
The {\em Goldman bracket} of $\gamma_1$ and $\gamma_2$ is the element of $\mathbb Q[\rm C]$ defined by
\begin{eqnarray}
\{\gamma_1,\gamma_2\}&\defeq &\sum_{x\in \gamma_1\cap \gamma_2}\ii _x \cdotp \gamma_1\jl_x \gamma_2.\label{eq:fond}
\end{eqnarray}

\end{definition}

We illustrate in Picture \ref{fig:GoldBra}, the Goldman bracket of two curves.
\begin{figure}
 \centering
 \subfloat[Two curves]{\label{fig:gull}\includegraphics[width=0.5\textwidth]{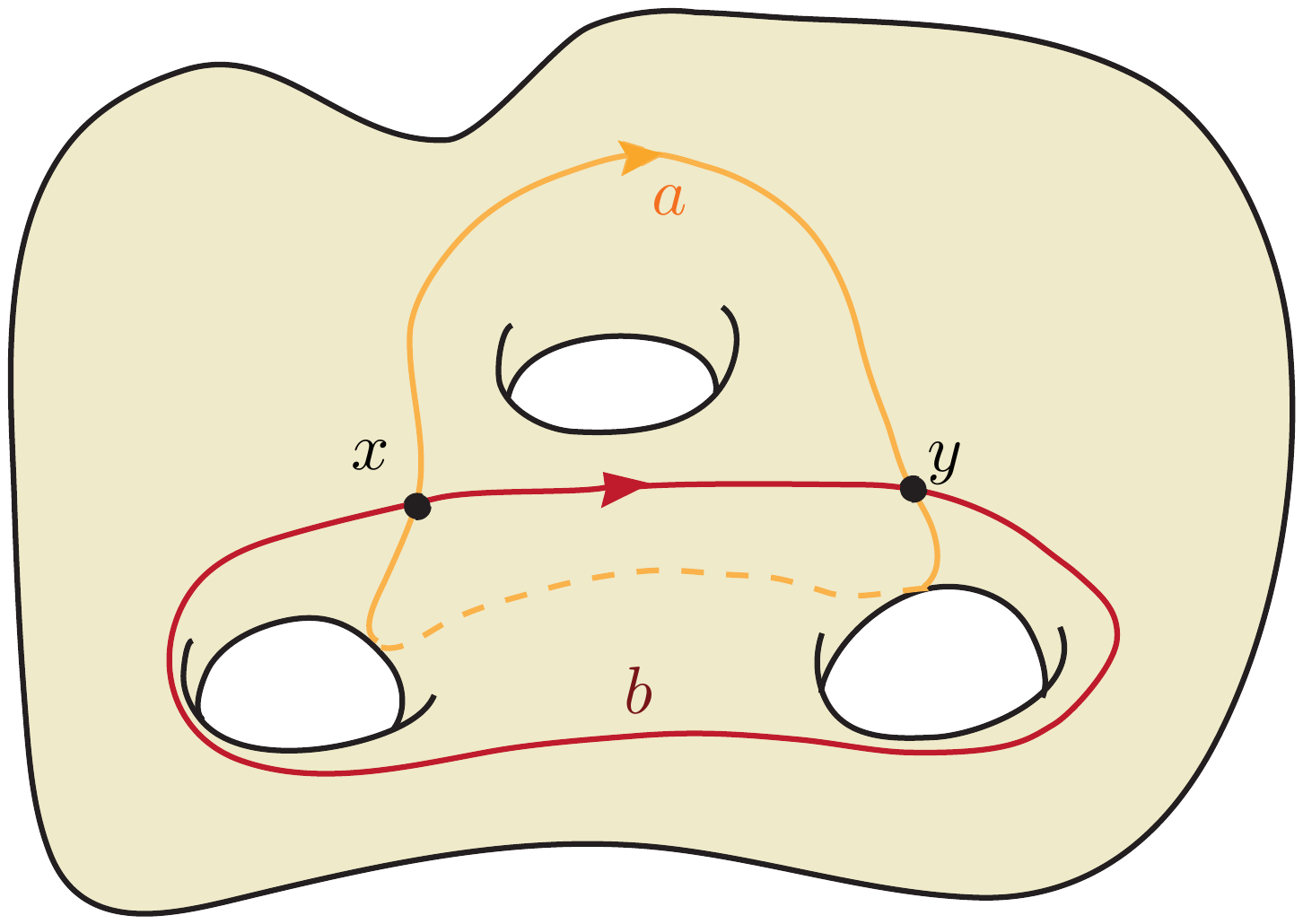}}        
 \subfloat[Their Goldman bracket]{\label{fig:tiger}\includegraphics[width=0.5\textwidth]{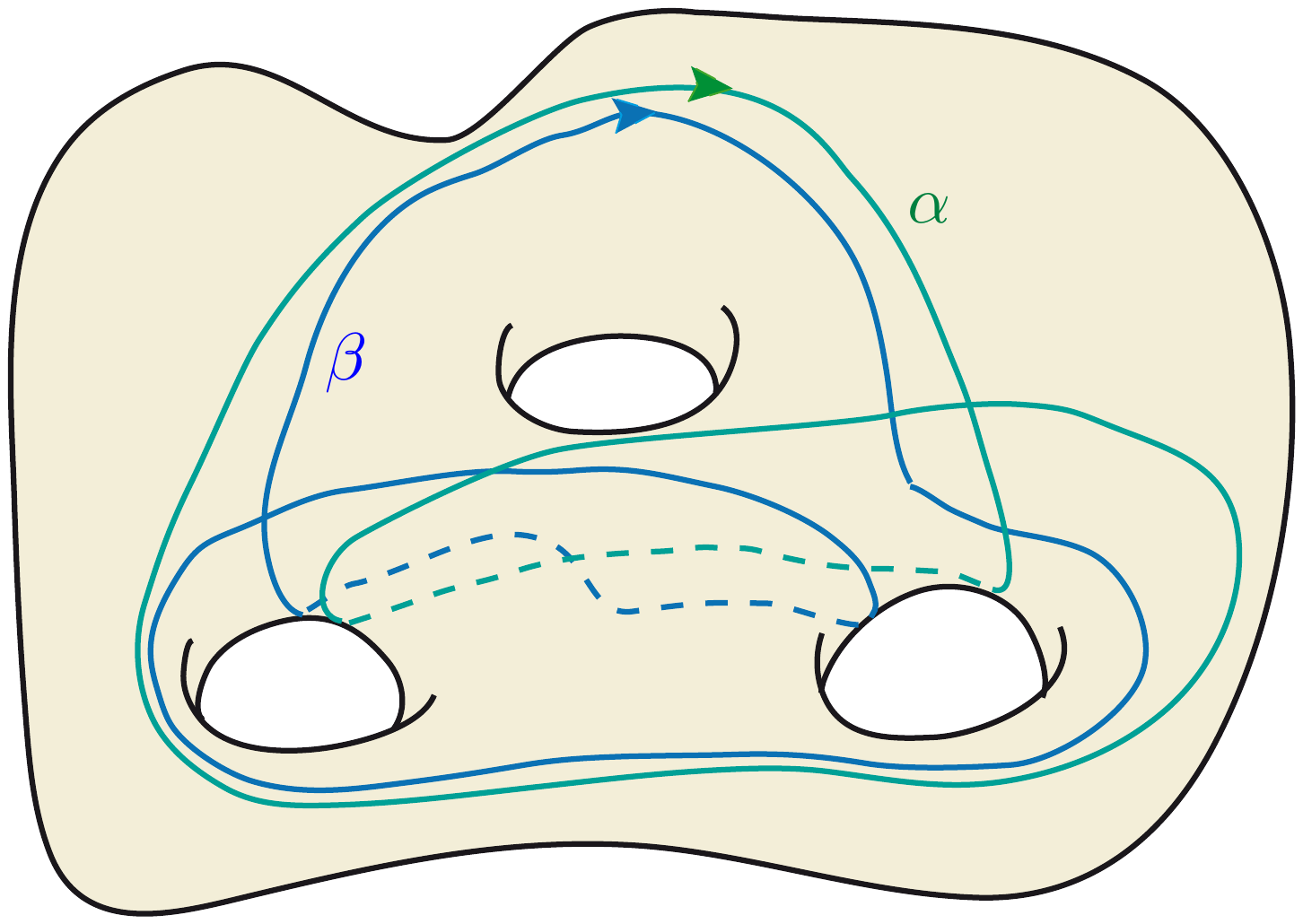}}
 \caption{$\{b,a\}=\alpha-\beta$}
 \label{fig:GoldBra}
\end{figure}
Goldman proved in \cite{Goldman:1986} that this bracket does not depend on the choice of representatives and is a Lie bracket. Moreover this bracket is related to the ABG-Poisson bracket as follows.
\begin{theorem}{\sc{[Goldman]}}
Let $\gamma_1$ and $\gamma_2$ be two loops on $S$. Then the ABG-Poisson bracket of the two corresponding Wilson loops in $\Hn$ is 
\begin{eqnarray}
\{\ww({\gamma_1}),\ww({\gamma_2})\}_S=\ww({\{\gamma_1,\gamma_2\}})- \frac{\ii (\gamma_1,\gamma_2)}{n} \ww({\gamma_1}) \cdotp\ww({\gamma_2})\, .\label{ABGw}
\end{eqnarray}
\end{theorem}
We just stated Goldman theorem for the case of $\Hn$, but the theorem has a formulation in the general case of character varieties for semi-simple groups. A different proof can also be found in \cite{Labourie:2013ka}.

\section{Vanishing sequences and the main results}\label{sec:main}

In this section, we first recall the definition of the length functions on the character varieties, introduce the notion of a vanishing sequence of finite index subgroups of a surface group and state our main results relating the swapping algebra to the Goldman algebra. All these results will be proved in Section \ref{sec:gold}. Let as usual $${\rm I}_S: \mathcal B(\bgrf)\to C^\infty(\Hn),$$ denote the restriction of functions from $\mathbb B(\bgrf)$ to
$\Hn$. 

\subsection{Poisson brackets of length functions}

We explain in this section our results concerning length functions (See Paragraph \ref{def:lf} for notations and definitions).
Our first result relates the Goldman and the swapping Poisson bracket.

\begin{theorem}\label{theo:braleng}
Let $\gamma$ and $\eta$ be two geodesics with at most one intersection point, then we have 
$$
\lim_{n\to\infty}{\rm I}_S(\{\hat\ell_{\gamma^n}(y),\hat\ell_{\eta^n}(y)\})=\frac{1}{4}\{\ell_{\gamma},\ell_{\eta}\}_S\, .
$$
\end{theorem}

In the course of the proof of this result, we prove the following result of independent interest which is an extension of Wolpert Formula \cite{Wolpert:1983td,Wolpert:1982eo}
\begin{theorem}{\sc [Generalised Wolpert Formula]}\label{theo:genWolp}.

Let $\gamma$ and $\eta$ two closed geodesics with a unique intersection point then the Goldman bracket of the two length functions $\ell_\gamma$ and $\ell_\eta$ seen as functions on the Hitchin component is 
\begin{equation}
\{\ell_\gamma,\ell_\eta\}_S=\ii(\gamma,\eta)\sum_{\epsilon,\epsilon^\prime\in\{-1,1\}}\epsilon\epsilon^\prime.\tw(\gamma^{\epsilon}.\eta^{\epsilon^\prime}),
\end{equation}
where we recall that
$$
\tw(\xi,\zeta)(\rho)=\bb_\rho(\xi^+,\zeta^+,\zeta^-,\xi^-).
$$
\end{theorem}

We prove these two results in Paragraph \ref{proof-length}

\subsection{Poisson brackets of multifractions}

We now relate in general the swapping bracket and the Goldman bracket. Our result can be described by saying that the swapping bracket is an inverse limit (with respect to sequences of covering) of the Goldman bracket, or in other words that the swapping racket is a universal (in genus) Goldman bracket.

\subsubsection{Vanishing sequences}\label{par:vaseq}

We now assume that $S$ is equipped with an auxiliary hyperbolic metric. Let $\tilde S$ be the universal cover of $S$ so that $S=\tilde S/\pi_1(S)$. For any $\gamma$ in $\pi_1(S)$, we denote by $\tilde\gamma$ its axis in $\tilde S$ and $\langle\gamma\rangle$ the cyclic subgroup that it generates.
Recall that we say that two elements $\gamma$ and $\eta$ of $\grf$ are {\em coprime} if $\langle\gamma\rangle\cap\langle\eta\rangle=\{1\}$. 

Let $\seq{\Gamma}$ be a sequence of nested finite index subgroups of $\Gamma_0\defeq \grf$. Then let $S_n\defeq \tilde S/\Gamma_n$. For any $\gamma\in\Gamma$ let $\langle\gamma\rangle_n\defeq \langle\gamma\rangle\cap\Gamma_n$. Finally, let $\pi_n$ be the projection from $\tilde S$ to $S_n$ and let $\tilde \gamma_n\defeq \pi_n(\tilde\gamma)$.

\begin{definition}
Let $\seq{\Gamma}$ be a sequence of nested finite index normal subgroups of $\Gamma_0\defeq \grf$. We say that $\seq{\Gamma}$ is a {\em vanishing sequence} 
if for all $\gamma$ and $\eta$ in $\grf$, for any set $H$, invariant by left multiplication by $\gamma$ and right multiplication by $\eta$, whose projection in $\langle\eta\rangle\backslash\grf/\langle\gamma\rangle$ is finite, there exists $n_0$, such that for all $n>n_0$, $H\cap\Gamma_n\subset\langle\eta\rangle.\langle\gamma\rangle$.
\end{definition}

We shall use freely the following immediate consequence
\begin{proposition}\label{VSbis} Let $\seq{\Gamma}$ be a {\em vanishing sequence} with $\Gamma_0=\grf$.
For any $\eta$ and $\gamma$ in $\grf$, for any finite subset $H_0$ of $\grf$ so that $H_0\cap\left(\langle\eta\rangle \cdotp\langle\gamma\rangle\right)=\emptyset$, there exists $p_0$ so that for all $p>p_0$, then  
$$H_0\cap\left(\langle\eta\rangle \cdotp\Gamma_p \cdotp\langle\gamma\rangle\right)=\emptyset.$$
\end{proposition}

We prove in Appendix \ref{sec:vanishexist} that vanishing sequences exist. This is an immediate consequence of a result by G.~Niblo \cite{Niblo:1992uy}.

\subsubsection{Sequences of subgroups and limits}

Let $\PP$ be the subset of $\bgrf$ given by the end points of periodic geodesics. Let $\ms G$ be the set of pairs of points $\gamma=(\gamma^-,\gamma^+)$ in $\PP$ which correspond to fixed points of by an element of the group $\partial_\infty(\pi_1({S})$. Observe that given any finite index subgroup $\Gamma$ of $\pi_1(S)$, the set $\ms G$ is in bijection with the set of primitive elements of $\Gamma$.

In the sequel, we shall freely identify elements of $\ms G$ with primitive elements in $\grf$ or any of its finite index subgroup.

We associate to a sequence $\sigma=\seq{\Gamma}$ of finite index subgroups of $\pi_1({S})$ the inverse limit ${S}_\sigma$ of $\seq{S_m\defeq \tilde{S}/\Gamma}$, where $\tilde S$ is the universal cover of $S$.

Observe that we have a map ${\ms I}$ from $\mathcal B(\PP)$ to $\cA({S}_{\sigma})$ which by definition is the projective  limit of $\{\cA (S_m)\}_{m\in\mathbb N}$.

\begin{definition}
Let $\seq{g}$ be a sequence of functions, so that $g_m\in \cA({S}_m)$, we say that $\seq{g}$ {\em converges} to the function $h$ in $\cA({S}_{\sigma})$ and write
$$
\lim_{m\rightarrow\infty}g_m=h,
$$
if for all $p$
$$
\lim_{n\rightarrow\infty}{\ms I}_{{S}_p}(g_n)={\ms I}_{{S}_p}(h),
$$
where ${\ms I}_{{S}_p}$ is the restriction with value in $\cA(S_p)$. 
\end{definition}

\subsubsection{Poisson brackets of multifractions}

The following result explains that the algebra of multifractions is an inverse limit of Goldman algebras with respect to vanishing sequences.

\begin{theorem}\label{vanish-sequence} Let $\seq{\Gamma}$ be a vanishing sequence of subgroups of $\pi_1(S)$. Let $\PP\subset\bgrf$ be the set of end points of geodesics. Let $b_0$ and $b_1$ be two multifractions in $\mathcal B(\PP)$. Then we have
$$
\lim_{n\rightarrow\infty}\{{\rm I}(b_0),{\rm I}(b_1)\}_{{S}_n}={\ms I}\left(\{b_0,b_1\}_{W}\right).
$$
\end{theorem}

We prove this result in Paragraph \ref{proof:vanish-sequence}.

\section{Product formulas and bouquet in good position}

In this section, we wish to describe the Goldman bracket of curves which are compositions of many arcs. We shall call such a description a {\em product formula} and produce several instances of such formulas. This section is part of the technical core of this article.

The first formula -- see Proposition \ref{prodform} -- deals with a rather general situation computing the Goldman bracket of curves which are compositions of many arcs. Then, considering repetition, and using special collection of arcs called {\em bouquets in good positions} -- see Definition \ref{def:gp} -- we prove a refinement of the product formula in Proposition \ref{prodform2}. Proposition \ref{prodform2} is the first key result of this section.

Finally, in Proposition \ref{SettingArcGoodPosition}, we explain under which topological condition we can find bouquet in good position and compute the various intersection numbers involved in Proposition \ref{prodform2}. Proposition \ref{SettingArcGoodPosition} is the second key result of this section.

\subsection{An alternative formulation of the Goldman bracket} We first need to give an alternative description of the Goldman bracket.

Let $\bar\gamma_1$ and $\bar\gamma_2$ be two arcs passing through a base point $x_0$. For any point $x$ in $\bar\gamma_i$, let $a_i(x)$ be the path along $\bar\gamma_i$ joining $x_0$ to $x$. 
\begin{definition}{\sc[Intersection loops]}
Following this notation, for any $x\in\bar\gamma_1\cap\bar\gamma_2$, the homotopy class 
$$c_x(\bar\gamma_1,\bar\gamma_2)\defeq a_1(x). a_2(x)^{-1}\in\pi_1(S,x_0)$$ is called an {\em intersection loop} at $x$.
-- see Picture \ref{fig:InterLoop}.
\end{definition}

\begin{figure}[h]
 \begin{center}
  \includegraphics[width=0.6\textwidth]{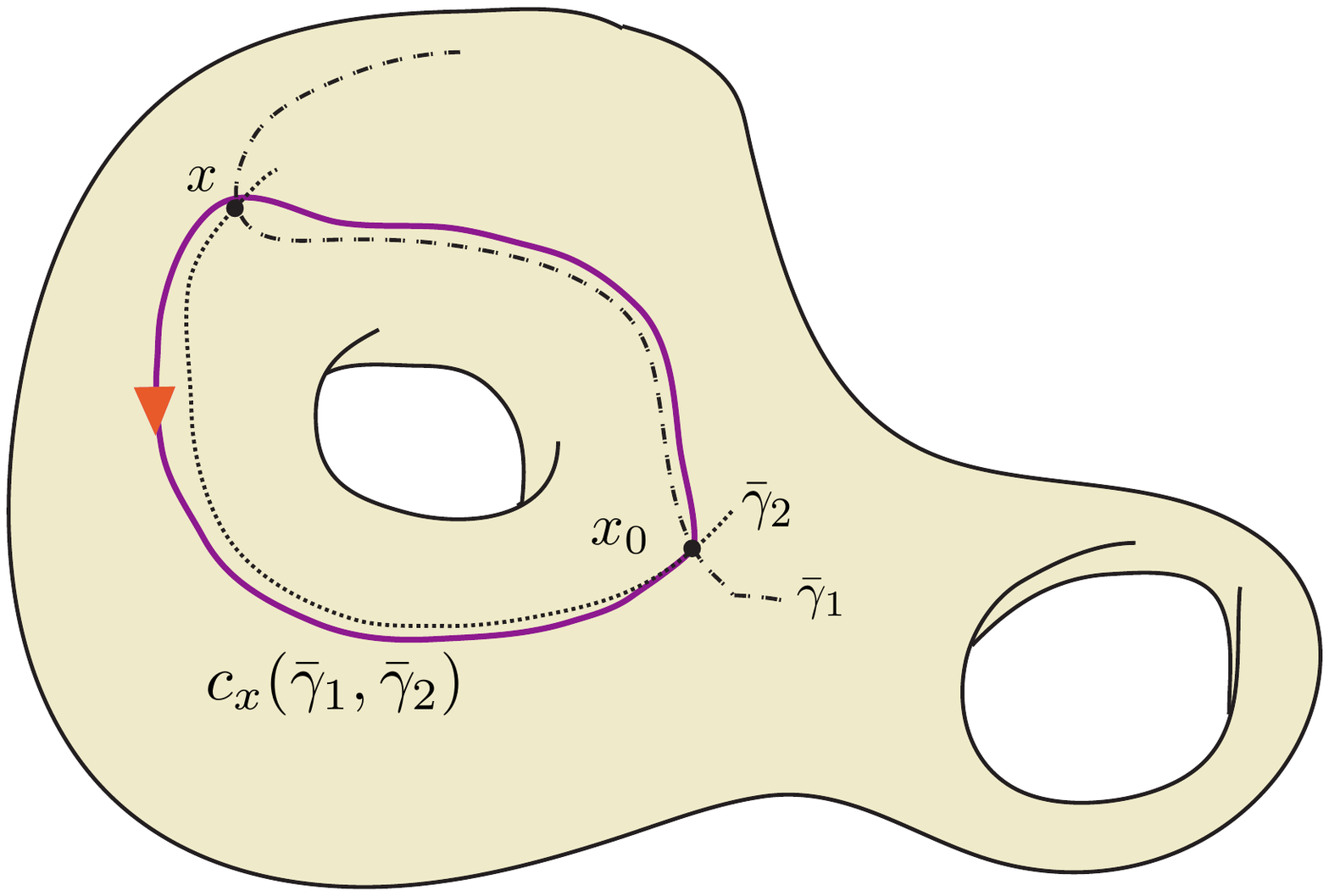}
 \end{center}
 \caption{Intersection loop}
 \label{fig:InterLoop}
\end{figure}
The goal of this paragraph is the following proposition
\begin{proposition}\label{alt-Goldbracket]}
Let $\gamma_1$ and $\gamma_2$ be two free homotopy classes of loops represented by curves $\bar\gamma_1$ and $\bar\gamma_2$ passing though $x_0$.
Then, the Goldman bracket in $\mathbb Q[\rm C]$ of the associated loops is given using intersection loops by
\begin{eqnarray}
\{\gamma_1,\gamma_2\}_S=\sum_{x\in \gamma_1\cap\gamma_2}{\ii}_x\bar\gamma_1\cdotp c_x\cdotp\bar\gamma_2. c_x^{-1}.
\end{eqnarray}
\end{proposition}
This proposition is an immediate consequence of the following
\begin{proposition}
Let $\gamma_1$ and $\gamma_2$ be two loops passing though $x_0$. Then for every $x\in\gamma_1\cap\gamma_2$, we have
$$
\gamma_1\jl_x\gamma_2=\gamma_1.c_x(\gamma_1,\gamma_2)\gamma_2.c_x(\gamma_1,\gamma_2)^{-1},
$$
as free homotopy classes of curves
\end{proposition}

\begin{proof}
Let as before $a_i$ the arc along $\gamma_i$ joining $x_0$ to $x$ and $c_x=a_1.a_2^{-1}$, then
\begin{eqnarray*}
\gamma_1\jl_x\gamma_2&=&a_1^{-1} \gamma_1 a_1 a_2^{-1} \gamma_2 a_2\cr
&=&a_1^{-1} \gamma_1 c_x \gamma_2 a_2\cr
&=&a_1^{-1} \gamma_1 c_x \gamma_2 c_x^{-1} a_1
\end{eqnarray*}
Thus $\gamma_1\jl_x\gamma_2$ is freely homotopic to $\gamma_1 c_x \gamma_2c_x^{-1}$ 
\end{proof}

\subsection{The product formula}\label{pf1} We need to express the Goldman bracket of Wilson loops of curves consisting of many arcs. We work with the following data, see Figure \ref{fig:flower} for a partial drawing:
\begin{itemize}
\item Let $\xi_0,\ldots,\xi_q$ and $\zeta_0,\ldots,\zeta_{q'}$ be two tuples of arcs so that $A=\xi_0\ldots\xi_q$ and $B=\zeta_0\ldots\zeta_q$ are closed curves.
\item Assume furthermore that for all pairs $(i,j)$, $\xi_i$ and $\zeta_j$ have transverse intersections and do not intersect at their end points.
\item Let $u_i$ -- respectively $v_i$ -- be arcs joining a base point $x_0$ to the origin of $\xi_i$ -- respectively $\zeta_i$.
\end{itemize}
Let us introduce the following notations 
\begin{itemize}
\item for every $x\in \xi_i\cap\zeta_j$, let
$
c^{i,j}_x\defeq c_x(u_i\xi_i,v_j\zeta_j),
$
\item for any $\xi\in\pi_1(S)$, let 
$
I_{i,j}(\xi)\defeq \sum_{x\in\xi_i\cap\xi_j\mid \xi=c^{i,j}_x}\ii(x),
$
\item let us denote  $A_i\defeq u_i\xi_{i}\xi_{i+1}\ldots\xi_{i-1}u_i^{-1}$ and $B_j\defeq v_j\zeta_{j}\zeta_{j+1}\ldots\zeta_{j-1}v_{j}^{-1}$.
\end{itemize}

\begin{figure}[h]
 \begin{center}
  \includegraphics[width=0.7\textwidth]{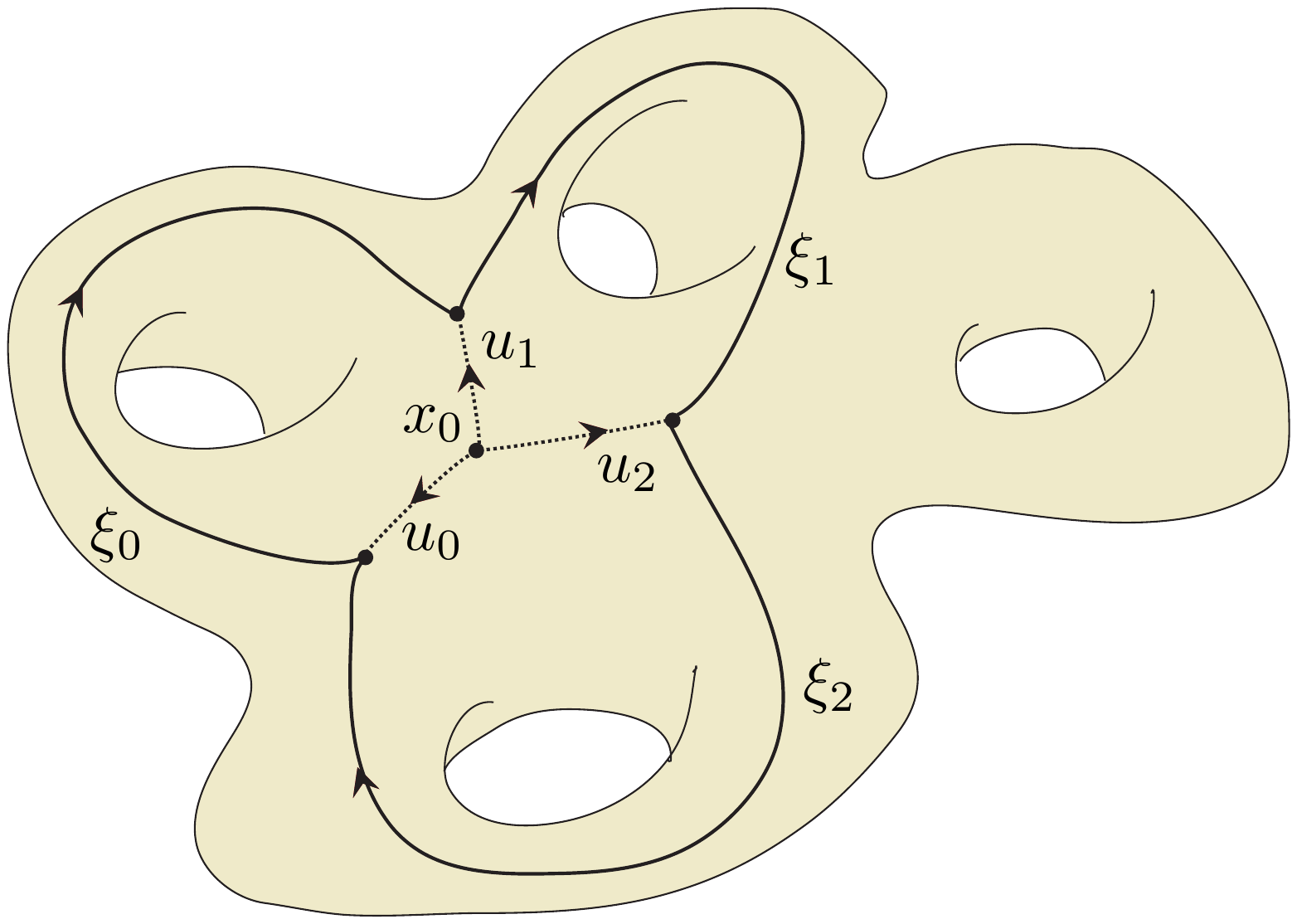}
 \end{center}
 \caption{Arcs $\xi_i$ and $u_i$}
 \label{fig:flower}
\end{figure}
\begin{proposition}{\sc[Product formula]}\label{prodform} Using the notations and assumptions described above, we have the following equality in $\mathbb Q[\rm C]$,
\begin{eqnarray}
\{A,B\}
&=&\sum_{\substack{0\leq i\leq q\\ 0\leq j\leq q'}}\left(\sum_{x\in \xi_i\cap\zeta_j}{\ii}_xA_i .c_x^{i,j}. B_j (c_x^{i,j})^{-1}\right)\label{eq:pf1}\\
&=&\sum_{\substack{0\leq i\leq q\\ 0\leq j\leq q'}}\left(\sum_{\xi\in \pi_1(S)}{I_{i,j}(\xi)}A_i .\xi. B_j \xi^{-1}\right)\, .\label{eq:pf1b}
\end{eqnarray}
\end{proposition}
We first prove a preliminary proposition and postpone the proof of Proposition \ref{prodform} until the next paragraph. 
\subsubsection{A preliminary case}
We first study the following simple situation
\begin{itemize}
\item Let $\xi$ and $\eta$ be two closed curves. Assume that $\xi=\xi_1.\xi_2$ and $\zeta=\zeta_1.\zeta_2$. Assume that of all $i,j$, $\xi_i$ and $\zeta_j$ are closed curves with transverse intersections that do not intersect at their origin.
\item  Let $u_i$ and $v_j$ be arcs from $x_0$ to $\xi_i$ and $\eta_j$ respectively.
\item Let $\tilde\xi_i\defeq u_i\xi_iu_{i+1}^{-1}$, $\tilde\zeta_j\defeq v_j\zeta_jv_{j+1}^{-1}$ and $c^{i,j}_x\defeq c_x(\tilde\xi_i,\tilde\zeta_j)\in\pi_1(S,x_0)$ for $x\in \xi_i\cap\zeta_j$. 
\end{itemize}

\begin{proposition}\label{prelim-prod}
We have the following equality in $\mathbb Q[\rm C]$,
\begin{eqnarray}
\sum_{x\in \xi\cap\zeta}{\ii}_x\cdotp\xi\jl_x\zeta=\mathop{\sum_{\substack{x\in \xi_i\cap\zeta_j\\ 1\leq i,j\leq 2}}}{\ii}_x\cdotp\tilde\xi_{i}.\tilde\xi_{i+1}.c^{i,j}_x.\tilde\zeta_{j}.\tilde\zeta_{j+1}.(c^{i,j}_x)^{-1}.
\end{eqnarray}
\end{proposition}
\begin{proof}
First, we observe that for any two pairs of curves $(\xi_1,\xi_2)$ and 
$(\zeta_1,\zeta_2)$ we have $$(\xi_1.\xi_2)\cap(\zeta_1.\zeta_2)=\bigsqcup_{i,j}(\xi_i\cap\zeta_j).$$
Let us denote
$$
c_x\defeq c_x(\tilde\xi_1.\tilde\xi_2,\tilde\zeta_1.\tilde\zeta_2).
$$
We then have
\begin{eqnarray*}
x\in \xi_1\cap\zeta_1&\implies& c_x=c^{1,1}_x,\\
x\in \xi_2\cap\zeta_1&\implies& c_x=\tilde\xi_1.c^{2,1}_x,\\
x\in \xi_1\cap\zeta_2&\implies& c_x=c_x^{1,2}.\tilde\zeta_1^{-1},\\
x\in \xi_2\cap\zeta_2&\implies& c_x=\tilde\xi_1.c_x^{2,2}.\tilde\zeta_1^{-1}.
\end{eqnarray*}
Thus in all cases, if $x\in \xi_i\cap\zeta_j$ we have the following equality of free homotopy classes
$$
\tilde\xi_1\tilde\xi_2c_x\tilde\zeta_1\tilde\zeta_2=\tilde\xi_i\tilde\xi_{i+1}. c_x^{i,j}.\tilde\zeta_j.\tilde\zeta_{j+1}.(c_x^{i,j})^{-1},
$$
Thus we obtain the product formula.
\begin{multline}
\sum_{x\in (\xi_1\xi_2)\cap(\zeta_1\zeta_2)}{\ii}_x\left(\tilde\xi_1.\tilde\xi_2. c_x.\tilde\zeta_1.\tilde\zeta_2.c_x^{-1}\right)\\=\sum_{i,j}\left(\sum_{x\in \xi_i\cap\zeta_j}{\ii}_x (\tilde\xi_i\tilde\xi_{i+1}. c_x^{i,j}.\tilde\zeta_j.\tilde\zeta_{j+1}.(c_x^{i,j})^{-1})\right).
\end{multline}
This concludes the proof.
\end{proof}

\subsubsection{Proof of Proposition \ref{prodform}}

Obviously Formula \eqref{eq:pf1b} is an immediate consequence of Formula \eqref{eq:pf1}, so we concentrate on the latter. 

First, we observe that the product formula when $\xi_i$ and $\zeta_j$ are closed curves follows by induction from Proposition \ref{prelim-prod}.

Let us now make the following observation. Let $a$, $\xi$ and $\zeta$ be three arcs, transverse to a curve $\kappa$. Assume that $\xi.a.a^{-1}.\zeta$ is a closed curve, then we have the following equalities in $\mathbb Q[\rm C]$
\begin{eqnarray}
\xi.a.a^{-1}.\zeta&=&\xi.\zeta,\cr
\sum_{x\in (\xi.\zeta)\cap\kappa}\!\!\!\!{\ii}_x\xi.\zeta .c_x\kappa .c_x^{-1}&=& \!\!\!\!\!\!\!\!\!\!\!\!\sum_{x\in (\xi.a.a^{-1}.\zeta)\cap\kappa}\! \!\!\!{\ii}_x\xi.a.a^{-1}.\zeta. c_x\kappa.c_x^{-1}.\label{inter-cancel}
\end{eqnarray}
The first equality is obvious. For the second we notice that
every intersection point of $a$ with $\kappa$ appears twice with a different sign. 

We can now extend the product formula to arcs: we choose auxiliary arcs $\alpha_i$ joining $x_0$ to the initial point of $\xi_i$, similarly auxiliary arcs $\beta_i$ joining $x_0$ to the initial point of $\zeta_i$ and replace $\xi_i$ by the closed curves $\hat\xi_i=\alpha_i\xi_i\alpha_{i+1}^{-1}$ and $\hat\zeta_i=\beta_i\zeta_i\beta_{i+1}^{-1}$ respectively.
From Assertion \eqref{inter-cancel}, since the product formula holds for the closed curves $\hat\zeta_j$ and $\hat\xi_i$, it holds for the arcs $\zeta_j$ and $\xi_i$.

\subsection{Bouquets in good position and the product formula}\label{sec:gp}

We shall need a special case of the product formula when we allow some repetitions in the arcs.

\subsubsection{Bouquets in good position}

\begin{definition}{\sc[Flowers and bouquets]}
\begin{enumerate}
\item A {\em flower} based at $(x_0,\ldots,x_q)$ is a collection of arcs
$$
\mathcal S\defeq ((g_0,\ldots,g_q),(\alpha_0,\ldots,\alpha_q)),
$$
such that 
 \begin{itemize}
\item $g_i$ are closed curves based at $x_i$ representing primitive elements in the fundamental group, 
\item  $\alpha_i$ are arcs, called {\em connecting arcs}, joining $x_{i}$ to $x_{i+1}$.
\end{itemize}
\item A {\em bouquet} is a triple
$$\cF=(\mathcal S_0 ,\mathcal S_1 , V),$$
where $\mathcal S_1 $ and $\mathcal S_0 $ are flowers based at $(x_0,\ldots,x_q)$ and $(y_0,\ldots,y_{q'})$ respectively and $V$ is an arc joining $x_0$ and $y_0$. 
\item We finally say that the bouquet
 $\cF$ {\em represents} $((\gamma_0,\ldots,\gamma_q),(\eta_0,\ldots,\eta_{q'}))$, where $\gamma_i$ and $\eta_j$ are the elements of $\pi_1(S,x_0)$ defined by $\gamma_i=U_ig_i U_i^{-1}$ and $\eta_j=V_jh_j V_j^{-1}$, where $U_i\defeq \alpha_0\ldots\alpha_{i-1}$ and $V_j\defeq V.\beta_0\ldots\beta_{j-1}$, 
\end{enumerate}
\end{definition}
We shall also need bouquets which have specially neat configurations: let
$$
\cF=\left(\left((g_0,\ldots,g_q),(\alpha_0,\ldots,\alpha_q)\right),\left((h_0,\ldots,h_{q'}),(\beta_0,\ldots,\beta_{q'})\right),V\right)
$$
be a bouquet of flowers based respectively at $(x_0,\ldots,x_q)$ and $(y_0,\ldots,y_{q'})$. 
 
\begin{definition}\label{def:gp}{\sc[Good position]} We say
\begin{enumerate} \item $\cF$ is in a {\em good position} if \begin{itemize}
\item the arcs $\alpha_i$ and $g_i$ intersect transversely the arcs $\beta_j$ and $h_j$ at points different than $x_i$ and $y_j$ for all $i,j$,
\item the closed curves $\alpha_0\ldots\alpha_q$ and $\beta_0\ldots\beta_{q'}$ are homotopic to zero.
\end{itemize}
\item $\cF$ is in a {\em homotopically good position} if it is in a good position and if the following intersection loops are homotopically trivial
\begin{eqnarray}
c_x(U_i.\alpha_i, V_j.\beta_j),\hbox{ for } x\in \alpha_i\cap\beta_j\cr
c_x(U_i.\alpha_i, V_j.h_j),\hbox{ for } x\in \alpha_i\cap h_j\cr
c_x(U_i.g_i, V_j.\beta_j),\hbox{ for } x\in g_i\cap\beta_j,
\end{eqnarray}
where $U_i\defeq \alpha_0\ldots\alpha_{i-1}$ and $V_j\defeq V.\beta_0\ldots\beta_{i-1}$.
\end{enumerate}
\end{definition}

In Figure \ref{fig:flower2}, we have represented two flowers, one in blue, the other in red where the connecting arcs $\alpha_i$ and $\beta_i$ are dotted. In this figure all intersection loops corresponding to the four yellow transverse intersection points are drawn in the orange contractible region. Thus the bouquet is in a homotopically good position.
\begin{figure}[h]
 \begin{center}
  \includegraphics[width=1\textwidth]{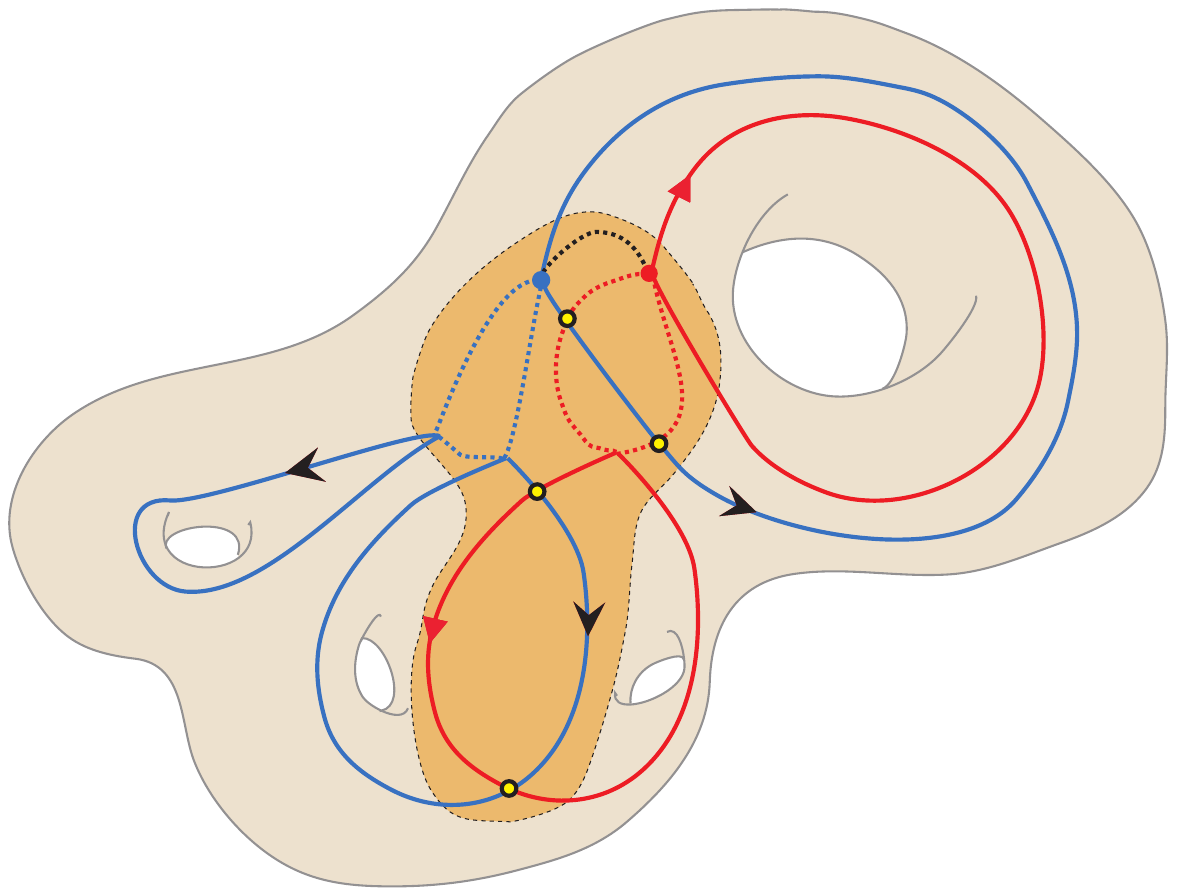}
 \end{center}
 \caption{Bouquet in good position}
 \label{fig:flower2}
\end{figure}

\subsubsection{Product formula for bouquets}\label{FpGp}

Let $\cF$ be a bouquet as above in good position. Let us consider the closed curves
\begin{eqnarray*}
{\bf F}_{i}^{(p,n)}&\defeq &U_i.g_i^{n}.(\alpha_i.g_{i+1}^p\alpha_{i+1}\ldots g_{i-1}^p\alpha_{i-1})g_i^{p-n}U_i^{-1},\\ 
{\bf G}_{i}^{(p,n)}&\defeq &V_i.h_i^{n}.(\beta_i.h_{i+1}^p\beta_{i+1}\ldots h_{i-1}^p\beta_{i-1})h_i^{p-n}V_i^{-1}.\end{eqnarray*}
To simplify notation, let us write
$
 {\bf F}^{(p)}\defeq {\bf F}_0^{(p,0)}$ and $
 {\bf G}^{(p)}\defeq {\bf G}_0^{(p,0)}$.
 Let us denote
 \begin{eqnarray*}
 H_{i,j}&\defeq &\{c_x(U_i.g_i,V_j.h_j)\mid x\in g_i\cap h_j, c_x(U_i.g_i,V_j.h_j)\hbox{ is homotopically trivial }\},\cr
 C_{i,j}&\defeq &\{c_x(U_i.g_i,V_j.h_j)\mid x\in g_i\cap h_j, c_x(U_i.g_i,V_j.h_j)\hbox{ not homotopically trivial }\}.
 \end{eqnarray*}
Let finally
\begin{align}
{\rm f}_{i,j}(\cF)&\defeq \sum_{\xi\in H_{i,j}}I_{i,j}(\xi),&
{\rm m}_{i,j}(\cF)&\defeq {\ii}(g_i,\beta_j),\cr
{\rm n}_{i,j}(\cF)&\defeq {\ii}(\alpha_i,h_j),&
{\rm q}_{i,j}(\cF)&\defeq {\ii}(\alpha_i,\beta_j),\label{fmnq}
\end{align}
where we recall that for any $\xi \in\pi_1(S)$, we denote
$$
I_{i,j}(\xi)=\sum_{x\in g_i\cap h_j\mid c_x(U_i.g_i,V_j.h_j)=\xi}\ii_x.
$$ 

We can rewrite the product formula.
\begin{proposition}\label{prodform2}{\sc[Product Formula in good position]}
Assuming the bouquet $\cF$ is in a homotopically good position and using the above notation, we have the following equality in $\mathbb Q[\rm C]$,
\begin{eqnarray} \{{\bf F}^{(p)},{\bf G}^{(p)}\} &=&\sum_{\substack{0\leq i\leq q\\ 0\leq j\leq q'}}\Bigg(\mathop{\sum_{1\leq m'\leq p}}_{1\leq m\leq p}{\rm f}_{i,j} (\cF)\left({\bf F}^{(p,m')}_i {\bf G}^{(p,m)}_j\right)
+\mathop{\sum}_{1\leq m\leq p}{\rm m}_{i,j}(\cF) \left({\bf F}^{(p,m)}_i {\bf G}^{(p,0)}_j \right)\cr
& &
+\mathop{\sum_{1\leq m'\leq p}}{\rm n}_{i,j}(\cF) \left({\bf F}^{(p,0)}_i  {\bf G}^{(p,m')}_j \right)+{\rm q}_{i,j}(\cF) \left({\bf F}^{(p,0)}_i {\bf G}^{(p,0)}_j \right)\Bigg)
\cr
& &+\sum_{\substack{0\leq i\leq q\\ 0\leq j\leq q'}}\sum_{\xi\in C_{i,j}}I_{i,j}(\xi)\Bigg(\mathop{\sum_{1\leq m'\leq p}}_{1\leq m\leq p}\left({\bf F}^{(p,m')}_i \xi{\bf G}^{(p,m)}_j\xi^{-1}\right)\Bigg).\label{eq:pf2}\end{eqnarray}
\end{proposition}

\begin{proof} This will be just another way to write the product formula. 
We consider the arcs $\xi_i$ defined by
\begin{itemize}
\item $\xi_i\defeq g_{j}$ if $i=j.(p+1)+n$ with $1\leq n\leq p$,
\item $\xi_i\defeq \alpha_{j}$ if $i=j(p+1)$.
\end{itemize}
Similarly, we consider the arcs $\zeta_i$ 
\begin{itemize}
\item $\zeta_i\defeq h_{j}$ if $i=j.(p+1)+n$ with $1\leq n\leq p$,
\item $\zeta_i\defeq \beta_{j}$ if $i=j.(p+1)$.
\end{itemize}
Let now finally consider the following arcs,
\begin{itemize}
\item $u_i\defeq U_j=\alpha_0.\ldots\alpha_j,$ if $i=j.(p+1)+n$ with $1\leq n\leq p$,
\item $v_i\defeq V_j=V.\beta_0.\ldots\beta_j,$, if $i=j.(p+1)+n$ with $1\leq n\leq p$,
\end{itemize}
so that $u_i$, respectively $v_i$, goes from $x_0$ to $x_j$, respectively $x_0$ to $y_j$.

We now apply Formulae \eqref{eq:pf1} and \eqref{eq:pf2} for the arcs $\xi_i,u_i,\zeta_j,v_j$.
Observe that using the notation of Paragraph \ref{pf1}, we have
$${\bf F}^{(p)}=A, \ \ {\bf G}^{(p)}=B.$$ 
 We now have to identify the term in the right hand sides of Formulae \eqref{eq:pf1} and \eqref{eq:pf2}, and in particular understand the arcs 
$A_i$, $B_j$, $c_x^{i,j}$ that appears in the right hand side of Formula \eqref{eq:pf2}. By definition
$$A_i=u_i\xi_{i}\xi_{i+1}\ldots\xi_{i-1}u_i^{-1}.$$
Thus if $i=j(p+1)+m$ with $0\leq m\leq p$
$$
A_i={\bf F}_j^{(p,m)},
$$
and by a similar argument
$$
B_i={\bf G}_j^{(p,m)}.
$$
By definition if $x\in\xi_i\cap\zeta_j$, 
$$
c_x^{i,j}=c_x(u_i\xi_i,v_j\zeta_j).
$$
We now observe that,
\begin{enumerate}
\item if $i=j.(p+1)$, then $u_i\xi_i=U_{j}.g_j$,
\item if $i=j.(p+1)+n$ with $1\leq m\leq p$, then $u_i\xi_i=U_{j}.\alpha_j$,
\end{enumerate}
and similarly
\begin{enumerate}
\item if $i=j.(p+1)$, then $v_i\zeta_i=V_{j}.h_j$,
\item if $i=j.(p+1)+n$ with $1\leq m\leq p$, then $v_i\zeta_i=V_{j}.\beta_j$,
\end{enumerate}

 Then the special Product Formula \eqref{eq:pf2} is a consequence of the Product Formula \eqref{eq:pf1}: indeed, thanks to the ``homotopically good position" hypothesis, many of the intersection loops  $c_x^{i,j}$ are homotopically trivial.
\end{proof}

\subsection{Bouquets and covering}

Let $\pi:S_1\to S_0$ be a finite covering. Let $$
\cF=\left(\left((g_0,\ldots,g_q),(\alpha_0,\ldots,\alpha_q)\right),\left((h_0,\ldots,h_{q'}),(\beta_0,\ldots,\beta_q)\right),V\right)
$$
be a bouquet of flowers in $S_0$ based respectively at $(x_0,\ldots,x_q)$ and $(y_0,\ldots,y_{q'})$. 
Let $\hat x_0$ be a lift of $x_0$ in $S_1$.

\begin{definition}
The bouquet of flowers in $S_1$
$$
\widehat{\cF}=\left(\left((\hat g_0,\ldots,\hat g_q),(\hat\alpha_0,\ldots,\hat\alpha_q)\right),\left((\hat h_0,\ldots,\hat h_{q'}),(\hat\beta_0,\ldots,\hat\beta_q)\right),\hat V\right)
$$
is {\em the lift of $\cF$ through $\hat x_0$} if,\begin{itemize}
\item all arcs $\hat V$, $\hat\alpha_i$ and $\hat\beta_i$ are lifts of the arcs $V$, $\alpha_i$ and $\beta_i$.
\item $\hat g_0$ is based at $\hat x_0$.
\item the closed curves $\hat g_i$ and $\hat h_j$ are the {\em primitive lifts} of 
the curves $g_i$ and $h_j$, in other words the primitive curves which are lift of positive powers of the curves $g_i$ and $h_j$.  \end{itemize}
\end{definition}

Observe that the lift of a bouquet in homotopically good position is itself a bouquet in homotopically good position.

\subsection{Finding bouquets in good position}

Let $S$ be a closed hyperbolic surface and $\tilde S$ its universal cover. Let $G=(\gamma_0,\ldots,\gamma_q)$ and $F=(\eta_0,\ldots,\eta_{q'})$ be two tuples of primitive elements of $\pi_1(S)$ such that for all $i$, $(\gamma_i$, $\gamma_{i+1})$, are pairwise coprime as well as $(\eta_i,\eta_{i+1})$, where the index $i$ lives in $\mathbb Z/q\mathbb Z$ and $\mathbb Z/q'\mathbb Z$ respectively.  Recall that we denote by $\tilde\zeta$ the axis of the element $\zeta\in\pi_1(S)$. 

\begin{definition}\label{def:GPH}
We say $G$ and $F$ satisfy the {\em Good Position Hypothesis} if there exists a metric ball $B$ in $\tilde S$ such that
\begin{enumerate}
\item \label{GP1} for all $i$ and $j$ so that $\gamma_i$ and $\eta_j$ are coprime,
\begin{equation}
\tilde\gamma_i\cap\tilde\eta_j\subset B\label{hyp:gp3}
\end{equation}
\item \label{GP2} for all $\xi\in\pi_1(S)\setminus\{ 1\}$ we have
\begin{equation}
B\cap\xi(B)=\emptyset. \label{hyp:gp2}
\end{equation} 
\item \label{GP3}for all $\zeta\in \{\gamma_0,\ldots,\gamma_q,\eta_0,\ldots,\eta_{q'}\}$, for all $\xi\in\pi_1(S)\setminus\langle\zeta\rangle$ we have
\begin{equation}
B\cap\xi(\tilde\zeta)=\emptyset. \label{hyp:gp4}
\end{equation}
\item \label{GP4}for all $\zeta\in \{\gamma_0,\ldots,\gamma_q,\eta_0,\ldots,\eta_{q'}\}$, for all $\xi\in\pi_1(S)\setminus\langle\zeta\rangle$ we have
\begin{equation}
\tilde\zeta\cap\xi(\tilde\zeta)=\emptyset. \label{hyp:gp5}
\end{equation}
In other words, the closed geodesic corresponding to $\zeta$ is embedded.
\end{enumerate}
\end{definition}

Then, we have the following result
\begin{proposition}\label{SettingArcGoodPosition}
With the notations above, assume that $G$, $F$ and $\pi_1(S)$ satisfy the Good Position Hypothesis, then there exist two bouquets $\cF_L$ and $\cF_R$ in $S$ in a homotopically good position, both representing $(G,F)$ such that furthermore 
\begin{eqnarray}
\frac{1}{2}({\rm f}_{i,j}({\cF_L})+{\rm f}_{i,j}(\cF_R))&=&[\gamma_i^-\gamma_i^+,\eta_j^-\eta_j^+],\label{eq:f}\\
\frac{1}{2}({\rm n}_{i,j}(\cF_L)+{\rm n}_{i,j}(\cF_R))&=&[\gamma_i^-\gamma_{i+1}^-,\eta_j^-\eta_j^+],\label{eq:m}\\
\frac{1}{2}({\rm m}_{i,j}(\cF_L)+{\rm m}_{i,j}(\cF_R))&=&[\gamma_i^-\gamma_i^+,\eta_j^-\eta_{j+1}^-],\label{eq:n}\\
\frac{1}{2}({\rm q}_{i,j}(\cF_L)+{\rm q}_{i,j}(\cF_R))&=&[\gamma_i^-\gamma_{i+1}^-,\eta_{j}^-\eta_{j+1}^-].\label{eq:q}
\end{eqnarray}
\end{proposition}

\begin{proof} Let $G$ and $F$ be as above and $B$ be a metric ball in $\tilde S$ satisfying the assumptions \eqref{hyp:gp3}, \eqref{hyp:gp2} and \eqref{hyp:gp4}. We subdivide the proof in several steps. We denote by $\pi$ the projection from $\tilde S$ to $S$.

\vskip 0.2truecm
\par\noindent{\sc Step 1: Construction of the bouquet in good position}

Let $\tilde\gamma_i$ be the axis of $\gamma_i$, let $\epsilon$ be some constant that we shall choose later to be very small and $\tilde\eta^\epsilon_j$ be a curve (with constant geodesic curvature) at distance $\epsilon$ of the axis $\tilde\eta_j$ of $\eta_j$ (Notice that we have two such curves, for the moment we choose arbitrarily one of them). We choose $\epsilon$ small enough so that Assertions \eqref{hyp:gp3} and \eqref{hyp:gp4} still hold when $\tilde\eta_j$ are replaced by $\tilde\eta^\epsilon_j$. 

For every $i$, let $x_i\in \tilde\gamma_i\cap B$ so that
$$
\tilde\gamma_i\cap B\subset [x_i,\gamma_i^+[,
$$
similarly, let $y_j\in \tilde\eta^\epsilon_j$ so that
$$
\tilde\eta^\epsilon_i\cap B \subset [y_i,\eta_i^+[_\epsilon,
$$
where $[a,b]_\epsilon$ denote an arc joining $a$ to $b$ along a curve at a distance $\epsilon$ to a geodesic (see Figure \ref{fig:BouleB}).

\begin{figure}[h] \begin{center}
  \includegraphics[width=0.6\textwidth]{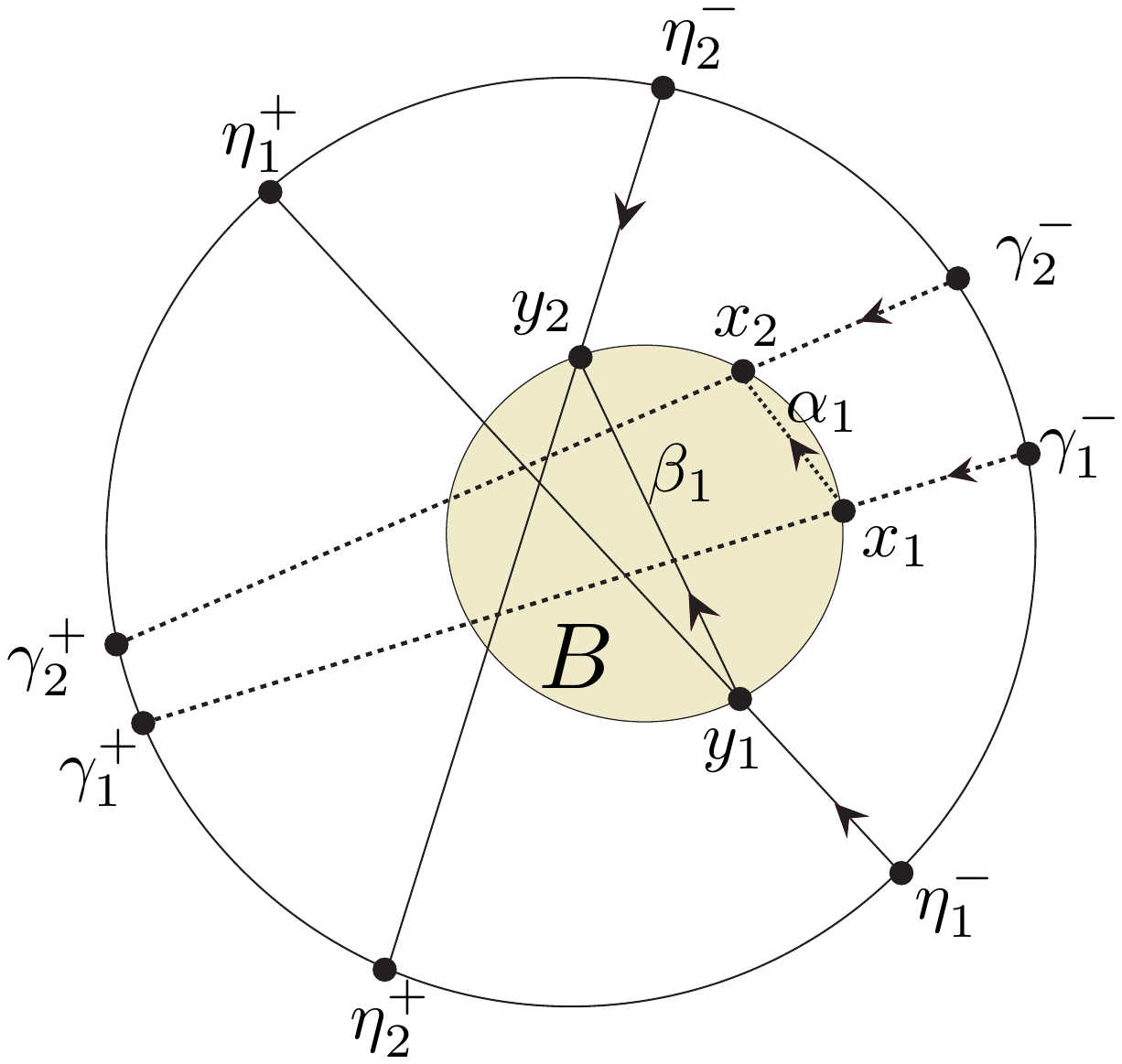}
 \end{center}
 \caption{Finding a bouquet in good position}
 \label{fig:BouleB}
\end{figure}
We now consider geodesics arcs $\tilde\alpha_i$, $\tilde\beta_j$ and $\tilde V$ in $\tilde S$ joining respectively $x_i$ to $x_{i+1}$, $y_j$ to $y_{j+1}$ and $x_0$ to $y_0$. We furthermore choose $B$ (and $\epsilon$) so that all the arcs $\tilde\alpha_i$, $\tilde\eta^\epsilon_j$ $\tilde\gamma_j$ and $\tilde\beta_j$ are transverse. 
In particular, if \begin{eqnarray}
\alpha_i=\pi(\tilde\alpha_i),&&\,\beta_i=\pi(\tilde\beta_i),\,V=\pi(\tilde V),\cr
\tilde g_i=[x_i,\gamma_(x_i)]),&& \tilde h_j=[y_j, \eta_j(y_j)]_\epsilon,\cr
 g_i=\pi([x_i,\gamma_(x_i)]),&& h_j=\pi([y_j, \eta_j(y_j)]_\epsilon),
\end{eqnarray}
 then
 $$\cF=\left(\left((g_0,\ldots,g_q),(\alpha_0,\ldots,\alpha_q)\right),\left((h_0,\ldots,h_{q'}),(\beta_0,\ldots,\beta_{q'})\right), V\right)$$ is in good position. Observe furthermore that $\cF$ represents $(G,F)$.

\vskip 0.2truecm
\par\noindent{\sc Step 2: homotopically good position}

Let us now prove that $\cF$ is in a homotopically good position. Let as usual
\begin{equation}
U_i=\alpha_0\ldots\alpha_{i-1},\ \ \ 
V_j=V.\beta_0\ldots\beta_{j-1},
\end{equation}
and \begin{equation}
\tilde U_i=\tilde\alpha_0\ldots\tilde\alpha_{i-1},\ \ \ 
\tilde V_j=\tilde V.\tilde\beta_0\ldots\tilde\beta_{j-1}.
\end{equation}
Then $\tilde U_i$ and $\tilde V_j$ are the lifts of $U_i$ and $V_j$ respectively starting from $x_0$ and $y_0$ an ending respectively in $x_i$ and $y_j$.

Observe that all the arcs $\tilde\alpha_k$, $\tilde\beta_l$ and $\tilde V$ lie in $B$.
Thus so do the paths $\tilde U_i$ and $\tilde V_j$.

Let $W_i$ be equal to $\alpha_i$ or $g_i$. Let $\widehat W_j$ be equal to $\beta_j$ or $h_j$. Let us fix from now $x\in W_i\cap\widehat W_j$. 
Let us introduce some notation,

\begin{itemize}
\item Let $a$ be the path along $W_i$ from $\pi(x_i)$ to $x$, and $\widehat a$ be the path along $\widehat W_j$ from $\pi(y_j)$ point to $x$.
\item Let $b$ and $\widehat b$ the lifts of $a$ and $\widehat a$ in $\tilde S$ starting respectively from $x_i$ and $x_j$,
\item Let finally $z$ and $\widehat z$ be the endpoints of $b$ and $\widehat b$ and $\zeta\in\pi_1(S)$ so that $z=\zeta(\widehat z)$.
\end{itemize}
By construction $\zeta$ is conjugated to the intersection loop $c_x(U_iW_i,U_j\widehat W_j)$. 

Let us now consider the various possibility about the position of $z$ and $\widehat z$.
\begin{enumerate}
\item $W_i=\alpha_i$, then $b\subset\tilde\alpha_i$ and thus $z$ belongs to $B$.
\item $W_i=g_i$, then $z\in[x_i,\gamma_i(x_i)[\subset [x_i,\gamma_i^+[$,
\item $\widehat W_j=\beta_j$, then, symmetrically, $\widehat z=\zeta(z)$ belongs to $B$.
\item $\widehat W_j=h_i$, then, symmetrically, $\widehat z\in[y_j,\eta_j^\epsilon(x_j)[_\epsilon \subset [y_i,\eta_j^+[_\epsilon$ (where the intervals are subsets of $\tilde\eta_j^\epsilon$). 
\end{enumerate}
Our goal is now to prove that $\zeta=1$ unless, maybe, $W_i=g_i$ and $\widehat W_j=h_j$.
\begin{enumerate}
\item $W_i=\alpha_i$ and  $\widehat W_j=\beta_j$ then by Assertions (1) and (3) above, both $z$ and $\zeta(z)$ belong to $B$, and thus by Assertion \eqref{hyp:gp2}, $\zeta=1$.
\item $W_i=\alpha_i$ and  $\widehat W_j=h_j$, then by Assertions (1) and (4), $\zeta(z)\in\tilde\eta^\epsilon_j$ and $z\in B$. Thus $\zeta^{-1}(\tilde\eta_i^\epsilon)\cap B\not=\emptyset$. Then by Hypothesis \eqref{hyp:gp4}, $
\zeta\in\langle\eta_j\rangle$. In particular $z\in B \cap\tilde\eta_j^\epsilon$ and thus we have 
\begin{eqnarray}
z\in [y_j,\eta_j(y_j)[_\epsilon,\label{proo:zep1}
\end{eqnarray}
Recall that from Assertion (4)
\begin{eqnarray}
\zeta(z)=\widehat z \in [y_j,\eta_i(y_j)[_\epsilon.\label{proo:zep2}
\end{eqnarray}
Since $\eta_j$ is primitive and $\zeta\in\langle\eta_j\rangle$, we obtain from Assertions \eqref{proo:zep1} and \eqref{proo:zep2} that $\zeta=1$.
\item A symmetric argument proves that when $W_i=g_i$ and  $\widehat W_j=\beta_j$, then $\zeta=1$.

\end{enumerate}
This finishes the proof that $\cF$ is in a homotopically good position.
\vskip 0.2truecm
\par\noindent{\sc Step 3: computation of the intersection numbers}

Recall that for each (oriented) axis $\tilde\eta_j$ we had two choices of curves at distance $\epsilon$. Let us denote by $\tilde\eta_j^L$ -- respectively $\tilde\eta_j^R$ -- the curve on the left --respectively on the right-- to $\tilde\eta_j$. 
Let then $\cF^L$ and $\cF^R$ the corresponding collections of arcs.

We have proved that both $\cF^L$ and $\cF^R$ are in homotopically good position. Let us now compute the intersection numbers. We will do that step by step.

We shall repeat the following observation several times: let $g$ and $h$ be two curves in $S$ passing through a point $x_0$, intersecting transversely  in a finite number of points $x_1,\ldots x_n$. Let $\tilde g$ and $\tilde h$ be the lift of these curves in $\tilde S$ passing through a point $\tilde x_0$. Then the projection realises a bijection between the set of those $x_i$ whose intersection loop is trivial, and the intersections points of $\tilde g$ and $\tilde h$.

In particular
\begin{eqnarray}
\sum_{x\in g\cap h\mid c_x(g,h)=1}\ii(x)&=&\sum_{z\in \tilde g\cap \tilde h}\ii(z)\label{egalinter}
\end{eqnarray}

\vskip 0.5 truecm
\noindent{\em Proof of Equation \eqref{eq:f}:} If $\gamma_i$ and $\eta_j$ are coprime, by Formula \eqref{egalinter} and since two geodesics have at most one intersection point, we have that
$$
{\rm f}_{i,j}=[\gamma_i^-\gamma_i^+,\eta_j^-\eta_j^+].
$$
If $\gamma_i$ and $\eta_j$ are not coprime, since $g_i$ is embedded by Assumption \eqref{hyp:gp5}, it follows that
$$
\ii(g_i,h_j)=0=[\gamma_i^-\gamma_i^+,\eta_j^-\eta_j^+].
$$
Thus in both cases
$$
{\rm f_{i,j}}(\cF^L)={\rm f_{i,j}}(\cF^R)=[\gamma_i^-\gamma_i^+,\eta_j^-\eta_j^+].
$$
\vskip 0.5 truecm
\noindent{\em Proof of Equation \eqref{eq:m}:}
Since all the corresponding intersection loops are trivial, we see that
$$
\ii(g_i,\beta_j)=\ii(\tilde\gamma_i,\tilde\beta_j).
$$
We know that $\tilde\beta_j\subset B$. To simplify, let us first consider the case when $\gamma_i$ and $\eta_{k}$ are coprime for $k=j,j+1$. Then $\tilde\gamma_i\cap\tilde\eta^\epsilon_k\subset B$ and thus
$$
\ii(\tilde\gamma_i,]\eta_k^-,y_k]_\epsilon)=0.
$$
It follows then that
$$
\ii(g_i,\beta_j)=\ii(\tilde\gamma_i,]\eta_j^-,y_j]\cup\tilde\beta_j\cup]y_{j+1},\eta_{j+1}^-[)=[\gamma_i^-\gamma_i^+,\eta_j^-\eta_{j+1}^-].
$$
We illustrate that situation in Figure \ref{fig:case1}

\begin{figure}
 \centering
 \subfloat[non-coprime elements]{\label{fig:case1}\includegraphics[width=0.5\textwidth]{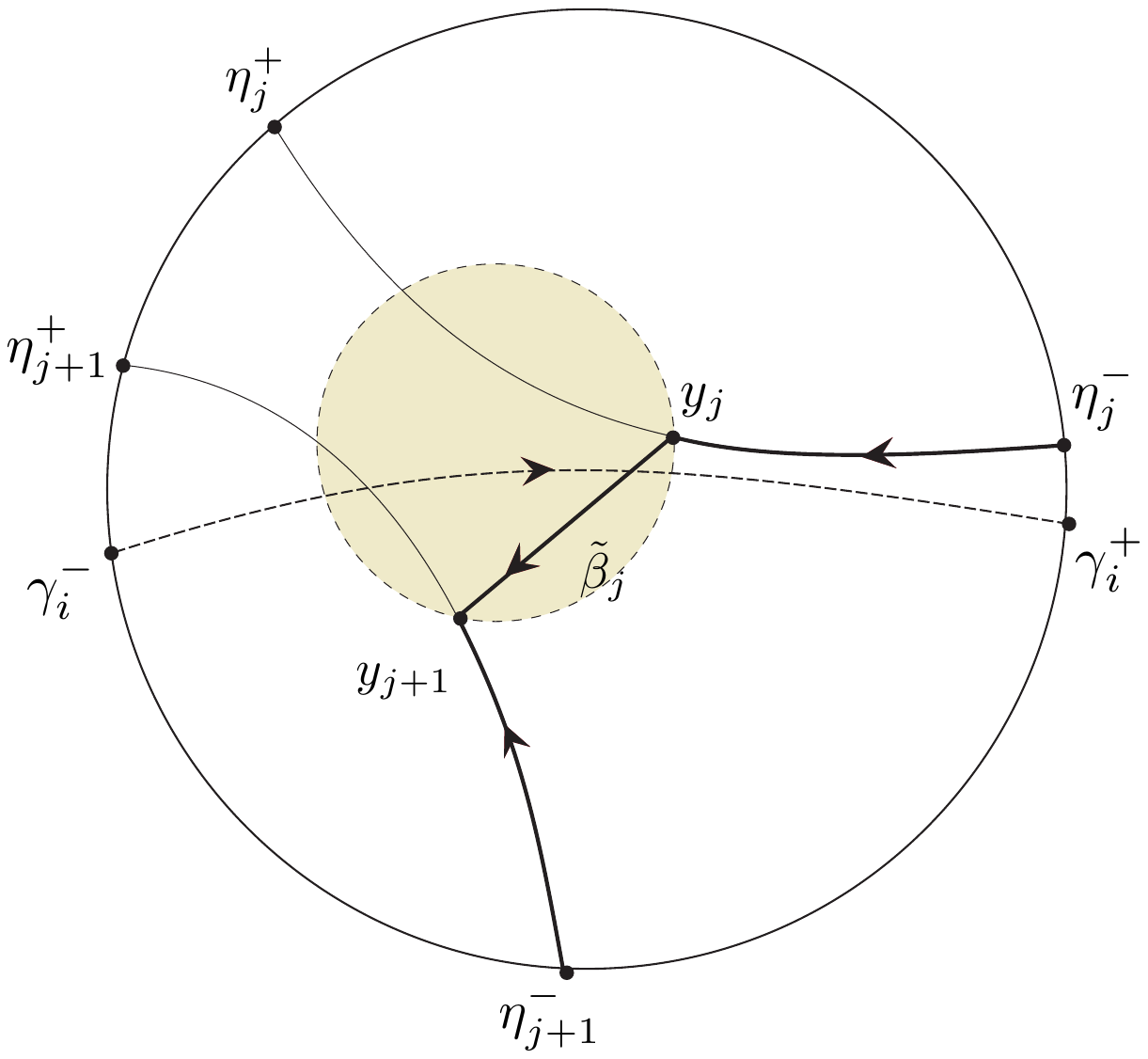}}        
 \subfloat[Coprime elements]{\label{fig:case2}\includegraphics[width=0.5\textwidth]{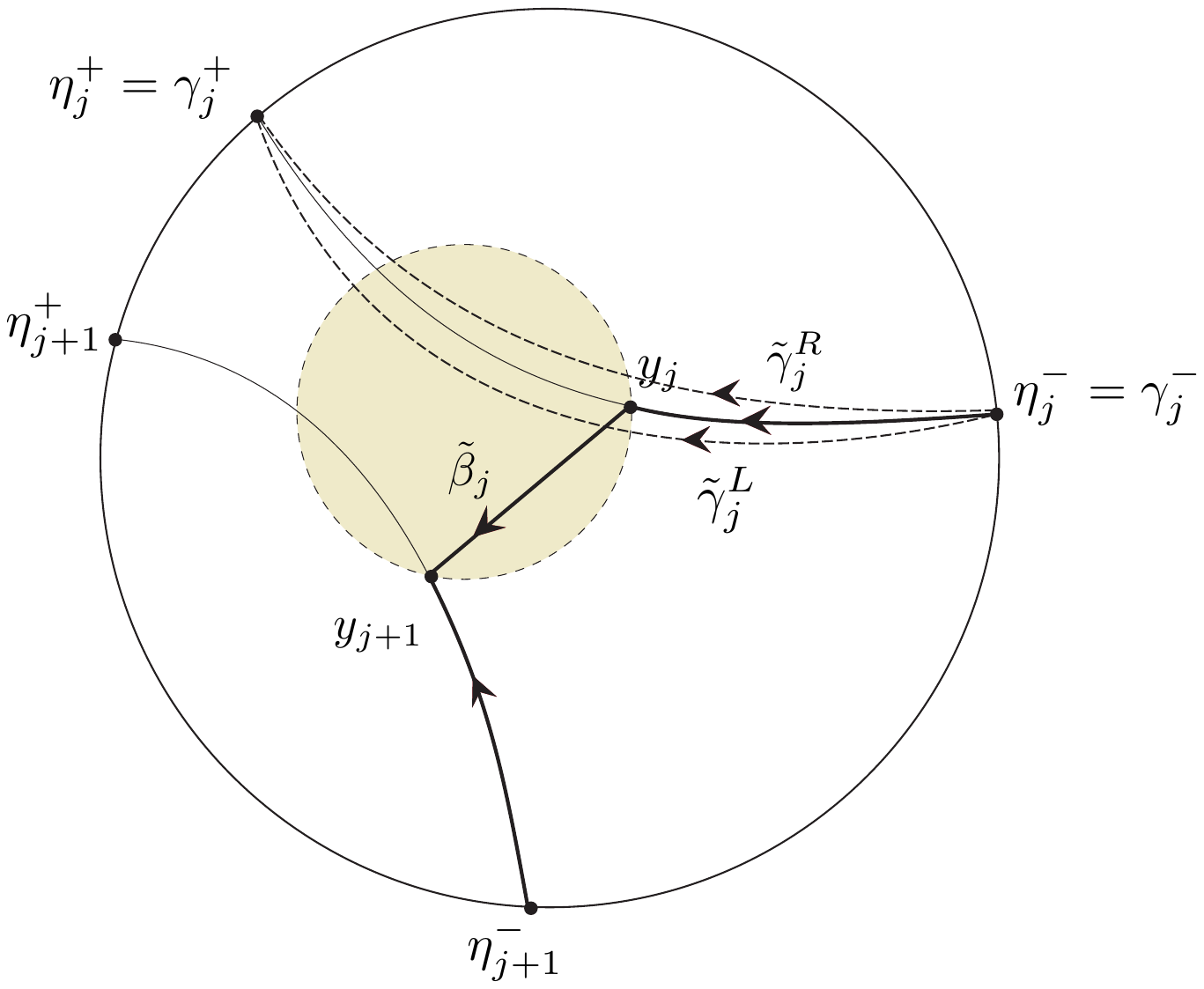}}
 \caption{Intersection computations}
 \label{fig:Cases}
\end{figure}

Let us move to the remaining cases.
The purpose of taking the ``left and right perturbations" of $\tilde\eta_j$ is to take care of the situation when $\eta_j$ (or $\eta_{j+1}$) and $\gamma_i$ are not coprime. So let us assume now that $\tilde\eta_j=\tilde\gamma_i$ (the case when $\tilde\eta_{j+1}=\tilde\gamma_i$ is symmetric).

Then in this case assume that $\eta_{j+1}^-$ is on the left of $\tilde\eta_j$ and $\tilde\gamma_i$ has the same orientation as $\tilde\eta_j$ (the other cases being symmetric). It then follows that
\begin{eqnarray}
{\rm m}_{i,j}({\cF}^L)=\ii(\tilde\gamma_i,\tilde\beta_j^L)&=&0,\\
{\rm m}_{i,j}({\cF}^R)= \ii(\tilde\gamma_i,\tilde\beta_j^R)&=&1.
\end{eqnarray}
It follows that
$$
\frac{1}{2}({\rm m}_{i,j}({\cF}^L)+{\rm m}_{i,j}({\cF}^R))
=\frac{1}{2}=[\gamma_i^-\gamma_i^+,\eta_j^-\eta_{j+1}^-].$$
We illustrate that case in Figure \ref{fig:case2}.
This finishes the proof of Equation \eqref{eq:m}.

\vskip 0.5 truecm
\noindent{\em Proof of Equation \eqref{eq:n} and \eqref{eq:q}:} The proof follows from the same ideas as the previous ones.
\vskip 0.5 truecm
\end{proof}

\section{Asymptotics}
This section is the main computational core of this article. Our goal is to compute asymptotic product formulas, namely understand the behaviour of the special product formula when the repetition in the arcs becomes infinite. This allows us to describe the limit of certain Wilson loops as elementary functions -- see Proposition \ref{AsymBP}.

The goal of this section is to obtain Corollary \ref{apcomp} which is an asymptotic product formula for the Goldman bracket of elementary functions.

We first need some facts about vanishing sequences 
\subsection{Properties of vanishing sequences}

In this paragraph, we shall be given a vanishing sequence $\{\Gamma_p\}_{p\in\mathbb N}$ of finite index subgroups of $\pi_1(S)$. We need some notation and definition 
\begin{itemize}
\item Let $\gir_p$ be the function defined on $\Hn$ by $$\gir_p(\rho)\defeq\gir\left(\left.\rho\right\vert_{\Gamma_p}\right).$$
\item For any positive integer $p$ and primitive element $\xi$ in $\Gamma_0$, let $\xi(p)$ be the positive integer so that
$$
\langle\xi ^{\xi(p)}\rangle=\langle\xi\rangle\cap\Gamma_p.
$$
We write $\xi_p=\xi^{\xi(p)}$ and we denote the associated closed geodesic by $\tilde\xi_p$. 
\end{itemize}
\begin{definition}{\sc [$N$-nice covering]}
 Let $\gamma$ and $\eta$ be primitive coprime elements of $\Gamma_0=\pi_1(S)$. Let $N$ be a positive integer. We say that $\Gamma_p$ is {\em $N$-nice} with respect to $\gamma$ and $\eta$, if the intersection loop $c_x(\tilde\gamma_p,\tilde\eta_p)$ is either trivial or satisfies 
$$
\pi_p \left(c_x(\tilde\gamma_p,\tilde\eta_p)\right)=\gamma^{k_1}.\eta^{-k_2},
$$
where $k_1$ and $k_2$ satisfy $$\gamma(p)-N> k_1 >N \hbox{ and }\eta(p)-N> k_2 >N.$$ \end{definition}

We need the following properties of vanishing sequences which we summarise in the next proposition. \begin{proposition}\label{GPH0}
Let $\{\Gamma_p\}_{p\in\mathbb N}$ be a vanishing sequence of finite index subgroups of $\pi_1(S)$ and $\{S_p\}_{p\in\mathbb N}$ the corresponding sequence of coverings so that $\pi_1(S_p)=\Gamma_p$. Then
\begin{enumerate}
\item when $p$ goes to infinity, $\gir_p$ converge uniformly to 0 on every compact of $\Hn$,
\item for any primitive coprime elements $\gamma$ and $\eta$, for all $N$, there exists $p_0$ so that for every $p>p_0$, $\Gamma_p$ is $N$-nice with respect to $\gamma$ and $\eta$.
\item Let $G=(\gamma_0,\ldots,\gamma_p)$ and $F=(\eta_0,\ldots,\eta_q)$ be tuples of primitive elements of $\grf\setminus\{1\}$ such that 
$(\gamma_i,\gamma_{i+1})$ as well as $(\eta_j,\eta_{j+1})$ are pairwise coprime. Then for $q$ large enough, $G$ and $F$ satisfy the Good Position Hypothesis \ref{def:GPH} as elements of $\pi_1(S_q)$. 
\end{enumerate}
\end{proposition}

\subsubsection{Proof of Proposition \ref{GPH0}}
Proposition \ref{GPH0} will the concatenation of Propositions \ref{VanishGPH1}, \ref{VanishGPH3} and \ref{VanishGPH2} proved thereafter.  Proposition \ref{VanishGPH2a} is an intermediate step in proving Proposition \ref{VanishGPH2}. We fix in this paragraph a vanishing sequence $\{\Gamma_p\}_{p\in\mathbb N}$.

Remember that we identify primitive elements in $\grf$ and in any of its finite index subgroup.

\begin{proposition}\label{VanishGPH1}
When $p$ goes to infinity, $\gir_p$ converge uniformly to 0 on every compact of $\Hn$.
\end{proposition}

\begin{proof}  For all positive number $K$ and compact $C$ in $\mathsf H(n,S)$, let consider the following subset of $\grf$
$$
Z_K\defeq \left\{\gamma\in\grf\setminus\{\id\}\mid \exists \rho\in C, \ \ \left\vert\frac{\lambda_2(\rho(\gamma))}{\lambda_1(\rho(\gamma))}\right\vert>K\right\}.
$$
By Proposition \ref{girth}, the set of conjugacy classes in $Z_K$ is a finite set. Let $Z^0_K$ be a finite set in $\grf$ of representatives of the conjugacy classes of $Z_K$.
From the definition of vanishing sequences, it follows that there exists $p_0$ that that for all $p>p_0$, we have
$$
Z^0_K\cap\Gamma_{p}=\emptyset.
$$
Since $\Gamma_p$ is normal, it follows that
$$
Z_K\cap\Gamma_{p}=\emptyset.
$$
Then by definition, the girth of any representation in $C$ restricted to $\Gamma_p$ is smaller than $K$.
Thus the family of functions $\gir_p$ converges uniformly to zero on $C$, when $p$ goes to $\infty$.
\end{proof}
The following proposition is well known.

\begin{proposition}\label{VanishGPH2a} Let $\gamma$ be an element of $\Gamma_0$. Then 
there exists $p_0$ such that for all $p>p_0$,
 the geodesics $\tilde\gamma_p$ is simple.\end{proposition} 
 
 \begin{proof}
Let 
$$
\hat A_\gamma\defeq \{\xi\in\Gamma_0\mid \xi(\tilde\gamma)\cap\tilde\gamma\not=\emptyset.\}\subset\Gamma_0/\langle \gamma\rangle.
$$
Observe that $\hat A_\gamma$ is invariant by right multiplication by $\gamma$ and that its projection in $\Gamma_0/\langle \gamma\rangle $ is a finite set. Thus there exists $p_0$ so that for every $p>p_0$,
$$
A_\gamma\cap\Gamma_p\subset \langle\gamma\rangle.
$$
This implies that the projection of $\tilde\gamma$ in $S_p$ is a simple closed geodesic, indeed the existence of a self intersection point implies the existence of an element $\xi$ in $\Gamma_p$ so that $\xi(\tilde\gamma)\cap\tilde\gamma\not=\emptyset$.
\end{proof}

We finally need.

\begin{proposition}\label{VanishGPH3}
Let $\gamma$ and $\eta$ two coprime primitive elements of $\Gamma_0=\pi_1(S)$. Let $N$ be a positive integer. 

Then there exists $p_0$ such that for all $p>p_0$, $\Gamma_p$ is $N$-nice with respect to $\gamma$ and $\eta$. \end{proposition} 
 \begin{proof} We assume using the previous proposition that $\tilde\gamma_p$ and $\tilde\eta_p$ are simple.
\vskip 0.2truecm
\par\noindent{\sc Step 1} We shall prove the following assertion
\vskip 0.2truecm
\par{\em For any $N>0$, there exists $p_0$ such that for any $p>p_0$, for any integers $k$ such that $0< k \leq N$ and for any $m$ then 
$$
\gamma^{k}\eta^{m}\not\in\Gamma_p\hbox{ and }\gamma^{m}\eta^{-k}\not\in\Gamma_p.
$$
}
This is an immediate application of Proposition \ref{VSbis}. Let $H\defeq \{\gamma^k\mid 0<k\leq N\}$. Since $\gamma$ and $\eta$ are coprime then $H\cap\langle\eta\rangle=\emptyset$. Using Proposition, \ref{VSbis}, we get that there exists $p_0$ so that for all $p>p_0$, $$H\cap\left(\Gamma_p.\langle\eta\rangle\right)=\emptyset.$$ In other words, for all $n$ and $k$ so that $0<k\leq N$, 
$$
\gamma^k.\eta^n\not\in\Gamma_p.
$$
A symmetric argument concludes the proof.
\vskip 0.2truecm
\par\noindent{\sc Step 2} We now prove.
\vskip 0.2truecm
\par{\em If $x\in\gamma_p\cap\eta_p$, then there exists positive integers $k_1$ and $k_2$ such that the intersection loop $c_p(x)\defeq c_x(\tilde\gamma_p,\tilde\eta_p)$ satisfies
$$
\pi_p \left(c_p(x)\right)=\gamma^{k_2}.\eta^{-k_1},
$$
where the equality is as homotopy classes in $S_0=S$}
\vskip 0.2truecm
We can as well assume (using the first step and a shift in $p$) that the projection of the axis of $\gamma$ and $\eta$ are simple geodesics in $S_0$. Let also
$$
A_p\defeq \{\xi\in\Gamma_p\mid \xi(\tilde\eta)\cap\tilde\gamma\not=\emptyset\}\subset \Gamma_p.
$$
Observe that $A_p$ is invariant by left multiplication by $\gamma$ and right multiplication by $\eta$. Let $\hat A_p$ be the projection of $A_p$ in $\langle \eta_p\rangle\backslash \Gamma_p/\langle \eta_p\rangle$. Observe also that we have a bijection from $\hat A_p$ to 
$$
I_p\defeq \pi_p(\tilde\gamma)\cap\pi_p(\tilde\eta)\subset S_p
$$
given by 
$$
\langle\gamma\rangle.\xi.\langle\eta\rangle\to\pi_{p}(\xi(\tilde\eta)\cap\tilde\gamma).
$$
In particular $\hat A_p$ is finite since $I_p$ is finite. Moreover, if $x$ in $I_p$ comes in this procedure from an element $a$ in $A_p$, then $a$ represents the intersection loop of $x$.

Since $\hat A_0$ is finite, using the double coset separability property, there exists $p_0$ such that for all $p>p_0$, we have
$$
A_0\cap \Gamma_p\subset \langle\gamma\rangle\langle\eta\rangle,
$$
Since 
$
 A_p\subset A_0\cap\Gamma_p,
$
It follows that the projection in $S_0$ of any intersection loop $c_x(\gamma_p,\eta_p)$ is homotopic to $\gamma^n.\eta^{-m}$ with $n$ and $m$ positive integers.

\vskip 0.2truecm
\par\noindent{\sc Conclusion of the proof} The Proposition follows at once from the two steps of the proof.
\end{proof}

\begin{proposition}\label{VanishGPH2}
 Let $G=(\gamma_0,\ldots,\gamma_p)$ and $F=(\eta_0,\ldots,\eta_q)$ be tuples of primitive elements of $\grf\setminus\{1\}$ such that 
$(\gamma_i,\gamma_{i+1})$ as well as $(\eta_j,\eta_{j+1})$ are pairwise coprime. Then for $m$ large enough, $G$ and $F$ satisfy the Good Position Hypothesis \ref{def:GPH} as elements of $\pi_1(S_m)$. 
\end{proposition} 

\begin{proof}
Let us check the four conditions of the Good Position Hypothesis.
Let $G=(\gamma_0,\ldots,\gamma_p)$ and $F=(\eta_0,\ldots,\eta_q)$ be primitive elements of $\grf\setminus\{1\}$ such that 
$(\gamma_i,\gamma_{i+1})$ as well as $(\eta_j,\eta_{j+1})$ are pairwise coprime. 

\begin{enumerate}
\item 
Let $B\subset \tilde S$ be a ball  containing all the intersections         $\tilde\gamma_i\cap \tilde\eta_j$ when $\gamma_i$ and $\eta_j$ are coprime. Thus Condition \eqref{GP1} of the Good Position Hypothesis is satisfied.
\item Let 
$$
F\defeq \{\xi\in\grf \mid B\cap\xi(B)\not=\emptyset\}.
$$
The set $F$ is finite. Thus, by Proposition \ref{VSbis} (applied to $\gamma=\eta=\id$), there exists $p_0$, such that for all $p>p_0$, we have
$$
F\cap \Gamma_p=\{\id\}.
$$
Thus Condition \eqref{GP2} of the Good Position Hypothesis is satisfied.

\item Next, for every $\zeta\in \{\gamma_0,\ldots\gamma_p,\eta_0,\ldots,\eta_q\}$ the set
$$
H_\zeta\defeq \{\xi\in \Gamma/\langle\zeta\rangle\mid \xi(\tilde\zeta)\cap B\not=\emptyset\}.
$$
finite. Thus by Proposition \ref{VSbis} (applied to $\gamma=\id$, $\eta=\zeta$), there exists $p_0$, such that for all $p>p_0$, we have
$$
H_\zeta\langle\zeta\rangle\cap \Gamma_p=\langle\zeta\rangle.
$$
Thus Condition \eqref{GP3} of the Good Position Hypothesis is satisfied.

\item Finally Condition \eqref{GP4} of the Good Position Hypothesis is satisfied for $p$ large enough by Proposition \ref{VanishGPH2a}.

\end{enumerate}

 \end{proof}

\subsection{Asymptotic product formula for Wilson loops}

In all this paragraph, we shall be given a finite index subgroup $\Gamma_k$ of $\Gamma_0=\pi_1(S)$, corresponding to a covering $S_k\to S_0=S$. Then, if $\rho$ is a Hitchin representation of $\pi_1(S)$ in $\sln$, $\rho_k$ will denote the restriction of $\rho$ to $\Gamma_k$.

Let $(\gamma_0,\ldots,\gamma_q)$ and $(\eta_0,\ldots,\eta_{q'})$ be two tuples of primitive elements of $\pi_1(S)$. We assume that $(\gamma_i,\gamma_{i+1})$ as well as $(\eta_j,\eta_{j+1})$ are all pairwise coprime.

Let then $\hat\gamma_i$ and $\hat\eta_i$ be the representatives of $\gamma_i$ and $\eta_i$ in $\Gamma_k$, and 

\begin{eqnarray}
{\bf F}^{(p)}=\hat\gamma_{1}^p\ldots\hat\gamma_{q}^p, & & {\bf G}^{(p)}=\hat\eta_{1}^p\ldots\hat\eta_{q'}^p.
\end{eqnarray}

We want to understand the asymptotics when $p$ goes to infinity of the following function
$$
{B^k_p}(\gamma_0,\ldots,\gamma_q;\eta_0,\ldots,\eta_{q'}):\, \ms H (n,S_k)\to \mathbb R,
$$
defined by
\begin{equation}
{B^k_p}(\gamma_0,\ldots,\gamma_q;\eta_0,\ldots,\eta_{q'})\defeq \frac{\ww(\{{\bf G}^{(p)},{\bf F}^{(p)}\}_{S_k})}{{\ww({\bf G}^{(p)})\ww({\bf F}^{(p)})}}\label{eq:bpdef}
\end{equation}
Let then
\begin{align}
{\rm f}_{i,j}=[\gamma_i^-\gamma_i^+,\eta_j^-\eta_j^+],& &
{\rm n}_{i,j}=[\gamma_i^-\gamma_{i+1}^-,\eta_j^-\eta_j^+],\cr
{\rm m}_{i,j}=[\gamma_i^-\gamma_i^+,\eta_j^-\eta_{j+1}^-],& &
{\rm q}_{i,j}=[\gamma_i^-\gamma_{i+1}^-,\eta_{j}^-\eta_{j+1}^-].\label{eq:fmnq}
\end{align}
The following paragraph is devoted to the proof of the following Proposition
\begin{proposition}\label{AsymBP0}{\sc [Asymptotic product formula]}
For every compact set $U$ in $\Hn$, for every positive integer $N$, for $k$ large enough, we have for every $\rho$ in $U$
\begin{multline}
B^k_p(\gamma_0,\ldots,\gamma_q,\eta_0,\ldots,\eta_{q'})(\rho)
=\sum_{\substack{0\leq i\leq q\\ 0\leq j\leq q'}}\Bigg((p-1)^2 {\rm f}_{i,j}\tw(\gamma_i,\eta_j)\\
+(p-1)\left(\frac{\tw(\gamma_{i+1},\gamma_{i},\eta_j)}
{\tw(\gamma_i,\gamma_{i+1})}\left({\rm n}_{i,j}+{\rm f}_{i+1,j}
\right)
+\frac{\tw(\gamma_i,\eta_{j+1},\eta_j)}
{\tw(\eta_j,\eta_{j+1})}\left({\rm m}_{i,j}+{\rm f}_{i,j+1}
\right)\right)
\\
+\frac{\tw(\gamma_{i+1},\gamma_i,\eta_{j+1},\eta_j)}{\tw(\gamma_{i+1},\gamma_i)\tw(\eta_j,\eta_{j+1})}\left({\rm q}_{i,j}+{\rm n}_{i,j+1}+{\rm m}_{i+1,j}+{\rm f}_{i+1,j+1}\right)\Bigg)\\
+p^2 {\rm R}_{i,j}\tw(\gamma_i,\eta_j)+K \cdotp\left(\gir_k(\rho)+\gir_0(\rho)^N\right),
\end{multline}

where
 \begin{itemize}
\item $K$ is bounded on $U$, 
\item  $\gir_k(\rho)=\gir\left(\left.\rho\right\vert_{\Gamma_k}\right)$ where $\gir(\rho)$ is the girth of $\rho$ -- see Definition \ref{def:girth}.
\item the integers ${\mathrm f}_{i,j}$, ${\mathrm m}_{i,j}$, ${\mathrm n}_{i,j}$ and ${\mathrm q}_{i,j}$ are defined in Equations \eqref{eq:fmnq} .
\item ${\rm R}_{i,j}$ is an integer that only depends on $\gamma_i$ and $\eta_j$. 
\end{itemize}
\end{proposition}
We will use bouquets to express this asymptotics using our product formula for bouquets.
\subsubsection{Preliminary asymptotics}\label{par:prelimas}
Let $\rho$ be a representation of $\Gamma_0=\pi_1(S)$. For any $k$, let $\rho_k\defeq \left.\rho\right\vert_{\Gamma_k}$. Let $\gamma_0,\ldots,\gamma_q$ and $\eta_0,\ldots,\eta_{q'}$ be primitive elements of $\Gamma_0$ and $\wh\gamma_0,\ldots,\wh\gamma_q$ and $\wh\eta_0,\ldots,\wh\eta_{q'}$ be the corresponding elements in $\Gamma_k$, so that
\begin{eqnarray}
\wh{\gamma}_i=\gamma_i^{Q_i}, & &\wh{\eta}_j=\eta_j^{P_j}, 
\end{eqnarray}
where $Q_i$ and $P_j$ are positive integers.
In this proof, $K$, $K_0$,$K_1$, \ldots will be the generic symbol for a function of $\rho$ bounded by a continuous function that only depends on the relative position of the eigenvectors of $\rho(\gamma_i)$ and $\rho(\eta_i)$ and does not depend on $k$. Let us define
\begin{eqnarray}
\wh{\ms g}_i=\rho_k(\wh\gamma_i), & &
\wh{\ms h}_i=\rho_k(\wh\eta_i),\\
{\ms g}_i=\rho(\gamma_i), & &
{\ms h}_i=\rho(\eta_i),
\end{eqnarray}
and
\begin{eqnarray}
\wh{\ms F}_i^{(p,m)}&\defeq &\wh{\ms g}_i^{m}\wh{\ms g}_{i+1}^{p}\ldots\wh{\ms g}_{i-1}^{p}\wh{\ms g}_i^{p-m},\\
\wh{\ms G}_i^{(p,m)}&\defeq &\wh{\ms h}_i^{m}\wh{\ms h}_{i+1}^{p}\ldots\wh{\ms h}_{i-1}^{p}\wh{\ms h}_i^{p-m},\\
{\ms F}_i^{(p,m)}&\defeq &{\ms g}_i^{m}{\ms g}_{i+1}^{p}\ldots{\ms g}_{i-1}^{p}{\ms g}_i^{p-m},\\
{\ms G}_i^{(p,m)}&\defeq &{\ms h}_i^{m}{\ms h}_{i+1}^{p}\ldots{\ms h}_{i-1}^{p}{\ms h}_i^{p-m}.
\end{eqnarray}
We prove in this paragraph two propositions
\begin{proposition}\label{pro:fpm}
For all positive integer $p$, for all integer $m$, with $0<m<p$, we have for any $\rho$ in a compact set $U$ of $\Hn$,
\begin{eqnarray}
\frac{\wh{\ms F}_{i}^{(p,0)}}{\tr \left(\wh{\ms F}_{i}^{(p,0)}\right)}&=&
\frac{\bu{\ms g}_{i}\bu{\ms g}_{i-1}}{\tr (\bu{\ms g}_{i}\bu{\ms g}_{i-1})}+K_1.\gir_k(\rho)^p,\\
\frac{\wh{\ms F}_{i}^{(p,p)}}{\tr \left(\wh{\ms F}_{i}^{(p,0)}\right)}&=&
\frac{\bu{\ms g}_{i+1}\bu{\ms g}_{i}}{\tr (\bu{\ms g}_{i}\bu{\ms g}_{i+1})}+K_2.\gir_k(\rho)^p,\\
\frac{\wh{\ms F}_{i}^{(p,m)}}{\tr \left(\wh{\ms F}_{i}^{(p,0)}\right)}&=&\bu{\ms g}_{i}+K_3.\gir_k(\rho)^{\inf(m,p-m)},
\end{eqnarray}
where $K_i$ are locally bounded functions of $\rho$.
\end{proposition}
We recall that $\bu{\ms g}$ is the projector on the eigendirection of the highest eigenvalue of $\ms g$.
\begin{proof} Observe that for all $m$, 
$$
\tr \left(\wh{\ms F}_{i}^{(p,m)}\right)=\tr \left(\wh{\ms F}_{i}^{(p,0)}\right).
$$
We use Corollary \ref{Asym1-coro} and get that for all $p$,\begin{eqnarray}
\frac{\wh{\ms F}_{i}^{(p,0)}}{\tr \left(\wh{\ms F}_{i}^{(p,0)}\right)}&=&
\frac{\bu{\ms g}_{i}\bu{\ms g}_{i-1}}{\tr (\bu{\ms g}_{i}\bu{\ms g}_{i-1})}+K_3.\gir_k(\rho)^p,\\
\frac{\wh{\ms F}_{i}^{(p,p)}}{\tr \left(\wh{\ms F}_{i}^{(p,0)}\right)}&=&
\frac{\bu{\ms g}_{i+1}\bu{\ms g}_{i}}{\tr (\bu{\ms g}_{i}\bu{\ms g}_{i+1})}+K_4.\gir_k(\rho)^p,\\
\frac{\wh{\ms F}_{i}^{(p,m)}}{\tr \left(\wh{\ms F}_{i}^{(p,m)}\right)}&=&\bu{\ms g}_{i}+K_5.\gir_k(\rho)^{\inf(m,p-m)}, \hbox{ for }m\not\in\{0,p\}.
\end{eqnarray}
\end{proof}
We use the same notation as in the beginning of the paragraph.

\begin{proposition}\label{asymconj} Let us fix $i$ and $j$. Let
\begin{itemize}
\item $\{N_1,\ldots, N_r\}$ be a sequence of pairwise distinct integers so that $N_l\geq N$ and $Q_j-N_l\geq N$.
\item $\{M_1,\ldots, M_r\}$ be a sequence of pairwise distinct integers so that $M_l\geq N$ and $P_j-M_l\geq N$.
\end{itemize}
Then,  for any $\rho$ in a compact set $U$ in $\Hn$, for any positive integers $p$, $m$
and $m'$, we have
\begin{equation}
\sum_{1\leq l\leq r}
\frac{\gg_i^{-N_l} \wh{\ms F_i}^{(p,m)}\gg_i^{N_l}.\hh_j^{-M_l} \wh{\ms G_j}^{(p,m')}\gg_i^{M_l}}
{\tr \left(\wh{\ms F}_i^{(p,0)}\right)\tr \left(\wh{\ms G}_j^{(p,0)}\right)} =r.\bu\gg_i\bu\hh_j+K \cdotp\gir_0(\rho)^{M+N}+r\gir_0(\rho)^{Np},\label{eq:asymconj1}
\end{equation}
where $K$ is a locally bounded function of $\rho$ and
$$
M=\inf(Q_i(m-1),P_j(m'-1),Q_ip-Qm',M_jp-m)\, .
$$
\end{proposition}
\begin{proof} In this proof, as usual $K_i$ will denote a locally bounded function of $\rho$.
For the purpose of this proof, we define
$$
\widetilde{F}_i^{(p)}=\wh{\ms g}_{i+1}^p\ldots\wh{\ms g}_{i-1}^p, \ \ \widetilde {G}_j^{(p)}=\wh{\ms h}_{j+1}^p\ldots\wh{\ms h}_{j-1}^p.
$$
By definition, if $m\geq 1$, $m'\geq 1$, $n<Q_i$ and $r<P_j$,
\begin{eqnarray*}
\gg_i^{-n} \wh{\ms F}_i^{(p,m)}\gg_i^{n}&=&\ms g_i^{Q_im-n}\widetilde{\ms F}_i^{(p)}\ms g_i^{Q_i(p-m)+n}, \\
\hh_j^{-r} \wh{\ms G}_j^{(p,m')}\hh_j^{r}&=&\ms h_j^{P_jm'-r}\widetilde{\ms G}_j^{(p)}\ms h_i^{P_j(p-m')+r}.
\end{eqnarray*}
Observe also that 
\begin{eqnarray*}
\tr \left(\wh{\ms F}_i^{(p,0)}\right)&=&\tr \left(\wh{\ms F}_i^{(p,m)}\right),\\
\tr \left(\wh{\ms G}_j^{(p,0)}\right)&=&\tr \left(\wh{\ms G}_j^{(p,m')}\right).
\end{eqnarray*}
Thus using the asymptotics of Corollary \ref{Asym1-coro}, we get that
\begin{eqnarray*}
\frac{\gg_i^{-N_l} \wh{\ms F}_i^{(p,m)}\gg_i^{N_l}} 
{\tr \left(\wh{\ms F}_i^{(p,0)}\right)}& = &\bu{\ms g}_{i}+K_3.\gir_0(\rho)^{R_l},
\end{eqnarray*}
where $A_l=\inf(Q_im-N_l,Q_i(p-m)+N_l,Np)$ and we have observed that for all $k$, $Q_k\geq N$.
Similarly
\begin{eqnarray*}
\frac{\hh_j^{-M_l} \wh{\ms G}_j^{(p,m')}\hh_i^{M_l}}
{\tr \left(\wh{\ms G}_j^{(p,0)}\right)} & = &\bu{\ms h}_{i}+K_4.\gir_k(\rho)^{M}.
\end{eqnarray*}
where $B_l=\inf(P_jm'-N_l,P_j(p-m')+N_l,Np )$. 
Thus 
\begin{eqnarray*}
\sum_{1\leq l\leq r}\frac{\gg_i^{-N_l} \wh{\ms F}_i^{(p,m)}\gg_i^{N_l}.\hh_j^{-M_l} \wh{\ms G}_j^{(p,m')}\hh_i^{M_l}}
{\tr \left(\wh{\ms F}_i^{(p,0)}\right)\tr \left(\wh{\ms G}_j^{(p,0)}\right)}
& =& r\bu\gg_i\bu\hh_j+K_0 \cdotp\left(
\sum_{1\leq l\leq r} \gir_0(\rho)^{R_l}\right)\, ,
\end{eqnarray*}
where   
$$
R_l=\inf(Q_im-N_l,Q_i(p-m)+N_l,P_jm'-M_l,P_j(p-m')+M_l,Np )).
$$
To conclude the proof, we will show that 
\begin{equation}
\sum_{1\leq l\leq r} \gir_0(\rho)^{R_l}\leq \frac{4\gir_0(\rho)^{N+M}}{1-\gir_0(\rho)}+r\gir_0(\rho){Np}.\label{eq:asymconj2}
\end{equation}
Let 
\begin{eqnarray*}
\mathcal A&=&\{l\mid R_l=Q_im-N_l\},\cr
\mathcal B&=&\{l\mid R_l=Q_ip-Qm+N_l\},\cr
\cF&=&\{l\mid R_l=P_jm'-M_l\},\cr
\mathcal D&=&\{l\mid R_l=P_jp-Pm'+M_l\}.
\end{eqnarray*}
By definition, 
\begin{eqnarray*}
\sum_{l\in\mathcal A} \gir_0(\rho)^{R_l}&=&\sum_{l\in\mathcal A} \gir_0(\rho)^{Qm-N_l}\cr
&\leq&\sum_{n\geq Q_i(m-1)+N}\gir_0(\rho)^n\cr
&\leq&\frac{\gir_0(\rho)^{N+Q_i(m-1)}}{1-\gir_0(\rho)}.
\end{eqnarray*}
Symmetric arguments show that
\begin{eqnarray*}
\sum_{i\in\mathcal B} 
\gir_0(\rho)^
{R_l}
\leq \frac{\gir_0(\rho)^{N+Q_i(p-m)}}{1-\gir_0(\rho)}\, ,\cr
\sum_{l\in\cF} 
\gir_0(\rho)^
{R_l}
\leq \frac{\gir_0(\rho)^{N+P_j(m'-1)}}{1-\gir_0(\rho)}\, ,\cr
\sum_{l\in\mathcal D} 
\gir_0(\rho)^
{R_l}
\leq \frac{\gir_0(\rho)^{N+P_j(p-m')}}{1-\gir_0(\rho)}\, .
\end{eqnarray*}
Inequality \eqref{eq:asymconj2} -- and thus the result -- follows.
\end{proof}

\subsubsection{Asymptotics and bouquets}

We use the same notations as in the beginning of this section: Let $G=(\gamma_0,\ldots,\gamma_q)$ and $(F=\eta_0,\ldots,\eta_{q'})$ be two tuples of primitive elements of $\pi_1(S)$. We assume that $(\gamma_i,\gamma_{i+1})$ as well as $(\eta_j,\eta_{j+1})$ are all pairwise coprime. We shall use the notation
 of Paragraph \ref{FpGp}.

\begin{proposition}\label{AsymBP} 
Assume that $G$ and $F$ and $\Gamma_k$ satisfy the Good Position Hypothesis. Assume also that $\Gamma_k$ is $N$-nice for all pairs $(\gamma_i,\eta_j)$. Let $C$ be a bouquet in a good position representing $G$ and $F$.

Then for every compact set $U$ in $\Hn$, we have for every $\rho$ in $U$, \begin{multline}
B^k_p(\gamma_0,\ldots,\gamma_q,\eta_0,\ldots,\eta_{q'})(\rho)
=\sum_{\substack{0\leq i\leq q\\ 0\leq j\leq q'}}\Bigg((p-1)^2 {\rm f}_{i,j}(\cF)\tw(\gamma_i,\eta_j)\\
+(p-1)\left(\frac{\tw(\gamma_{i+1},\gamma_{i},\eta_j)}
{\tw(\gamma_i,\gamma_{i+1})}\left({\rm n}_{i,j}(\cF)+{\rm f}_{i+1,j}(\cF)
\right)
+\frac{\tw(\gamma_i,\eta_{j+1},\eta_j)}
{\tw(\eta_j,\eta_{j+1})}\left({\rm m}_{i,j}(\cF)+{\rm f}_{i,j+1}(\cF)
\right)\right)
\\
+\frac{\tw(\gamma_{i+1},\gamma_i,\eta_{j+1},\eta_j)}{\tw(\gamma_{i+1},\gamma_i)\tw(\eta_j,\eta_{j+1})}\left({\rm q}_{i,j}(\cF)+{\rm n}_{i,j+1}(\cF)+{\rm m}_{i+1,j}(\cF)+{\rm f}_{i+1,j+1}(\cF)\right)\Bigg)\\
+p^2\left(\sum_{i,j}I_{i,j}(1)\sharp(C_{i,j})\tw(\gamma_i,\eta_j)\right)
+K \cdotp(\gir_k(\rho)+\gir_0(\rho)^N)\,,
\end{multline}
where
 \begin{itemize}
\item $K$ is bounded by a continuous function that only depends on the relative position of the eigenvectors of $\rho(\gamma_i)$ and $\rho(\eta_j)$.
\item $\gir(\rho)$ is the girth of $\rho$ as defined in Definition \ref{def:girth} and $\gir_k(\rho)=\gir\left(\left.\rho\right\vert_{\Gamma_k}\right)$
\item the integers ${\mathrm f}_{i,j}(\cF)$, ${\mathrm m}_{i,j}(\cF)$, ${\mathrm n}_{i,j}(\cF)$ and ${\mathrm q}_{i,j}(\cF)$ are defined in Equations \eqref{fmnq}.
\end{itemize}

\end{proposition}

\subsubsection{Proof of Proposition \ref{AsymBP}}
We now recall the Product Formula \eqref{eq:pf2} that we write using the notation of Paragraph \ref{par:prelimas} as
\begin{eqnarray}
B_p= B^0_p+\sum_{\xi\in C_{i,j}}I_{i,j}(\xi)B^\xi_p,
\end{eqnarray}
where
\begin{eqnarray}
B^0_p&\defeq &\frac{1}{\tr\left(\wh{\ms F}^{(p,0)}\right). \tr\left(\wh{\ms G}^{(p,0)}\right).}\sum_{\substack{0\leq i\leq q\\ 0\leq j\leq q'}}\Bigg(\mathop{\sum_{1\leq m'\leq p}}_{1\leq m\leq p}{\rm f}_{i,j}(\cF) \tr\left(\wh{\ms F}^{(p,m') }_i \wh{\ms G}^{(p,m)}_j\right)
\cr
& &+\mathop{\sum}_{1\leq m\leq p}{\rm m}_{i,j}(\cF) \tr\left(\wh{\ms F}^{(p,m) }_i \wh{\ms G}^{(p,0)}_j \right)\cr
& &
+\mathop{\sum_{1\leq m'\leq p}}{\rm n}_{i,j}(\cF) \tr\left(\wh{\ms F}^{(p,0) }_i  \wh{\ms G}^{(p,m')}_j \right)+{\rm q}_{i,j}(\cF) \tr\left(\wh{\ms F}^{(p,0) }_i \wh{\ms G}^{(p,0)}_j \right)\Bigg)\,,
\end{eqnarray}
and
\begin{eqnarray}
B^\xi_p&\defeq &\frac{1}{\tr\left(\wh{\ms F}^{(p,0)}\right). \tr\left(\wh{\ms G}^{(p,0)}\right)}\mathop{\sum_{1\leq m'\leq p}}_{1\leq m\leq p}\tr\left(\wh{\ms F}^{(p,m') }_i \rho(\xi)\wh{\ms G}^{(p,m)}_j\rho(\xi)^{-1}\right)\, .
\end{eqnarray}

Proposition \ref{AsymBP} will follow from the two next propositions that treat independently the term $B^0_p$ and the term involving the $B^\xi_p$.

\begin{proposition}\label{AsymBPtt}
We have
\begin{multline}
B^0_p(\gamma_0,\ldots,\gamma_q,\eta_0,\ldots,\eta_{q'})
=\sum_{\substack{0\leq i\leq q\\ 0\leq j\leq q'}}\Bigg((p-1)^2 {\rm f}_{i,j}(\cF)\tw(\gamma_i,\eta_j)\\
+(p-1)\left(\frac{\tw(\gamma_{i+1},\gamma_{i},\eta_j)}
{\tw(\gamma_i,\gamma_{i+1})}\left({\rm n}_{i,j}(\cF)+{\rm f}_{i+1,j}(\cF)
\right)
+\frac{\tw(\gamma_i,\eta_{j+1},\eta_j)}
{\tw(\eta_j,\eta_{j+1})}\left({\rm m}_{i,j}(\cF)+{\rm f}_{i,j+1}(\cF)
\right)\right)
\\
+\frac{\tw(\gamma_{i+1},\gamma_i,\eta_{j+1},\eta_j)}{\tw(\gamma_{i+1},\gamma_i)\tw(\eta_j,\eta_{j+1})}\left({\rm q}_{i,j}(\cF)+{\rm n}_{i,j+1}(\cF)+{\rm m}_{i+1,j}(\cF)+{\rm f}_{i+1,j+1}(\cF)\right)\Bigg)+K \cdotp\gir_k(\rho),
\end{multline}
where $K$ only depends on the position of the eigenvectors of $\rho(\gamma_i)$ and $\rho(\eta_j)$.
\end{proposition}
\begin{proof} Using the estimates for $\wh{\ms F}_i^{(p,m)}$ and $\wh{\ms G}_i^{(p,m)}$coming from Proposition \ref{pro:fpm} we get that 
\begin{eqnarray}
& &B^0_p= \sum_{\substack{0\leq i\leq q\\ 0\leq j\leq q'}}\Bigg({\rm f}_{i.j}(\cF)(p-1)^2\tr (\bu{\ms g}_i.\bu{\ms h}_j)\cr
&+&{\rm f}_{i.j}(\cF)\left(\frac{\tr (\bu{\ms g}_{i}.\bu{\ms g}_{i-1}.\bu{\ms h}_{j}.\bu{\ms h}_{j-1})}
{\tr (\bu{\ms g}_{i}.\bu{\ms g}_{i-1})\tr (\bu{\ms h}_{j}.\bu{\ms h}_{j-1})}
+(p-1)\left(\frac{\tr (\bu{\ms g}_i.\bu{\ms h}_{j}.\bu{\ms h}_{j-1})}{\tr (\bu{\ms h}_{j}.\bu{\ms h}_{j-1})}+
\frac{\tr (\bu{\ms h}_j.\bu{\ms g}_{i}.\bu{\ms g}_{i-1})}{\tr (\bu{\ms g}_{i}.\bu{\ms g}_{i-1})}\right)
\right)\cr
&+&{\rm m}_{i.j}(\cF)\left((p-1)\frac{\tr (\bu{\ms g}_i.\bu{\ms h}_{j+1}.\bu{\ms h}_{j})}{\tr (\bu{\ms h}_j.\bu{\ms h}_{j+1})}
+\frac{\tr (\bu{\ms g}_{i}.\bu{\ms g}_{i-1}.\bu{\ms h}_{j+1}.\bu{\ms h}_{j})}
{\tr (\bu{\ms g}_{i}.\bu{\ms g}_{i-1})\tr (\bu{\ms h}_{j+1}.\bu{\ms h}_{j})}\right)\cr
&+&{\rm n}_{i.j}(\cF)\left((p-1)
\frac{\tr (\bu{\ms g}_{i+1}.\bu{\ms g}_{i}.\bu{\ms h}_j)}{\tr (\bu{\ms g}_{i+1}.\bu{\ms g}_{i})}
+\frac{\tr (\bu{\ms g}_{i+1}.\bu{\ms g}_{i}.\bu{\ms h}_{j}.\bu{\ms h}_{j-1})}
{\tr (\bu{\ms g}_{i}.\bu{\ms g}_{i-1})\tr (\bu{\ms h}_{j+1}.\bu{\ms h}_{j})}\right)\cr
&+&{\rm q}_{i.j}(\cF)\frac{\tr (\bu{\ms g}_{i+1}.\bu{\ms g}_{i}.\bu{\ms h}_{j+1}.\bu{\ms h}_{j})}{\tr (\bu{\ms g}_{i+1}.\bu{\ms g}_{i})\tr (\bu{\ms h}_{j+1}.\bu{\ms h}_{j})}\Bigg)+K \cdotp\gir_k(\rho).\end{eqnarray}
Using the definition of multifractions, and after reordering terms, we obtain the asymptotics of the proposition.
\end{proof}

Finally we need to understand the last term involving the sum of the terms $B^\xi_p$. 
\begin{proposition}
We have
\begin{equation}
\sum_{i,j}\sum_{\xi\in C_{i,j}}I_{i,j}(\xi) B^\xi_p=p^2\left(\sum_{i,j} I_{i,j}(1)\sharp(C_{i,j})\tr(\bu\gg_i\bu\hh_j)\right) + K \cdotp\gir(\rho_0)^N,\label{bxi}
\end{equation}
where $K$ only depends on the position of the eigenvectors of $\rho(\gamma_i)$ and $\rho(\eta_j)$.
\end{proposition}
\begin{proof} We use again the notations set up in the beginning of Paragraph \ref{par:prelimas}. By definition of an $N$-nice covering, any element $\xi\in C_{i,j}$ can be written as 
$$
\xi=\gamma_i^{N_\xi}\eta_j^{M_\xi},
$$
where $N<N_\xi<Q_i-N$, and $N<M_\xi<P_j-N$. Since for $0<m<Q_j$, $\gamma_i^m\not\in\Gamma_k$, we obtain that $\xi\to N_\xi$ and $\xi\to M_\xi$ are bijections.

Moreover, since the bouquet $C$ is a lift of a bouquet $C^0$ in $S_0$, then 
$$
I_{i,j}(\xi)=I_{i,j}(1).
$$
It follows that for any $i$ and $j$,
\begin{eqnarray}
\sum_{\xi\in C_{i,j}} I_{i,j}(\xi) B^\xi_p=I_{i,j}(1)\mathop{\sum_{1\leq m'\leq p}}_{1\leq m\leq p}\frac{\tr\left(B^{m,m',i,j}_p\right)}{\tr\left(\wh{\ms F}^{(p,0)}\right). \tr\left(\wh{\ms G}^{(p,0)}\right)}\label{eq:bxi2}\, ,\end{eqnarray}
where
$$
B^{m,m',i,j}_p=
\sum_{\xi\in C_{i,j}}
\left(
\gg_i^{-N_\xi}\wh{\ms F}_i^{(p,m)}\gg_i^{N_\xi}
.\hh_j^{-M_\xi}\wh{\ms G}_j^{(p,m')} \hh_j^{M_\xi}\right).
$$
We now apply Proposition \ref{asymconj} to get
\begin{eqnarray}
& &\frac{B^{m,m',i,j}_p}{\tr\left(\wh{\ms F}^{(p,0)}\right). \tr\left(\wh{\ms G}^{(p,0)}\right)}\cr&=&I_{i,j}(1)\sharp (C_{i,j})\left(\bu\gg_i\bu\hh_j+K_0\gir_0(\rho)^{N+M(m,m')} + K_0\gir_0(\rho)^{Np}\right) \label{eq:bxi3},
\end{eqnarray}
where $M(m,m')=\inf(Q_i(p-m),Q_i(m-1),P_j(p-m'),P_j(m'-1))$.
Observe that for any $\lambda<1$
$$
\mathop{\sum_{1\leq m'\leq p}}_{1\leq m\leq p}\lambda^{M(m,m')}\leq 4\sum_{n\leq 0}\lambda^n=\frac{4}{1-\lambda}.
$$
Thus Equations \eqref{eq:bxi2} and \eqref{eq:bxi3} together yield
\begin{eqnarray}
\sum_{i,j}\sum_{\xi\in C_{i,j}} I_{i,j}(\xi) B^\xi_p=p^2\left(\sum_{i,j}I_{i,j}(1)\sharp(C_i,j). \tr\left(\bu\gg_i.\bu\hh_j\right)\right) +K\gir_0(\rho)^N,
\end{eqnarray}
where we have used that $p^2\gir_0(\rho)^{Np}\leq K_5\gir_0(\rho)^{N}$ for some constant $K_5$ only depending on a compact neighborhood of $\rho$.
The result finally follows from the fact that $\tr(\bu\gg_i.\bu\hh_j)=\tw(\gamma_i,\eta_j)$.
\end{proof}

\subsubsection{Proof of Proposition \ref{AsymBP0}}

Since $G$, $F$ and $\Gamma_k$ satisfy the Good Condition Hypothesis, by Proposition
\ref{SettingArcGoodPosition}, there exist two bouquets $\cF_L$ and $\cF_R$ in $S$ in a homotopically good position, both representing $G$ and $F$ such that furthermore 
\begin{eqnarray*}
\frac{1}{2}({\rm f}_{i,j}({\cF_L})+{\rm f}_{i,j}(\cF_R))&={\rm f}_{i,j}\, ,\\
\frac{1}{2}({\rm n}_{i,j}(\cF_L)+{\rm n}_{i,j}(\cF_R))&={\rm n}_{i,j}\, ,\\
\frac{1}{2}({\rm m}_{i,j}(\cF_L)+{\rm m}_{i,j}(\cF_R))&={\rm m}_{i,j}\, ,\\
\frac{1}{2}({\rm q}_{i,j}(\cF_L)+{\rm q}_{i,j}(\cF_R))&={\rm q}_{i,j}\, .
\end{eqnarray*}
Thus applying Proposition \ref{AsymBP} twice, once for $\cF_L$ and once for $\cF_R$, and taking the half sum, we obtain the final result.

\subsection{Asymptotics of brackets of multifractions}

The setting of this paragraph is the same as the previous one: we shall be given a finite index subgroup $\Gamma_k$ of $\Gamma_0=\pi_1(S)$, corresponding to a covering $S_k\to S_0=S$. Then, if $\rho$ is a Hitchin representation of $\pi_1(S)$ in $\sln$, $\rho_k$ will denote the restriction of $\rho$ to $\Gamma_k$.

Let $G=(\gamma_0,\ldots,\gamma_q)$ and $F=(\eta_0,\ldots,\eta_{q'})$ be two tuples of primitive elements of $\pi_1(S)$. We assume that $(\gamma_i,\gamma_{i+1})$ as well as $(\eta_j,\eta_{j+1})$ are all pairwise coprime. Observe that there exists $M\in\mathbb N$ so that for all $i$ and $j$, $\hat\gamma\defeq \gamma_i^M$ and $\hat\eta\defeq \eta_j^M$ belong to $\Gamma_k$ .

Then let 
$$
\otw(\gamma_1,\ldots,\gamma_q)\defeq \frac{\ww(\hat\gamma_1^p\ldots\hat\gamma_q^p)}{\prod_{i=1}^{i=q} \ww(\hat\gamma_i^p)},
$$
so that
\begin{equation}
\tw=\lim_{p\to\infty}{\otw}.\label{limTW}
\end{equation}
Let now
\begin{eqnarray}
A_p\defeq \frac{\{\otw(\gamma_0,\ldots,\gamma_q),\otw(\eta_0,\ldots,\eta_{q'})\}_S}{\otw(\gamma_0,\ldots,\gamma_q).\otw(\eta_0,\ldots,\eta_{q'})}\label{defAP}
\end{eqnarray}
Let $F=(\gamma_1,\ldots,\gamma_q)$ and $G=(\eta_1,\ldots,\eta_{q'})$. We first have
\begin{proposition}\label{vacote}
We have
\begin{equation}
A_p =B_p(F,G)-\sum_iB_p(\gamma_i,G)-\sum_j B_p(F,\eta_j)+\sum_{\substack{0\leq i\leq q\\ 0\leq j\leq q'}}B_p(\gamma_i,\eta_j).
\end{equation}

\end{proposition}
From this proposition and Proposition \ref{AsymBP}, we will deduce the following important corollary
\begin{corollary}\label{apcomp}
Assume that $G$ and $F$ and $\Gamma_0$ satisfy the Good Position Hypothesis.
Let $k$ be a positive integer so that $\Gamma_k$ is $N$-nice for all pairs $(\gamma_i,\eta_j)$. Then
\begin{eqnarray} 
& &\frac{\left\{\tw(\gamma_0,\ldots\gamma_q),\tw(\eta_0,\ldots\eta_{q'})\right\}_{S_k}}{\tw(\gamma_0,\ldots\gamma_q).\tw(\eta_0,\ldots\eta_{q'})}\cr
&=&\sum_{\substack{0\leq i\leq q\\ 0\leq j\leq q'}}
\Big(
\left({\rm q}_{i,j}+{\rm n}_{i,j+1}+{\rm m}_{i+1,j}+{\rm f}_{i+1,j+1}\right)
\frac{\tw(\gamma_{i+1}\gamma_i,\eta_{j+1},\eta_j)}{\tw(\gamma_{i+1},\gamma_i)\tw(\eta_j,\eta_{j+1})}\cr
& &-\left({\rm n}_{i,j}+{\rm f}_{i+1,j}\right)
\frac{\tw(\gamma_{i+1}\gamma_i,\eta_j)}{\tw(\gamma_{i+1},\gamma_i)}
-\left({\rm m}_{i,j}+{\rm f}_{i,j+1}\right)
\frac{\tw(\gamma_i,\eta_{j+1},\eta_j)}{\tw(\eta_j,\eta_{j+1})}\cr
& &+{\rm f}_{i,j}.\tw(\gamma_i,\eta_j)\Big) +K \cdotp(\gir_k(\rho)+\gir_0(\rho)^N),
\end{eqnarray}
where $K$ is bounded by a continuous function that only depends on the relative position of the eigenvectors of $\rho(\gamma_i)$ and $\rho(\eta_i)$
\end{corollary} 
We first prove the corollary from the proposition, then the proposition in the next paragraph
\vskip 0.2 truecm

\noindent{\em Proof of Corollary \ref{apcomp}:} We study one by one the terms in the right hand side of the formula of Proposition \ref{vacote} using the asymptotics given by Proposition \ref{AsymBP}. Let $\epsilon=\gir_k(\rho)+\gir_0(\rho)^N$.
First,
\begin{multline}
B_p(\gamma_0,\ldots,\gamma_q,\eta_0,\ldots,\eta_{q'})
=\sum_{\substack{0\leq i\leq q\\ 0\leq j\leq q'}}\Bigg((p^2 {\rm R}_{i,j}+(p-1)^2 {\rm f}_{i,j})\tw(\gamma_i,\eta_j)\\
+(p-1)\left(\frac{\tw(\gamma_{i+1},\gamma_{i},\eta_j)}
{\tw(\gamma_i,\gamma_{i+1})}\left({\rm n}_{i,j}+{\rm f}_{i+1,j}
\right)
+\frac{\tw(\gamma_i,\eta_{j+1},\eta_j)}
{\tw(\eta_j,\eta_{j+1})}\left({\rm m}_{i,j}+{\rm f}_{i,j+1}
\right)\right)
\\
+\frac{\tw(\gamma_{i+1},\gamma_i,\eta_{j+1},\eta_j)}{\tw(\gamma_{i+1},\gamma_i)\tw(\eta_j,\eta_{j+1})}\left({\rm q}_{i,j}+{\rm n}_{i,j+1}+{\rm m}_{i+1,j}+{\rm f}_{i+1,j+1}\right)\Bigg)+K\epsilon\,.\label{f01}
\end{multline}
Let us now consider the term $B_p(\gamma_i,\eta_0,\ldots,\eta_{q'})$. We can apply Formula \eqref{f01} using the fact that in this case ${\rm q}_{i,j}={\rm n}_{i,j}=0$ to get
\begin{multline}
B_p(\gamma_i,\eta_0,\ldots,\eta_{q'})
=\sum_{0\leq j\leq q'}\Bigg((p^2 {\rm R}_{i,j}+(p-1)^2 {\rm f}_{i,j})\tw(\gamma_i,\eta_j)\\
+(p-1)\left(\tw(\gamma_{i},\eta_j)
\left({\rm f}_{i,j}
\right)
+\frac{\tw(\gamma_i,\eta_{j+1},\eta_j)}
{\tw(\eta_j,\eta_{j+1})}\left({\rm m}_{i,j}+{\rm f}_{i,j+1}
\right)\right)
\\
+\frac{\tw(\gamma_i,\eta_{j+1},\eta_j)}{\tw(\eta_j,\eta_{j+1})}\left({\rm m}_{i,j}+{\rm f}_{i,j+1}\right)\Bigg)+K\epsilon\,.
\end{multline}
Similarly,
\begin{multline}
B_p(\gamma_0,\ldots,\gamma_q,\eta_j)
=\sum_{0\leq i\leq q}\Bigg((p^2 {\rm R}_{i,j}+(p-1)^2 {\rm f}_{i,j})\tw(\gamma_i,\eta_j)\\
+(p-1)\left(\frac{\tw(\gamma_{i+1},\gamma_{i},\eta_j)}
{\tw(\gamma_i,\gamma_{i+1})}\left({\rm n}_{i,j}+{\rm f}_{i+1,j}
\right)
+\tw(\gamma_i,\eta_j)
\left({\rm f}_{i,j}
\right)\right)
\\
+\frac{\tw(\gamma_{i+1},\gamma_i,\eta_j)}{\tw(\gamma_{i+1},\gamma_i)}\left({\rm n}_{i,j}+{\rm f}_{i+1,j}\right)\Bigg)+K\epsilon\,.
\end{multline}
Finally,
\begin{equation}
B_p(\gamma_i,\eta_j)
=\tw(\gamma_{i},\eta_j)\left(p^2 {\rm R}_{i,j}+(p-1)^2 {\rm f}_{i,j}
+2(p-1)+1\right)+K\epsilon\,.
\end{equation}
Thus, using Proposition \ref{vacote}, regrouping the terms that appears in $A_p$, we obtain that
\begin{itemize}
\item the coefficient of $\tw(\gamma_i,\eta_j)$ is
$
{\rm f}_{i,j}
$,
\item the coefficient of $\frac{\tw(\gamma_{i+1},\gamma_{i},\eta_j)}
{\tw(\gamma_i,\gamma_{i+1})}$
is
$
-({\rm n}_{i,j}+{\rm f}_{i+1,j})
$,
\item the coefficient of $
\frac{\tw(\gamma_i,\eta_{j+1},\eta_j)}
{\tw(\eta_j,\eta_{j+1})}$
is
$
-({\rm m}_{i,j}+{\rm f}_{i,j+1})
$,
\item the coefficient of $\frac{\tw(\gamma_{i+1},\gamma_i,\eta_{j+1},\eta_j)}{\tw(\gamma_{i+1},\gamma_i)\tw(\eta_j,\eta_{j+1})}$
is ${\rm q}_{i,j}+{\rm n}_{i,j+1}+{\rm m}_{i+1,j}+{\rm f}_{i+1,j+1}$.
\end{itemize}
Finally, we conclude the proof of the corollary by using Formula \eqref{limTW}.
\qed

\subsubsection{Proof of Proposition \ref{vacote}}

First we use the ``logarithmic derivative formula" for the Poisson bracket
$$
\frac{\{f.g,h\}_S}{fgh}=\frac{\{f,h\}_S}{fh}+\frac{\{g,h\}_S}{gh}.
$$
We obtain
\begin{multline}
A_p(F,G)
= \logd{\ww(\gamma_0^p\ldots\gamma_q^p)}{\ww(\eta_0^p\ldots\eta_{q'}^p)}-\sum_{0\leq i\leq q}\logd{\ww(\gamma_i^p)}{\ww(\eta_0^p\ldots\eta_{q'}^p)}\\ -\sum_{0\leq j\leq q'}\logd{\ww(\gamma_0^p\ldots\gamma_q^p)}{\ww(\eta_j)}+\sum_{\substack{0\leq i\leq q\\ 0\leq j\leq q'}}\logd{\ww(\gamma_i^p)}{\ww(\eta_j^p)}.
\end{multline}
Then, using the definition of Equation \eqref{ABGw} expressing the Goldman Poisson bracket of Wilson loops in terms of the bracket of loops in the Goldman Algebra, we get
$$
\logd{\ww(\gamma_0^p\ldots\gamma_q^p)}{\ww(\eta_0^p\ldots\eta_{q'}^p)}=B_p(F,G)-\frac{1}{n}\ii(\gamma_0^p\ldots\gamma_q^p,\eta_0^p\ldots\eta_{q'}^p).
$$
The proposition now follows from the fact that
$$
\ii(a.b,c)=\ii(a,c)+\ii(b,c),$$
and thus
\begin{multline}\ii(\gamma_0^p\ldots\gamma_q^p,\eta_0^p\ldots\eta_{q'}^p)=\\ \sum_{0\leq i\leq q}\ii(\gamma_i,\eta_0^p\ldots\eta_{q'}^p)+\sum_{0\leq j\leq q'}\ii(\gamma_0^p\ldots\gamma_q^p,\eta_j)-\sum_{\substack{0\leq i\leq q\\ 0\leq j\leq q'}}\ii(\gamma_i^p,\eta_j^p).
\end{multline}

\section{Goldman and swapping algebras: proofs of the main results}\label{sec:gold}

We finally prove the results stated in Section \ref{sec:main}. In the course of the proof, we prove the generalised Wolpert formula in Theorem \ref{theo:genWolp}.

\subsection{Poisson brackets of elementary functions and the proof of Theorem \ref{vanish-sequence}}\label{proof:vanish-sequence}
By Corollary \ref{ElemGen}, the algebra $\mathcal B(\PP)$ of multifractions is generated by elementary functions. Thus it is enough to prove the theorem when $b_0$ and $b_1$ are elementary functions.

Let $G=(\gamma_0,\ldots,\gamma_p)$ and $F=(\eta_0,\ldots\eta_{q'})$ be primitive elements of $\grf$. We assume that for all $i$ and $j$, 
$\gamma_i$ and $\gamma_{i+1}$ are coprime, as well as $\eta_i$ and $\eta_{i+1}$.
 
Let $b_0=\tw(\gamma_0,\ldots,\gamma_q)$ and $b_1=\tw(\eta_0,\ldots,\eta_q')$. 

By Proposition \ref{VanishGPH2}, we can assume that $G$ and $F$ satisfy the Good Position Hypothesis for $S_k$ when $k>k_0$ for some $n_0$. Let $N$ be a positive integer, we can further assume that $S_k\mapsto S_0$ is $N$-nice for all pairs $(\gamma_i,\eta_j)$ by Proposition \ref{VanishGPH3} for $k\geq k_0$ and $k_0$ large enough.

Recall also, using the notation of Proposition \ref{prop:braelem}, that
\begin{align}
{\rm f}_{i,j}&=[\gamma_i^-\gamma_i^+,\eta_j^-\eta_j^+]\ \ \ \ ={\rm a}_{i,j}\, ,&\cr
{\rm q}_{i,j}+{\rm n}_{i,j+1}+{\rm m}_{i+1,j}+{\rm f}_{i+1,j+1}&=[\gamma_i^-\gamma_{i+1}^+,\eta_{j}^-\eta_{j+1}^+]={\rm b}_{i,j}\, ,&\cr
{\rm f}_{i,j+1}+{\rm m}_{i,j}&=[\gamma_i^-\gamma_i^+,\eta_j^-\eta_{j+1}^+] \ \ ={\rm c}_{i,j}\, ,&\cr
{\rm f}_{i+1,j}+{\rm n}_{i,j}&=[\gamma_i^-\gamma_{i+1}^+,\eta_j^-\eta_j^+] \ \ ={\rm d}_{i,j}\, .&\end{align}

Thus Corollary \ref{apcomp} and the computation of the swapping bracket in Proposition \ref{prop:braelem} yield
\begin{multline}
\frac{\left\{\tw(\gamma_0,\ldots\gamma_q),\tw(\eta_0,\ldots\eta_{q'})\right\}_{S_k}}{\tw(\gamma_0,\ldots\gamma_q).\tw(\eta_0,\ldots\eta_{q'})}\cr=\frac{\left\{\tw(\gamma_0,\ldots\gamma_q),\tw(\eta_0,\ldots\eta_{q'})\right\}_W}{\tw(\gamma_0,\ldots\gamma_q).\tw(\eta_0,\ldots\eta_{q'})}+K\cdotp(\gir_k(\rho)+\gir(\rho)^N),
\end{multline}
where $K$ is a bounded function that only depends on the eigenvectors of $\rho(\gamma_i)$ and $\rho(\eta_j)$. In particular, there exists a real number $K_0$ and a compact neighbourhood $C$ of $\rho_0$ so that the previous equality holds with $K\leq K_0$ and $\rho$ in $C$.

Let $\epsilon$ be a positive real number.
By the last assertion in Proposition \ref{VanishGPH1}, we may furthermore choose $k_0$ so that if $k>k_0$, 
$$
\gir_k(\rho)\leq \frac{\epsilon}{2K_0}\, .
$$
Since $\sup\{\gir(\rho)\mid \rho\in C\}<1$, we may further choose $N$ -- and thus $k_0$ -- so that for all $\rho$ in $C$,
$$
\gir(\rho)^N\leq \frac{\epsilon}{2K_0}\, .
$$
It follows that for all $\rho$ in $C$, for all $k\geq k_0$, we have
\begin{equation}
\left\vert\frac{\left\{\tw(\gamma_0,\ldots\gamma_q),\tw(\eta_0,\ldots\eta_{q'})\right\}_{S_k}}{\tw(\gamma_0,\ldots\gamma_q).\tw(\eta_0,\ldots\eta_{q'})}-\frac{\left\{\tw(\gamma_0,\ldots\gamma_q),\tw(\eta_0,\ldots\eta_{q'})\right\}_W}{\tw(\gamma_0,\ldots\gamma_q).\tw(\eta_0,\ldots\eta_{q'})}\right\vert\leq \epsilon\, .
\end{equation}

This concludes the proof of Theorem \ref{vanish-sequence}.

\subsection{Poisson brackets of length functions}\label{proof-length}

We shall first prove a result of independent interest, namely the computation of the value of the Goldman bracket of two length functions of geodesics having exactly one intersection point. 

Given a Hitchin representation $\rho$ in $\sln$, or alternatively a rank-$n$ cross ratio $\bb_\rho$, the period -- or length -- of a conjugacy class $\gamma$ in $\grf$, is given by
\begin{align}
\ell_\gamma(\rho)=\log\left(\frac{\lambda_{\text{max}}(\rho(\gamma))}{\lambda_{\text{min}}(\rho(\gamma))}\right)=\log\left(\bb_\rho(\gamma^+,\gamma^-,\gamma(y),y)\right),
\end{align}
for any $y\in\bgrf$ different from $\gamma^+$ and $\gamma^-$ and where $\lambda_{\text{max}}(A)$ and $\lambda_{\text{min}}(A)$ denotes respectively the eigenvalue with the greatest and smallest modulus of the endomorphism $A$.

\subsubsection{A generalised Wolpert Formula}

We have the following extension of Wolpert Formula for the bracket of length functions:
\begin{theorem}{\sc [Generalised Wolpert Formula]}\label{theo:genWolp1}
Let $\gamma$ and $\eta$ two closed geodesic with a unique intersection point then the Goldman bracket of the two length functions $\ell_\gamma$ and $\ell_\eta$ seen as functions on the Hitchin component is 
\begin{equation}
\{\ell_\gamma,\ell_\eta\}_S=\ii(\gamma,\eta)\sum_{\epsilon,\epsilon^\prime\in\{-1,1\}}\epsilon\epsilon^\prime.\tw(\gamma^{\epsilon},\eta^{\epsilon^\prime}),\label{eq:KSF}
\end{equation}
where we recall that
$$
\tw(\xi,\zeta)(\rho)=\bb_\rho(\xi^+,\zeta^+,\zeta^-,\xi^-).
$$
\end{theorem}

\begin{proof}
Let us first remark that
\begin{align}
\ell_\gamma=\lim_{p\rightarrow+\infty}\frac{1}{p}\log\left({\tr(\rho(\gamma^p))}{\tr(\rho(\gamma^{-p}))}\right).
\end{align}
Thus assuming that $\gamma$ and $\eta$ have a unique intersection point $x$, whose intersection number is 
$\ii(\gamma,\eta)$, the Product Formula \eqref{eq:pf2} gives us if $\epsilon_i\in\{-1,1\}$,
\begin{align}
\{\gamma^{\epsilon.p}, \eta^{\epsilon^\prime.p}\}=\epsilon\epsilon^\prime. p^2. \ii(\gamma,\eta)\gamma^{\epsilon.p}.\eta^{\epsilon^\prime.p}.
\end{align}
It follows that
\begin{align}
&\left\{\log\left({\ww(\gamma^p)}{\ww(\gamma^{-p})}\right),\log\left({\ww(\eta^p)}{\ww(\eta^{-p})}\right)\right\}_S\cr &=\sum_{\epsilon,\epsilon^\prime\in\{-1,1\}}\epsilon\epsilon^\prime.\ii(\gamma,\eta)\frac{\{\ww(\gamma^{\epsilon.p}),\ww(\eta^{\epsilon^\prime.p})\}_S}{\ww(\gamma^{\epsilon.p}).\ww(\eta^{\epsilon^\prime.p})}\cr
&=\sum_{\epsilon,\epsilon^\prime\in\{-1,1\}}p^2.\epsilon\epsilon^\prime.\ii(\gamma,\eta)\frac{\ww(\gamma^{\epsilon.p}.\eta^{\epsilon^\prime.p})}{\ww(\gamma^{\epsilon.p}).\ww(\eta^{\epsilon^\prime.p})}+\frac{1}{n}\ii(\gamma,\eta)\sum_{\epsilon,\epsilon^\prime\in\{-1,1\}}\epsilon\epsilon^\prime\,.
\end{align}
Thus
\begin{align}
&\lim_{p\rightarrow\infty}\left(\left\{\frac{1}{p}\log\left({\ww(\gamma^p)}{\ww(\gamma^{-p})}\right),\frac{1}{p}\log\left({\ww(\eta^p)}{\ww(\eta^{-p})}\right)\right\}_S\right)\cr
&=\sum_{\epsilon,\epsilon^\prime\in\{-1,1\}}\epsilon\epsilon^\prime.\ii(\gamma,\eta)\lim_{p\to\infty}\left(\frac{\ww(\gamma^{\epsilon.p}.\eta^{\epsilon^\prime.p})}{\ww(\gamma^{\epsilon.p}).\ww(\eta^{\epsilon^\prime.p})}\right)\cr
&=\sum_{\epsilon,\epsilon^\prime\in\{-1,1\}}\epsilon\epsilon^\prime.\ii(\gamma,\eta)\tw(\gamma^{\epsilon},\eta^{\epsilon^\prime})\,.
\end{align}
This concludes the proof of the theorem.
\end{proof}

\subsubsection{Proof of Theorem \ref{theo:braleng}}

Recall that we want to prove the following result.

\begin{theorem}
Let $\gamma$ and $\eta$ be two geodesics with at most one intersection point, then we have 
$$
\lim_{n\to\infty}{\rm I}_S(\{\hat\ell_{\gamma^n}(y),\hat\ell_{\eta^n}(y)\})=\frac{1}{4}\{\ell_{\gamma},\ell_{\eta}\}_S.
$$
\end{theorem}

\begin{proof} This will be a consequence of the generalised Wolpert Formula.
By definition,
$$
\hat\ell_\gamma(y)=\frac{1}{2}\log(\bb(\gamma^+,\gamma^-,\gamma(y),\gamma^{-1}(y))).
$$
Thus
\begin{equation}
\{\hat\ell_{\alpha}(y),\hat\ell_\beta(y)\}=\frac{1}{4}\sum_{\substack{u,u^{\prime}\in\{-1,1\} \\ v,v^{\prime}\in\{-1,1\}}}u.u^\prime\frac{(\{(\alpha^{v},\alpha^{-uv}(y)), (\beta^{v^\prime},\beta^{-u^\prime v^\prime}(y))\}}{ (\alpha^{v},\alpha^{-uv}(y)).(\beta^{v^\prime},\beta^{-u^\prime v^\prime}(y))}.\label{l:lim1}
\end{equation}
But, 
\begin{multline}
\{(\alpha^{v},\alpha^{-uv.}(y)), (\beta^{v^\prime},\beta^{-u^\prime v^\prime}(y))\}\\
=[(\alpha^{v}\alpha^{-uv.}(y)),(\beta^{v^\prime}\beta^{-u^\prime v^\prime.}(y))]\alpha^{v}\beta^{-u^\prime v^\prime.}(y). \beta^{v^\prime}\alpha^{-uv.}(y).\label{l:lim2}
\end{multline}
We remark that for $n$ large enough, for all $u,v,u^\prime,v^\prime$, we have
\begin{align}
[(\gamma^{v}\gamma^{v.n}(y)),(\eta^{v^\prime}\eta^{-u^\prime v^\prime.n}(y))]&=0,\cr
[(\gamma^{v}\gamma^{-uv.n}(y)),(\eta^{v^\prime}\eta^{v^\prime.n}(y))]&=0,\cr
[(\gamma^{v}\gamma^{-v.n}(y)),(\eta^{v^\prime}\eta^{-v^\prime.n}(y))]&=vv^\prime[\gamma^+\gamma^-,\eta^+\eta^-].\label{l:lim3}
\end{align}
Combining the remark in Equation \eqref{l:lim3}, with \eqref{l:lim2} and \eqref{l:lim1}, we have that for $n$ large enough,
\begin{equation}
\{\hat\ell_{\gamma^n}(y),\hat\ell_{\eta^n}(y)\}
=\frac{[\gamma^+\gamma^-,\eta^+\eta^-]}{4}\sum_{\substack{v,v^{\prime}\in\{-1,1\}}}v.v^\prime\frac{\gamma^{v}\eta^{-v^\prime.n}(y). \eta^{v^\prime}\gamma^{-v}(y)}{ (\gamma^{v}\gamma^{-v.n}(y)).(\eta^{v^\prime}\eta^{-v^\prime n}(y))}\,.\label{l:lim4}
\end{equation}
Thus, taking the limit when $n$ goes to $\infty$ yields
\begin{align}
\lim_{n\to\infty}\left(\ms I_S\{\hat\ell_{\gamma^n}(y),\hat\ell_{\eta^n}(y)\}\right)
&=
\frac{[\gamma^+\gamma^-,\eta^+\eta^-]}{4}\sum_{\substack{v,v^{\prime}\in\{-1,1\}}}v.v^\prime\frac{\gamma^{v}\eta^{-v^\prime}. \eta^{v^\prime}\gamma^{-v}}{ \gamma^{v}\gamma^{-v}. \eta^{v^\prime}\eta^{-v^\prime}}\cr
&=
\frac{[\gamma^+\gamma^-,\eta^+\eta^-]}{4}\sum_{v,v^{\prime}\in\{-1,1\}}v.v^\prime.\tw
(\gamma^{v},\eta^{v^\prime}).
\end{align}
The result now follows from this last equation and the generalised Wolpert formula 
\eqref{eq:KSF}.
\end{proof}

\section{Drinfel'd-Sokolov reduction}\label{sec:oper}

The purpose of this section is to prove Theorem \ref{DF-swap} which explain the relation of the multifraction algebra with the Poisson structure on $\sln$-opers.

We spend the first three paragraphs explaining the Poisson structure on $\sln$-opers using the Drinfel'd-Sokolov reduction of the Poisson structure on connections on the circle. Although this is a classical construction (see \cite{Dickey:1997un,vanMoerbeke:1998wl,Guha:2007wf} and the original reference \cite{Drinfelcprimed:1981ua}) we take some time explaining the main steps in differential geometric terms, expanding the sketch of the construction given by Graeme Segal in \cite{Segal:1991}. 

Finally, we relate the swapping algebra and this Poisson structure in Theorem \ref{DF-swap}.

\subsection{Opers and non-slipping connections}

In this paragraph, we recall the definition $\sln$-opers and show that they can be interpreted as class of equivalence of ``non-slipping" connections on a bundle with a flag structure.

\subsubsection{Opers}
\begin{definition}{\sc [Opers] }
A {\em $\sln$-oper} is an $n^{\hbox{\tiny th}}$-order linear differential operator on the circle $\mathbb T=\mathbb R/\mathbb Z$ of the form
\begin{eqnarray}
D\ :\ \psi\mapsto\frac{\dd ^n\psi}{\dd t^n}+q_{2}\frac{\dd ^{n-2}\psi}{\dd t^{n-2}}+\ldots+q_n\psi,
\end{eqnarray}
where $q_i$ are functions. 
\end{definition}

Observe that this definition of an oper requires the choice of a parametrisation of the circle. Otherwise the $q_i$ would rather be $i^{\hbox{\tiny th}}$-order differentials.

We denote by ${\ms X}_n(\mathbb T)$ the space of $\sln$-opers on $\mathbb T$. Every oper has a natural holonomy which reflect the fact that the solutions may not be periodic. We consider the space ${\ms X}_n(\mathbb T)^0$ of opers with trivial holonomy, that is those opers $D$ for which all solutions of $D\psi=0$ are periodic. 
A Poisson structure on ${\ms X}_n(\mathbb T)$, whose symplectic leaves are opers with the same holonomy, was discovered in the context of integrable systems and Korteweg--de Vries equations -- for a precise account of the history, see Dickey \cite{Dickey:1997un}. Later, Drinfel'd and Sokolov interpreted that structure in a more differential geometric way in \cite{Drinfelcprimed:1981ua} and we shall now retrace the steps of that construction. 

\subsubsection{Non-slipping connections}\label{non-slip} 
Let $K$ be the line bundle of $(-1/2)$-densities over $\mathbb T$ (so that $\mathsf T\mathbb T=K^2$) and $P\defeq J^{n-1}(K^{n-1})$ be the rank $n$ vector bundle of $(n-1)$ jets of sections of the bundles of $(-(n-1)/2)$-densities.

Let $F_p$ be the vector subbundle of $P$ defined by 
$$
F_p\defeq \{j^{n-1}\sigma\mid j^{n-p-1}\sigma=0\}.
$$
The family $\{F_p\}_{1\leq p\leq n}$ is a filtration of $P$: we have $F_n=P$, $F_{p-1}\subset F_{p}$ and $\dim(F_p)=p$. Observe that $$
W_p\defeq F_p/F_{p-1}=\left(\mathsf T^*\mathbb T\right)^{n-p}\otimes K^{n-1}=\left(K^{-2}\right)^{n-p}\otimes K^{n-1}=K^{2p-n-1}.$$ 
In particular $W_{n-p-1}=W_p^*$ and it follows that
$$
\det(P)=\bigotimes_{p=1}^n \det(W_p),
$$
is canonically isomorphic to $\mathbb R$. Thus $P$ carries a canonical volume form.

We say a family of sections $\{e_1,\ldots,e_n\}$ of $P$ is a {\em basis for the filtration} if for every integer $p$ no greater than $n$, for every $x\in S^1$, $\{e_1(x),\ldots,e_p(x)\}$ is a basis of the fibre of $F_p$ at $x$.

\begin{definition}{\sc[non-slipping connections]}
A connection $\nabla$ on $P$ is {\em non-slipping} if it satisfies the following conditions
\begin{itemize}
\item $\nabla F_p\subset F_{p+1}$ for all $p$,
\item If $\alpha_p$ is the projection from $F_{p+1}$ to $F_{p+1}/F_p$, then the map $$
(X,u)\to \alpha_p(\nabla_{X}(u)),
$$ 
considered as a linear map from $K^2\otimes F_p/F_{p-1}=K^{2p-n+1 }$ to $F_{p+1}/F_{p-1}=K^{2p-n+1}$ is the identity.
\end{itemize}

\end{definition}

We denote by $\DD_0$ the space of non--slipping connections on $P$. The first classical proposition is that
\begin{proposition}
Let $\nabla$ be a non-slipping connection, then there exists a unique basis $\{e_1,\ldots,e_n\}$ of determinant 1 for the filtration so that
\begin{eqnarray}
\nabla_{\partial_t}e_i&=&-e_{i+1},\hbox{ for } i\leq n-1,\cr
\nabla_{\partial_t}e_n&\in&F_{n-1},
\end{eqnarray}
where $\partial_t$ is the canon-ical vector field on $\mathbb T$.
\end{proposition}
Observe here that the basis depends on the choice of a parametrisation of the circle. From this proposition, it follows that we can associate to a non-slipping connection $\nabla$ the differential operator $D=D^\nabla$ so that
$$
 \nabla^*_{\partial_t}\left(\sum_{i=1}^{i=n}\frac{\dd ^{i-1}\psi}{\dd t^{i-1}} \omega_i\right)=(D\psi)\,\omega_n,
$$
where $\nabla^*$ is the dual connection and $\{e_i\}_{1\leq i\leq n}$ is the dual basis to the basis $\{\omega_i\}_{1\leq i\leq n}$ associated to $\nabla$ in the previous proposition. One easily check that
$$
D\psi=\frac{\dd ^n\psi}{\dd t^n}+q_{2}\frac{\dd ^{n-2}\psi}{\dd t^{n-2}}+\ldots+q_n\psi\, ,
$$
where the functions $q_i$ are given by $q_i=\omega_{n-j+1}(\nabla_{\partial_t}e_n)$. 

We now introduce 
\begin{enumerate}
\item the {\em flag gauge group} as the group $\ms N$ of linear automorphism of the bundle $P$ defined by
$$
\ms N\defeq \{A\in \Omega^0(\mathbb T, \End(P)), \mid A(F_p)=F_p, \ \left.A\right\vert_{F_p/F_{p-1}}=\operatorname{Id}\},
$$
\item the {\em Lie algebra} $\mk n$ of the flag gauge group as
$$
\mk n\defeq \{A\in \Omega^0(\mathbb T, \End(P))\mid A(F_p)\subset F_{p-1}\}.
$$
\end{enumerate}

We now have
\begin{proposition}
The map $\nabla\mapsto D^\nabla$ realise an identification between $\DD_0/\ms N$ and ${\ms X}_n(\mathbb T)$ and this identification preserves the holonomy. 
\end{proposition}
It is interesting now to observe that the definition of an oper as an element of $\DD_0/\ms N$ does not depend of a parametrisation. 
\begin{proof} Let $\nabla$ be a non-slipping connection, let $\{e_i\}$ the basis obtained by the previous proposition and $\nabla'=n^*\cdot\nabla$ be a connection in the $\ms N$-orbit of $\nabla$. By definition of $\ms N$, 
$
\nabla' e_i=\nabla e_i +u_i$, with $u_i\in F_{i-1}$. The result follows \end{proof}

\subsection{The Poisson structure on the space of connections}

The purpose of Drinfel'd-Sokolov reduction is to identify the space of opers ${\ms X}_n(\mathbb T)=\DD_0/\ms N$ as a symplectic quotient of the space of all connections on $\mathbb T$ by the group $\ms N$. 

Again, we shall paraphrase Segal, and define in this paragraph, as a first step of the construction of Drinfel'd-Sokolov reduction, the classical construction of the Poisson structure on the space of connections.

In general, when we deal with a Fréchet space of sections of a bundle, we have to specify functionals that we deemed observables and for which we can compute a Poisson bracket. This is done by specifying a subspace of cotangent vectors and describing the Poisson tensor on that subspace. Observables are then functionals whose differentials belong to that specific subset. However, the Poisson bracket can be extended to more general pair of observables. Rather than describing a general formalism -- for which we could refer to \cite{Dickey:1997un} -- we explain the construction in the case of connections.

\subsubsection{Connections and central extensions}

Let $\ms G$ be the Gauge group of the vector bundle $P$. The choice of a trivialisation of $P$ give rise an isomorphism of $\ms G$ with the loop group of $\sln$. Then, we introduce the following definitions
\begin{enumerate}
\item the {\em Lie algebra} $\mk g$ of $\ms G$ is $\Omega^0(\mathbb T,\End_0(P))$ where $\End_0(P)$ stands for the vector space of trace free endomorphisms of $P$, the Lie algebra $\mk g$ is equipped naturally with a coadjoint action of $\ms G$,
\item the {\em dual Lie algebra} $\mk  g^{\dual}$ of $\ms G$ is $\Omega^1(\mathbb T,\End_0(P))$,
\item the {\em duality} is given by the non-degenerate bilinear mapping from $\mk  g\times \mk  g^{\dual}$ defined by
\begin{equation}
\langle \alpha,\beta\rangle=\int_{\mathbb T}\tr(\alpha \cdotp\beta)\,.\label{Dual}
\end{equation}
\end{enumerate}

Let us choose a connection $\nabla$ on $P$. Let $\Omega_\nabla$ be the 2-cocycle on $\mk  g$ given by the following formula 
$$
\Omega_\nabla(\xi,\eta)=\int_{\mathbb T}\tr(\xi \nabla\eta). 
$$
If $\nabla$ and $\nabla'$ are two connections on $P$, then 
$$
\Omega_\nabla(\xi,\eta)-\Omega_{\nabla'}(\xi,\eta)=\alpha([\xi,\eta]),
$$
where
$$
\alpha(\chi)=\int_{\mathbb T}\tr((\nabla-\nabla').\chi).
$$
In particular the cohomology class of the cocycle $\Omega_\nabla$ does not depend on the choice of $\nabla$.
Let $\widehat{ \ms G}$ -- whose Lie algebra is $\widehat{\mk  g}$ -- be the central extension of $\ms G$ corresponding to this cocycle, so that
$$
0\rightarrow \mathbb R\rightarrow \widehat{\mk  g} \xrightarrow{\pi} \mk  g.
$$
As we noticed, every connection defines a splitting of this sequence, that is a way to write
$\widehat{\mk  g}$ as $\mathbb R \oplus \mk  g$. 

Dually, we consider the vector space $\widehat{\mk g}^\dual$ defined by the exact sequence
$$
0\rightarrow \mk  g^\dual \xrightarrow{i}\widehat{\mk  g}^\dual \rightarrow \mathbb R,
$$
with the duality with $\widehat{\mk g}$ so that $\langle \gamma, i(\beta)\rangle=\langle \pi(\gamma), \beta\rangle$.

It follows that the space $\mathcal D$ of all $\sln$ connections on $P$ can be embedded in the space of such splittings which is in turn identified with the affine hyperplane $\DD$ in $\widehat{\mk  g}^{\dual}$ defined by
$$
\DD\defeq \{\beta\in \widehat{ \mk g}^{\dual}\mid \langle Z,\beta\rangle=1\},
$$
where $Z\in\widehat{\mk  g}$ the generator of the centre. The hyperplane $\ms D$ has $\mk g^\dual$ as a tangent space. Observe that the embedding $\mathcal D\to\DD$ is equivariant under the affine action of $\Omega^1(\mathbb T, {\operatorname{End}_0}(P))=\mk  g^{\dual}\subset\widehat{\mk  g}^{\dual} $ as well as the coadjoint action of $\ms G$ itself. In particular, the above embedding is surjective and we now identify $\DD$ as the space of all $\sln$-connections on $P$. The coadjoint orbits of $\ms G$ on $\DD$ are those connections with the same holonomy. 

\subsubsection{The Poisson structure}\label{Poisson}

Since we are working in infinite dimension, we are only going to define the Poisson tensor on certain ``cotangent vectors" to $\DD$. In our context we consider the set $\DD^{\dual}\defeq \mk g=\Omega^0(\mathbb T,\End_0(P))$ of cotangent vectors where the duality is given by Formula \eqref{Dual}. Using these notation, the Poisson structure is described in the following way
\begin{definition}{\sc[Poisson structure for connections]}\label{def:poisson}
\begin{itemize}
\item The {\em Hamiltonian mapping} from $\DD^{\dual}$ to $\DD$ at a connection $\nabla$ is
$$
H\ :\ \alpha\mapsto \dd^\nabla\alpha.
$$
\item The {\em Poisson tensor} on $\DD^{\dual}$ at a connection $\nabla$ is $$
\Pi_{\nabla}(\alpha,\beta)\defeq \langle\alpha,H(\beta)\rangle=\int_{\mathbb T}\tr(\alpha \cdotp\dd ^\nabla\beta).
$$

\item We say a functional $F$ is an {\em observable} if its differential $\dd_\nabla F$ belongs to $\DD^{\dual}$ for all $\nabla$. The {\em Poisson bracket} of two observables is
$$
\{f,g\}\defeq \Pi(\dd f,\dd g)=\langle\dd f,H(\dd g)\rangle.
$$
\end{itemize}

\end{definition}

\rmks
\begin{enumerate}
\item The Poisson bracket can be defined for more general pair of functionals than observables. Observe first that the differential of functionals on a Fréchet space of sections of bundles -- for instance connections -- are distributions. Thus we can define the Poisson bracket of a general differentiable functional with an observable. 
For the purpose of this paper, we shall say that two functionals $f$ and $g$ form an {\em acceptable pair of observables} if their derivatives $\dd f$ and $\dd g$ are distributions with disjoint singular support, or equivalently if they can be written as
\begin{eqnarray*}
\dd\!f &=&F+f_0\\
\dd g&=&G+g_0,
\end{eqnarray*}
where $F$ and $G$ have disjoint support and $f_0, g_0$ are observables in the previous meaning. In this case, their Poisson bracket is defined as
$$
\{f,g\}(\nabla)=\Pi(f_0,g_0)+\langle F, H(g_0)\rangle -\langle G, H(f_0)\rangle.
$$
This Poisson bracket agrees with regularising procedures.
\item We further observe that if $\DD_\nabla$ is the space of connections with the same holonomy as $\nabla$ -- that is the coadjoint orbit of $\nabla$ -- then the tangent space of $\DD_\nabla$ at $\nabla$ is the vector space of exact 1-forms $\dd ^\nabla\left(\Omega^0(\mathbb T,\End_0(P))\right)$, and moreover the Poisson tensor on $\DD_\nabla$ is dual to the symplectic form $\omega$ defined by
$$
\omega(\dd ^\nabla \alpha,\dd ^{\nabla}\beta)\defeq \int_{\mathbb T}\tr(\alpha \cdotp\dd ^\nabla\beta).$$
Thus the symplectic leaves of this Poisson structure are connections with the same holonomy. One can furthermore check that this formalism agrees with what we expect from coadjoint orbits.

\end{enumerate}

\subsection{Drinfel'd-Sokolov reduction}

We now describe the Drinfel'd-Sokolov reduction. In the first paragraph we describe more precisely the group that we are going to work with in order to perform the reduction.

\subsubsection{Dual Lie algebras} Let $\mk n$ be the Lie algebra of $\ms N$ as defined above. Let $\mk u$ be the subspace of $\mk g^\dual$ given by
$$
\mk u\defeq \{A\in\Omega^1(\mathbb T,\End_0(P))\mid A(F_p)\subset F_p\}.
$$
We observe

\begin{proposition}
We have
$
\mk u=\{A\in \widehat{\mk g}^\dual\mid \langle \alpha,A\rangle=0, \ \ \forall \alpha\in\mk n\}
$
\end{proposition}
Thus if
$
\mk n^{\dual}\defeq \Omega^1(\mathbb T,\End_0(P))/\mk u,
$
we have a duality $\mk n^{\dual}\times\mk n\to\mathbb R$ given by the map
$$
\langle\alpha,\beta\rangle\defeq \int_{\mathbb T}\tr(\alpha\beta).
$$
We now give another description of $\mk n^{\dual}$ more suitable for our purpose.
Let us first consider the natural projections
\begin{eqnarray}
\pi_p^+: \Hom(F_p,E/F_p)&\to&\Hom(F_p,E/F_{p+1}),\cr
\pi_p^-: \Hom(F_p,E/F_p)&\to&\Hom(F_{p-1},E/F_p).
\end{eqnarray}
Let 
$$
 M\defeq \left\{(u_1,\ldots ,u_{n-1})\in\bigoplus_{p=1}^{p=n-1}\Hom(F_p,E/F_p)\mid \pi^+_p(u_p)=\pi_{p+1}^-(u_{p+1})\right\}.
$$
We leave the reader check the following
\begin{proposition}
The map defined from $\mk n^{\dual}$ to $\Omega^1(\mathbb T,M)$ by
$$
A\to(\left.A\right\vert_{F_1},\ldots,\left.A\right\vert_{F_{n-1}})
$$
is an isomorphism.
\end{proposition}
\subsubsection{Drinfel'd-Sokolov reduction}
If $\nabla$ is a connection, we define the {\em slippage} -- denoted $\sigma(\nabla)$ -- of $\nabla$ as the element of $\Omega^1(\mathbb T,M)=\mk n^{\dual}$ given by
$$
(u_1,\ldots,u_p),
$$
where $u_p(X,v)=\alpha_p(\nabla_{X}v)$ and $\alpha_p$ is the projection from $E$ to $E/F_p$.

We are now going to define a canonical section of $\Omega^1(\mathbb T, M)$. We have a natural embedding 
$$
i_p:\Hom(F_p/F_{p-1},F_{p+1}/F_p)\to\Hom(F_p,F/F_p).
$$
Now observe that 
$$
\Omega^1(\mathbb T,\Hom(F_p/F_{p-1},F_{p+1}/F_p))= (K^2)^*\otimes(K^{2p-n-1})^*\otimes K^{2p-n+1}. 
$$
Thus, let
$$
{\bf I}_p\defeq i_p({\rm Id})\in (K^2)^*\otimes(K^{2p-n-1})^*\otimes K^{2p-n+1}.
$$
Finally, we set
$$
{\bf I}\defeq ({\bf I}_1,\ldots,{\bf I}_{n-1}),
$$
and we observe that ${\bf I}$ is invariant under the coadjoint action of $\ms N$.

\begin{theorem}{\sc[Drinfel'd-Sokolov reduction]}\label{DrinSok}
The map $\sigma$ is a moment map for the action of $\ms N$. Moreover $\DD_0=\sigma^{-1}({\bf I})$ and we thus obtain a Poisson structure on ${\ms X}_n(\mathbb T)$.
\end{theorem}

As a particular case of symplectic reduction, we briefly explain the construction of the Poisson bracket in our context of opers and non-slipping connections. 
If $f$ and $g$ are two functionals on the space of opers, they are observables if their pull back $F$ and $G$ on the space of non-slipping connections are observables and then their Poisson bracket is $\{f,g\}(D)\defeq \{F,G\}(\nabla)$ where $D$ is the oper associated with $\nabla$. 

\subsection{Opers and Frenet curves} 

\subsubsection{Curves associated to $\sln$-opers}
We recall that every oper $D$ give rise to a curve from $\mathbb R$ to $\pn$ {\em equivariant} under the holonomy that is a curve $$
 \xi:\mathbb R\to \pn,
 $$
 such that $\xi(t+1)=H(\xi(t))$ where $H$ is the holonomy.
 The construction runs as follows: the curve $\xi$ is given in projective coordinates by
 $$
 \xi\defeq [v_1,\ldots,v_n],
 $$
 where $\{v_1,\ldots, v_n\}$ are independent solutions of the equation $D\psi=0$. The curve $\xi$ is well defined up to the action of $\sln$. We call $\xi$ the curve {\em associated} to the oper.
 
\subsubsection{Hitchin opers}
Let us say an oper is {\em Hitchin} if it has trivial holonomy and can be deformed through opers with trivial holonomy to the trivial oper $\psi\mapsto \frac{\rm d^n\psi}{\dd t^n}$. Let us denote by ${\ms X}_n^0(\mathbb T)$ the space of Hitchin opers which by the previous section inherits a Poisson structure.

\subsubsection{Frenet curves}

We say a curve $\xi$ from $\mathbb T$ to $\pn$ is {\em Frenet}
if there exists a curve
$(\xi^{1},\xi^{2},\ldots,\xi^{n-1})$ defined on $\mathbb T$, called the {\em osculating flag curve},
with values in the flag variety such that for every $x$ in $\mathbb T$, $\xi(x)=\xi^{1}(x)$, and moreover
\begin{itemize}
\item For every pairwise distinct points $(x_{1},\ldots,x_{l})$ in $\mathbb T$ and positive integers $(n_1,\ldots,n_l)$ such that $$\sum_{i=1}^{i=l}n_i\leq n,$$ then the sum $$
\xi^{n_i}(x_i)+\ldots+\xi^{n_{l}}(x_{l})$$ is direct.

\item For every $x$ in $\mathbb T$ and positive integers $(n_1,\ldots,n_l)$ such that $$p=\sum_{i=1}^{i=l}n_i\leq n,$$
 then
$$
\lim_{\substack{(y_1,\ldots,y_l)\rightarrow x,\\ y_i
\text{all distinct}}}
\left(\bigoplus_{i=1}^{i=l}\xi^{n_i}(y_i)\right)=\xi^{p}(x).$$
\end{itemize}
We call $\xi^*\defeq \xi^{n-1}$ the {\em osculating hyperplane}.

Since the trivial connection is non-slipping with respect to the filtration given by osculating flags, we have the following obvious remark -- see also \cite{Fock:2006a} Section 9.12.

\begin{proposition}
Every smooth Frenet curve comes from a $\sln$-oper with trivial holonomy.
\end{proposition}

Conversely, we now prove 

\begin{proposition}
The curve associated to a Hitchin oper is Frenet.
\end{proposition}

\begin{proof} Let us first introduce some notation and definitions. 
A {\em weighted} $p$-tuple $X$ is a pair consisting of a $p$-tuple of pairwise distinct points $(x^1,\ldots,xp)$ in $\mathbb T$ -- called the {\em support} -- and a $p$-tuple of positive integers $(j_1,\ldots,j_p)$ so that $\sum_{1\leq k\leq p}j_k=n$. If $\eta
$ is a smooth curve defined on a subinterval $I$ of $\mathbb T$ with values in $\mathbb R^n\setminus\{0\}$, let $$
\hat\eta^{(p)}(x)\defeq \eta(x)\wedge\bu\eta(x)\wedge \ldots\wedge\eta^{(p-1)}(x)\in\Lambda^p(\mathbb R^n),
$$
where $\bu\eta$, $\eta^{(k)}$ denote respectively the derivative and $k$-th derivatives of $\eta$.
Moreover, if $X$ is a weighted $p$-tuple as above with support in $I$, let
$$
\hat\eta(X)\defeq \bigwedge_{1\leq k\leq p} \hat\eta^{(j^k)}(x^k)\in\Lambda^n(\mathbb R^n)=\mathbb R.
$$
Say that a weighted $p$-tuple is {\em degenerate} with respect to $\eta$ if $\hat\eta(X)=0$. Observe finally that being degenerate only depends on the projection of $\eta$ as a curve with values in $\pn$ and thus makes sense for curves with values in $\pn$.
By definition, a curve $\xi$ with values in $\pn$ is Frenet if it admits no degenerate weighted $p$-tuple.

Let us work by contradiction and assume that there exists a Hitchin oper so that the associated curve is not Frenet. Let $m$ be the smallest integer so that there exists a curve $\xi$ associated to an Hitchin oper which admits a degenerate $m$-tuple.

Let $\ms O_m$ be the set of Hitchin opers whose associate curve admits a degenerate $m$-tuple. By our standing assumption, $\ms O_m$ is non-empty, and moreover the trivial oper -- which corresponds to the Veronese embedding -- does not belong to $\ms O_m$.  We will now prove that  $\ms O_m$ is both open and closed which will yield a contradiction, since  ${\ms X}_n^0(\mathbb T)$ is connected.

\vskip 0.2 truecm
\noindent{{\sc Step 1: }\em The set $\ms O_m$ is open in ${\ms X}_n^0(\mathbb T)$}. 
Let $$X=\left((x^1,\ldots,x^m),(i_1,\ldots,i_m)\right)$$ be a degenerate $m$-tuple for the curve $\xi$ associated to the oper $D$. Without loss of generality, we can assume that $i_1$ is the greatest integer of all integers $j$ so that $\left((x^1,\ldots,x^m),(j,\ldots,i_m)\right) $ is degenerate. Let now $\eta$ be a lift (with values in $\mathbb R^n\setminus\{0\}$) of $\xi$ on an interval containing the support of $X$. Let us consider the function $f_D$ defined on a neighbourhood of $x^1$ by
$$
f_D\ :\ y\mapsto \hat\eta(X(y)),
$$
where $X(y)\defeq \left((y,x^2,\ldots,x^m),(i_1,\ldots,i_m)\right)$. We first prove that $\bu f_D(x^1)\not=0$. A computation yields
$$
\bu f_D(x^1)=\left(\hat\eta^{(i_1-2)}(x_1)\wedge\eta^{(i_1)}(x_1)\right)\wedge\left(\bigwedge_{2\leq j\leq m}\hat\eta^{i_j}(x_j)\right).
$$
Let us recall the following elementary fact of linear algebra. Let $u$, $v$ and $e_1\ldots,e_k$ are vectors in $\mathbb R^n$ so that
\begin{eqnarray}
u\wedge v\wedge e_1\wedge\ldots\wedge e_{k-1}&\not=&0,\label{eq:frenh1}\\
u\wedge e_1\wedge\ldots\wedge e_{k}&=&0,\label{eq:frenh2}
\end{eqnarray}
then
\begin{eqnarray}
v\wedge e_1\wedge\ldots\wedge e_{k}&\not=&0,\label{eq:frenh3}
\end{eqnarray}
Indeed, by Assertion \eqref{eq:frenh2}, $u$ belongs to the hyperplane $H$ generated by $(e_1,\ldots, e_{k})$. If Assertion \eqref{eq:frenh3} does not hold, then $v$ also belongs to $H$. Thus the vector space generated by $(u,v,e_1,\ldots, e_{k-1})$ also would lie in $H$, thus contradicting Assertion \eqref{eq:frenh1}.

By maximality of $i_1$, we know that $\hat\eta(Y)\not=0$, where $$Y=\left((x^1,\ldots,x^m),(i_1+1,i_2-1,\ldots,i_m)\right).$$ Since $f_D(x^1)=0$, the previous remark, with $u=\eta^{(i_1-1)}$, $v=\eta^{(i_1)}$ yields that 
$\bu f_D(x^1)\not=0$.

By transversality, it then follows that for $D'$ close to $D$ there exist $z$ close to $x^1$ so that $f_{D'}(z)=0$ and thus $D'\in \ms O_m$. 
\vskip 0.2 truecm
\noindent{{\sc Step 2: }\em The set $\ms O_m$ is closed in ${\ms X}_n^0(\mathbb T)$. } 

 Let $\{\xi^1_n\}_{n\in\mathbb N}$ be a sequence of curves associated to a sequence of opers in $\ms O_m$ converging to an oper $D$ associated to the curve $\xi$. Let $$
\{X_n=\left((x^1_n,\ldots,x^m_n),(j^1_n,\ldots,j^m_n)\right)\}_{n\in\mathbb N}.$$
be the corresponding sequence of degenerate $m$-tuples. 
We can extract a subsequence so that for every $i$, the sequence $\{j^i_n\}_{n\in \mathbb N}$ is constant and equal to $j^i$. After permutation of $\{1,\ldots, p\}$ and extracting a further subsequence, we can assume that there exists a $p$-tuple, with $p\leq m$, $$Y=\left((y^1,\ldots,y^p),(i^1,\ldots,i^p)\right),$$ and integers $k_1,\ldots, k_p$ so that 
\begin{enumerate}
\item $1=k_1\leq\ldots\leq k_p=m$
\item for all $i$, with $k_u\leq i<k_{u+1}$, $$\lim_{n\to\infty}(x^i_n)=y^u,$$
\item for all $v$, with $1\leq v\leq p$, $$i^v=\sum_{k_v\leq u<k_{v+1}}j^u.$$
\end{enumerate}

As an application of the Taylor formula, we have
$$
\hat\eta^{(p)}(x)\wedge\hat\eta^{(k)}(y)=(x-y)^{p.k}\hat\eta^{p+k}(x)+o((x-y)^k).
$$ 
It follows that for all $u$
$$
\lim_{n\to\infty}\left(\left(\prod_{v=k_u}^{k_{u+1}-1}\frac{1}{(x^v_n-y^u)^{N_v}}\right) \bigwedge_{v=k_u}^{k_{u+1}-1}\hat\eta^{(i^v)}(x_n^{v})\right)=\hat\eta^{(i_u)}(y^u),
$$
where $N_v=i^v\left(\sum_{w=k_u}^{v-1} i^w\right)$. 
In particular, $Y$ is degenerate for $\xi$. Thus $p=m$ by minimality and $D\in\ms O_m$.

\end{proof}

Finally, let us say a Frenet curve is {\em Hitchin} if it can be deformed through Frenet curves to the Veronese embedding. Then we obtain a consequence of the two previous propositions the following statement which seems to belong to the folklore but for which we could not find a proper reference.

\begin{theorem}
The map which associates to an $\sln$-oper its associated curve is a homeomorphism from the space of Hitchin opers to the space of Hitchin Frenet curves.
\end{theorem}
 
\subsection{Cross ratios and opers}
Let $\xi$ be a Frenet curve and $\xi^*$ be its associated osculating hyperplane curve. 
The {\em weak cross ratio} associated to this pair of curves is the function on $$
\mathbb T^{4*}\defeq \{(x,y,z,t)\in\mathbb T^4\mid z\not=y,x\not=t\},$$ defined by 
\begin{equation}
\bb_{\xi,\xi^*}(x,y,z,t)=\frac{\langle\widehat\xi(x)\vert\widehat\xi^*(y)\rangle\langle\widehat\xi(z)\vert\widehat\xi^*(t)\rangle}
{\langle\widehat\xi(z)\vert\widehat\xi^*(y)\rangle\langle\widehat\xi(x)\vert\widehat\xi^*(t)\rangle},\label{FRE}
\end{equation}
where for every $u$, we choose an arbitrary non-zero vector $\widehat\xi(u)$ and $\widehat\xi^*(u)$ respectively in $\xi(u)$ and $\xi^*(u)$. This weak cross ratio only depends on the oper $D$ and we shall denote it by $\bb_D$.

\subsubsection{Coordinate functions}\label{defFxy}
Let as in Section \ref{non-slip}, $K$ be the line bundle of $(-1/2)$-densities over $\mathbb T$ and $P\defeq J^{n-1}(K^{n-1})$ be the rank $n$ vector bundle of $(n-1)$ jets of sections of the bundle of $(-(n-1)/2)$-densities. We choose once and for all a trivialisation of $P$ given by $n$ fibrewise independent sections $\sigma_1,\ldots,\sigma_n$ of $P$ so that $F_p$ is generated by $\sigma_1,\ldots,\sigma_p$. 

Let $\nabla$ be a connection on $P$. Let $I$ be an interval in $\mathbb R$ with extremities $Y$ and $y$. We pull back $\nabla$, $P$ and $\sigma_i$ on $\mathbb R$ using the projection $$\pi:\mathbb R\to\mathbb R/\mathbb Z=\mathbb T.$$ We denote the pulled back objects by the same symbol. For any $y\in\mathbb R$, let $\sigma_y$ the $\nabla$-parallel section of $P$ on $I$ characterised by $\sigma_y(y)=\sigma_1(y)$. Similarly let $\sigma_Y^*$ be the $\nabla^*$ parallel section on $I$ of $P^*$ characterised by $\sigma^*_Y(Y)=\sigma_n^*(Y)$ where $(\sigma_1^*,\ldots,\sigma_n^*)$ is the dual basis to $(\sigma_1,\ldots,\sigma_n)$.

Then the function $t\mapsto\langle\sigma^*_Y(t),\sigma_y(t)\rangle$ is constant on $I$.

\begin{definition}{\sc[Coordinate function]}
The {\em coordinate function} associated to the points $Y$ and $y$ and the trivialisation of $P$, is the function
$$F_{Y,y}\ : \ \nabla\mapsto F_{Y,y}(\nabla)=\langle\sigma^*_Y(t),\sigma_y(t)\rangle, \hbox{ for }t\in I,$$
defined on the space of connections on $P$.
\end{definition}

We shall write $\sigma_Y^*\otimes \sigma_y=:\pp^{Y,y}=\pp^{Y,y}(\nabla)\in \Omega^0(\mathbb R,\End_0(P))$ so that 
\begin{eqnarray}
F_{Y,y}(\nabla)&=&\tr\left(\pp^{Y,y}\right)\label{Fp}
\end{eqnarray}
We then have
\begin{proposition} Assume that $\nabla$ has trivial holonomy. Then the coordinate function $F_{Y,y}$ only depends on the projections of $Y$ and $y$ in $\mathbb T$. Moreover there exists a section $\ms p^{Y_0,y_0}\in \Omega^0(\mathbb T, \operatorname{End}_0(P))$, so that $\ms p^{Y,y}$ is the pull back of $\ms p^{Y_0,y_0}$ .
\end{proposition}
\begin{proof} Let $Y_0$ and $y_0$ the projections of $Y$ and $y$. Since $\nabla$ has trivial holonomy, we may find parallel sections $\eta_{y_0}$ and $\eta^*_{Y_0}$, so that $\eta_{y_0}(y_0)=\sigma_1(y_0)$ and $\eta^*_{Y_0}(Y_0)=\sigma_1(Y_0)$. Then, $\sigma_y=\pi^*(\eta_{y_0})$ and $\sigma^*_Y=\pi^*(\eta_{y_0})$. Thus $F_{Y,y}(\nabla)=\langle\eta^*_{Y_0}(t),\eta_{y_0}(t)\rangle$. The first part of result follows. For the second part, we take $\ms p^{Y_0,y_0}=\eta^*_{Y_0}\otimes \eta_{y_0}$.
\end{proof}

\subsubsection{Differential of coordinate functions}
Our aim in that paragraph is to compute the differential of $F_{Y,y}$, where $Y,y$ belong to an interval $I$.
\begin{proposition}\label{PoiXx}
Let $\nabla$ be a connection. Let $y_0$ be a point in $\mathbb R\setminus I$. Let $\alpha$ be an element of $\Omega^1\left(\mathbb T,\End_0(P)\right)$. Then 
\begin{eqnarray}
\langle{\dd}_\nabla F_{Y,y},\alpha\rangle=\int_{\mathbb R}\psi^{Y,y,y_0}\tr\left(\pp^{Y,y}\pi^*(\alpha)\right),\label{dF}
\end{eqnarray}
where $\psi^{Y,y,y_0}(s)\defeq [y_0s,Yy]$.
 \end{proposition}
 We can observe that the right hand side of Equation \eqref{dF} does not depend on the choice of $y_0\in\mathbb R\setminus I$. Indeed, by the cocycle identity, $\psi^{Y,y,x}-\psi^{Y,y,z}$ is constant and equal to $[xz,Yy]=0$, if $x,z\not\in I$.

 \begin{proof} Let $\beta$ de a primitive of $\pi^*\alpha$ on $I$ such that $\beta(y)=0$. Let $t\mapsto \nabla^t$ be a one parameter smooth family of connections with $\nabla^0=\nabla$ so that $$
 \left.\frac{\dd }{\dd t}\right\vert_{t=0}\nabla^t=\alpha.$$
 Let $G^t$ be the family of sections of $\End(P)$ so that $G^t(z)={\rm Id}$ and $(G^t)^*\nabla=\nabla^t$. Then by construction $$
 \left.\frac{\dd }{\dd t}\right\vert_{t=0}G^t=\beta.$$
 Moreover 
 $$
 F_{Y,y}(\nabla^t)=\langle \sigma_n^*(Y), G^t\left(\sigma_y(Y)\right)\rangle.
 $$
Thus
\begin{eqnarray}
\langle{\dd}_\nabla F_{Y,y},\alpha\rangle&=&\langle \sigma^*_Y(Y), \beta(Y)\sigma_y(Y)\rangle.
\end{eqnarray}
Let $c(t)$ is a curve with value in $I$ so that $c(0)=y$ and $c(1)=Y$. Let $$f(t)=\left\langle\sigma^*_Y(c(t)),\beta(c(t))\sigma^*_y(c(t))\right\rangle.$$ Then
\begin{eqnarray}
\langle{\dd}_\nabla F_{Y,y},\alpha\rangle&=&f(1)-f(0)\cr
&=&\int_{0}^{1} \bu{f}(s)\,{\dd s}.\end{eqnarray}
Since $\sigma^*_Y$ and $\sigma_y$ are parallel,
$$
\bu{f}(s)=\langle\sigma^*_Y(c(s)), \pi^*(\alpha(\bu{c}(s)).\sigma_y(c(s))\rangle,
$$ 
and we have, letting $J$ be the interval whose endpoints are $Y$ and $y$
\begin{eqnarray}
\langle{\dd}_\nabla F_{Y,y},\alpha\rangle&=&\sig(Y-y)\int_{J}\langle \sigma^*_Y, \pi^*(\alpha).\sigma_y\rangle\cr
&=&\sig(Y-y)\int_{J}\tr\left(\pp^{Y,y}.\pi^*(\alpha)\right).\label{prel1}
\end{eqnarray}
We finally deduce the result from Equation \eqref{prel1} and the fact that for any $y_0\not\in I$, we have 
$$
\sig(Y-y)\int_{J}\gamma=\int_{\mathbb R}\psi^{Y,y,y_0}\gamma.\qedhere
$$
\end{proof}

\subsection{Poisson brackets on the space of connections} Since $F_{X,x}$ is not an observable in the sense of Paragraph \ref{Poisson}, we first need to regularise these functions

\subsubsection{Regularisation} Let $\mu$ and $\nu$ two $C^\infty$ measures compactly supported in a bounded interval $]a,b[$ of $\mathbb R$. Let us consider the function
$$
F_{\mu,\nu}\defeq \int_{\mathbb R^2}F_{X,x}\,\dd\mu \cdotp \dd\nu(X,x).
$$
We consider this function as defined on the space of connections over the bundle $P\to \mathbb T$.
We obviously have

\begin{proposition}
Let $\{(\mu_n,\nu_n)\}_{n\in\mathbb N}$ be two sequences of measures weakly converging to $(\mu,\nu)$, then $\{F_{\mu_n,\nu_n}\}_{n\in\mathbb N}$ converges uniformly on every compact to $F_{\mu,\nu}$.
\end{proposition}

We say the sequence $\{(\mu_n,\nu_n)\}_{n\in\mathbb N}$ is {\em regularising} for the pair $(X,x)$ if $\mu_n,\nu_n$ are smooth measures weakly converging to the Dirac measures supported at $X$ and $x$ respectively.

\subsubsection{Poisson brackets of regularisation}

We now have 
\begin{proposition} For any pair of smooth measures $(\mu,\nu)$ with compact support, $F_{\mu,\nu}$ is an observable.
Let $(\mu,\nu)$ and $(\bar\mu,\bar\nu)$ be two pairs of $C^\infty$ measures on $\mathbb R$. Then the Poisson bracket $\{F_{\mu,\nu},F_{\bar\mu,\bar\nu}\}$ is equal to
\begin{eqnarray}
\sum_{m\in\mathbb Z}\int_{\mathbb R^4}[m(Y)m(y),Xx] \left(F_{X,y}F_{Y,x}-\frac{1}{n^2}F_{X,x}F_{Y,y}\right)\ \!\dd \Lambda(X,x,Y,y),
\end{eqnarray}
where $m(u)=u+m$ and $\Lambda=\mu\otimes\nu\otimes\bar\mu\otimes\bar\nu$. In particular if all measures are supported on $[0,1]$, then the bracket $\{F_{\mu,\nu},F_{\bar\mu,\bar\nu}\}$ is equal to
\begin{eqnarray}
\int_{\mathbb [0,1]^4}[Yy,Xx] \left(F_{X,y}F_{Y,x}-\frac{1}{n^2}F_{X,x}F_{Y,y}\right)\ \!\dd\Lambda (X,x,Y,y).
\end{eqnarray}

\end{proposition}
\begin{proof} By Proposition \ref{PoiXx}, we have that if $a$ does not belong to the union $K$ of the supports of $\mu$ and $\nu$
\begin{eqnarray*}
\langle\dd F_{\mu,\nu},\alpha\rangle&=&\int_{\mathbb R^2}\psi^{X,x,a}\tr\left(\pp^{X,x}\pi^*\alpha\right)\ \dd\mu \cdotp\dd\nu(X,x).
\end{eqnarray*}
Let us denote by  ${\ms C}_0\defeq {\ms C}-\frac{1}{n}\tr({\ms C}){\rm Id}$ the trace free part of the endomorphism $\ms C$.
For any $s$ in $\mathbb R$, let
\begin{eqnarray}
\Lambda_{\mu,\nu}(s)\defeq \int_{\mathbb R^2}\psi^{X,x,a}(s)\pp_0^{X,x}(s)\ \dd\mu \cdotp\dd\nu(X,x).\label{Poi10}
\end{eqnarray}
Observe that $\Lambda_{\mu,\nu}\in\Omega^0(\mathbb R,\End_0(P))$ and the the support of $\Lambda_{\mu,\nu}$ is included in $K$. Let us trivialise $P$ using the connection $\nabla$. Then let
$$
G_{\mu,\nu}(s)\defeq \sum_{m\in\mathbb Z}\Lambda_{\mu,\nu}(s+m).
$$
Then $G_{\mu,\nu}(s)$ is periodic and thus of the form $\pi^*\beta$, with $\beta\in\Omega^0(\mathbb T,P)$. Moreover
\begin{eqnarray}
\int_{\mathbb T}\tr(\beta.\alpha)&=&\int_0^1\tr(\pi^*\beta.\pi^*\alpha)\cr
&=&\int_{\mathbb R}\tr(\Lambda_{\mu,\nu}.\pi^*\alpha)\cr
&=&\langle \dd F_{\mu,\nu},\alpha\rangle.\label{dF2}
\end{eqnarray}
It follows by Equation \eqref{dF2}, that 
\begin{eqnarray}
\dd F_{\mu,\nu}(s)=\beta\in \mk{g}=\ms D^0.
\end{eqnarray}
In particular, according to Definition \ref{def:poisson}, $F_{\mu,\nu}$ is an observable.
From Equation \eqref{Poi10}, we have
\begin{eqnarray*}
\Lambda_{\mu,\nu}(s)
&=&-\int_{-\infty}^s\!\int_s^\infty
\pp_0^{X,x}(s)\ \dd\mu \cdotp\dd\nu(X,x)\cr & &+\int_s^\infty\!\int_{-\infty}^s\pp_0^{X,x}(s)\ \dd\mu \cdotp\dd\nu(X,x).
\end{eqnarray*}
For any smooth probability measure $\xi$ let us write $\dd\xi=\d\xi\dd\lambda$ where $\lambda$ is the Lebesgue measure.
Then, since $\pp^{X,x}$ is parallel, we have
\begin{eqnarray*}
\nabla_{\partial_t}\Lambda_{\mu,\nu}(s)&=&-\d\mu(s)\int_s^\infty\pp_0^{s,x}(s)\ \dd\nu(x)+\d\nu(s)\int_{-\infty}^s\pp_0^{X,s}(s)\ \dd\mu(X) \\
& &-\d\mu(s)\int_{-\infty}^s\pp_0^{s,x}(s)\ \dd\nu(x)+\d\nu(s)\int_s^\infty\pp_0^{X,s}(s)\ \dd\mu(X)\\
&=&\d\nu(s)\int_{\mathbb R}\pp_0^{X,s}\dd\mu(X)-\d\mu(s)\int_{\mathbb R}\pp_0^{s,x}\ \dd\nu(x).
\end{eqnarray*}
It follows that
\begin{eqnarray}
& &\tr\left(\Lambda_{\mu,\nu}(s+m)\nabla_{\partial_t} \Lambda_{\bar\mu,\bar\nu}(s)\right)\cr&=&\d{\bar\nu}(s)\int_{\mathbb R^3}\psi^{X,x,a}(s+m)\tr\left(\pp_0^{Y,s}\pp_0^{X,x}\right)\ \dd\mu \cdotp\dd\nu \cdotp\dd\bar\mu(X,x,Y)\cr
&-&\d{\bar\mu}(s)\int_{\mathbb R^3}\psi^{X,x,a}(s+m)\tr\left(\pp_0^{s,y}\pp_0^{X,x}\right)\ \dd\mu \cdotp\dd\nu \cdotp\dd\bar\nu(X,x,y).\label{Poi2}
\end{eqnarray}
We can now compute the Poisson bracket as defined in Definition \ref{def:poisson}:
\begin{eqnarray}
\{F_{\mu,\nu},F_{\bar\mu,\bar\nu}\}&=&\int_{\mathbb R}\tr\left(\Lambda_{\mu,\nu}(s)\pi^*\left(\nabla_{\partial_t} \dd F_{\bar\mu,\bar\nu}(s)\right)\right)\ \dd\lambda(s)\cr
&=&\sum_{m\in\mathbb Z}\int_{\mathbb R}\tr\left(\Lambda_{\mu,\nu}(s)\nabla_{\partial_t} \Lambda_{\bar\mu,\bar\nu}(s+m)\right)\ \dd\lambda(s)\cr
&=&\sum_{m\in\mathbb Z}\int_{\mathbb R}\tr\left(\Lambda_{\mu,\nu}(s+m)\nabla_{\partial_t} \Lambda_{\bar\mu,\bar\nu}(s)\right)\ \dd\lambda(s).
\end{eqnarray}
Using Equation \eqref{Poi2}, we get that
\begin{eqnarray}
& &\{F_{\mu,\nu},F_{\bar\mu,\bar\nu}\}\cr
&=&\sum_{m\in\mathbb Z}\int_{\mathbb R^4}\psi^{X,x,a}(s+m)\tr\left(\pp_0^{Y,s}\pp_0^{X,x}\right)\ \dd\Lambda(X,x,Y,s)\label{lin1}\\
&-&\sum_{m\in\mathbb Z}\int_{\mathbb R^4}\psi^{X,x,a}(s+m)\tr\left(\pp_0^{s,y}\pp_0^{X,x}\right)\ \dd\Lambda(X,x,s,y).\label{lin2}
\end{eqnarray}
Using the dummy changes of variables $s=y$ on line \eqref{lin1} and $s=Y$ on line \eqref{lin2}, we finally get
\begin{align}
&\{F_{\mu,\nu},F_{\bar\mu,\bar\nu}\}\cr
&=\sum_{m\in\mathbb Z}\int_{\mathbb R^4}(\psi^{X,x,a}(y+m)-\psi^{X,x,a}(Y+m))\tr\left(\pp_0^{Y,y}\pp_0^{X,x}\right)\ \dd\lambda(X,x,Y,y)\cr
&=\sum_{m\in\mathbb Z}\int_{\mathbb R^4}[(Y+m)(y+m),Xx]\tr\left(\pp_0^{Y,y}\pp_0^{X,x}\right)\ \dd\Lambda(X,x,Y,y)\label{Poi4}
\end{align}
We conclude the proof of the proposition by remarking that
$$
\tr\left(\pp^{X,x}\pp^{Y,y}\right)=\tr\left(\pp^{X,y}\right)\tr\left(\pp^{Y,x}\right),
$$
and thus
\begin{equation}
\tr\left(\pp_0^{X,x}\pp_0^{Y,y}\right)=\tr\left(\pp^{X,y}\right)\tr\left(\pp^{Y,x}\right)-\frac{1}{n^2}\tr\left(\pp^{X,x}\right)\tr\left(\pp^{Y,y}\right).\label{Poi5}
\end{equation}
Combining Equations \eqref{Poi4} and \eqref{Poi5} yields the result.
\end{proof}
As corollaries, we obtain
\begin{corollary}
Let $(\mu_u,\nu_n)$ and $(\bar\mu_n,\bar\nu_n)$ be regularising sequences for $(X,x)$ and $(Y,y)$ respectively. Assume that $\{X,x,Y,y\}\subset ]0,1[$, then
$$
\lim_{n\to\infty}\left(\{F_{\mu_n,\nu_n},F_{\bar\mu_n,\bar\nu_n}\}\right)=[Yy,Xx] \left(F_{X,y}F_{Y,x}-\frac{1}{n^2}F_{X,x}F_{Y,y}\right).
$$ 
\end{corollary}
\begin{corollary}\label{FxxPois}
Let $(X,x,Y,y)$ be a quadruple of pairwise distinct points. Then $(F_{X,x},F_{Y,y})$ is a pair of acceptable observables. Moreover, 
$$
\{F_{X,x},F_{Y,y}\}=[Yy,Xx] \left(F_{X,y}F_{Y,x}-\frac{1}{n^2}F_{X,x}F_{Y,y}\right).
$$ 
\end{corollary}
This last corollary interprets the swapping algebra as an algebra of ``observables" on the space of connections.
\subsection{Drinfel'd-Sokolov reduction and the multifraction algebra}

We introduced in Paragraph \ref{defFxy} functions of connections depending on the choice of a trivialisation of $P$. We now introduce functions that only depend on the associated oper and do not rely on the choice of the trivialisation of $P$. 

We first relate cross ratios with our previously defined coordinate functions.

\subsubsection{Cross ratios}
The following proposition follows at once from the definitions.

\begin{proposition}\label{crossFxx} Let $D$ be a Hitchin oper associated to the connection $\nabla$ with trivial holonomy. Let $X,x,Y,y$ be a quadruple of pairwise distinct points of $\mathbb T$. Let $\tilde X,\tilde x,\tilde Y,\tilde y$ be lifts of $X,x,Y,y$ in $\mathbb R$, then
$$
\bb_D(X,x,Y,y)=\frac{F_{\tilde X,\tilde y}(\nabla).F_{\tilde Y,\tilde x}(\nabla)}{F_{\tilde X,\tilde x}(\nabla).F_{\tilde Y,\tilde y}(\nabla)}.
$$
\end{proposition}

\subsubsection{The main Theorem}

We can now prove our main theorem which relates the Poisson structure on the space of opers and the multifraction algebra.

\begin{theorem}\label{DF-swap}
Let $(X_0,x_0,Y_0,y_0,X_1,x_1,Y_1,y_1)$ be pairwise distinct points. Then the cross fractions $[X_0;x_0;Y_0;y_0]$ and $[X_1;x_1;Y_1;y_1]$ defines a pair of acceptable observables whose Poisson bracket with respect to the Drinfel'd-Sokolov reduction coincide with their Poisson bracket in the multifraction algebra.
\end{theorem}

\begin{proof} This is an immediate consequence of Proposition \ref{crossFxx} and Corollary \ref{FxxPois}, as well as the definition of the Poisson structure coming from the symplectic reduction in Theorem \ref{DrinSok}. 
\end{proof}
\section{Appendix: existence of vanishing sequences}\label{sec:vanishexist}

We prove the existence of vanishing sequences.

\begin{definition}{\sc [Separability in groups]}
Let $G$ be a group. The group $G$ is said to be {\em subgroup separable}, if given any finitely generated subgroup $H$ in $G$, given any $g\in G$ and $h\not\in Hg$, there exists a finite index subgroup $G_0$ in $G$ such that if $\pi$ is the projection of $G$ onto $G/G_0$, then
$$
\pi(h)\not\in \pi(Hg).
$$

Let $G$ be a group. The group $G$ is said to be {\em double coset separable}, if given any finitely generated subgroups $H$ and $K$ in $G$, given any $g\in G$ and $h\not\in HgK$, there exists a finite index subgroup $G_0$ in $G$ such that if $\pi$ is the projection of $G$ onto $G/G_0$, then
$$
\pi(h)\not\in \pi(HgK).
$$
\end{definition}

Observe that a double coset separable group is then subgroup separable and residually finite.
G.~Niblo proved in \cite{Niblo:1992uy}

\begin{theorem}
A surface group is double coset separable.
\end{theorem}

As an immediate consequence, since $\pi_1(S)$ is countable, we have
\begin{corollary}
Vanishing sequences exist.
\end{corollary}


\begin{thebibliography}{10}


\bibitem{Atiyah:1983}
Michael~F. Atiyah and Raoul Bott, \emph{{The Yang--Mills equations over Riemann
 surfaces}}, Philos. Trans. Roy. Soc. London Ser. A \textbf{308} (1983),
 no.~1505, 523--615.

\bibitem{Bourdon:1996}
Marc Bourdon, \emph{{Sur le birapport au bord des $\rm CAT(-1)$-espaces}},
 Publ. Math. Inst. Hautes {\'E}tudes Sci. (1996), no.~83, 95--104.

\bibitem{Bridgeman:2015wn}
	Martin~J. Bridgeman,
	 {\emph{The Poisson Bracket of Length functions in the Hitchin Component}},
	arxiv preprint:1502.05975v1.

\bibitem{Bridgeman:2015ba}
Martin~J. Bridgeman, Richard Canary, Fran{\c c}ois Labourie, and Andrés 
 Sambarino, \emph{{The pressure metric for Anosov representations}}, Geometric
 And Functional Analysis \textbf{25} (2015), no.~4, 1089--1179.

\bibitem{Dickey:1997un}
L.~A. Dickey, \emph{{Lectures on classical $W$-algebras}}, Acta Applicandae
 Mathematicae. An International Survey Journal on Applying Mathematics and
 Mathematical Applications \textbf{47} (1997), no.~3, 243--321.

\bibitem{Drinfelcprimed:1981ua}
Vladimir Drinfel'd and V.~V. Sokolov, \emph{{Equations of Korteweg-de Vries type,
 and simple Lie algebras}}, Dokl. Akad. Nauk SSSR \textbf{258} (1981), no.~1,
 11--16.

\bibitem{Fock:2006a}
Vladimir~V. Fock and Alexander~B. Goncharov, \emph{{Moduli spaces of local
 systems and higher {T}eichm{\"u}ller theory}}, Publ. Math. Inst. Hautes
 {\'E}tudes Sci. (2006), no.~103, 1--211.

\bibitem{Goldman:1984}
William~M. Goldman, \emph{{The symplectic nature of fundamental groups of
 surfaces}}, Advances in Mathematics \textbf{54} (1984), no.~2, 200--225.

\bibitem{Goldman:1986}
\bysame, \emph{{Invariant functions on Lie groups and Hamiltonian flows of
 surface group representations}}, Inventiones Mathematicae \textbf{85} (1986),
 no.~2, 263--302.

\bibitem{Govindarajan:1995cr}
Suresh Govindarajan, \emph{{Higher-dimensional uniformisation and
 $W$-geometry}}, Nuclear Physics. B \textbf{457} (1995), no.~1-2, 357--374.

\bibitem{Govindarajan:1995im}
Suresh Govindarajan and T~Jayaraman, \emph{{A proposal for the geometry of
 $W_n$-gravity}}, Physics Letters. B \textbf{345} (1995), no.~3, 211--219.

\bibitem{Guha:2007wf}
Partha Guha, \emph{{Euler--Poincar{\'e} Flows on $\mk{sl}(n)$-Opers and
 Integrability}}, Acta Applicandae Mathematicae: An International Survey
 {\ldots} (2007).

\bibitem{Guichard:2008tu}
Olivier Guichard, \emph{{Composantes de Hitchin et repr{\'e}sentations
 hyperconvexes de groupes de surface}}, Journal of Differential Geometry
 \textbf{80} (2008), no.~3, 391--431.

\bibitem{Guichard:2012eg}
Olivier Guichard and Anna Wienhard, \emph{{Anosov representations: domains of
 discontinuity and applications}}, Inventiones Mathematicae \textbf{190}
 (2012), no.~2, 357--438.

\bibitem{Hitchin:1992es}
Nigel~J Hitchin, \emph{{Lie groups and {T}eichm{\"u}ller space}}, Topology
 \textbf{31} (1992), no.~3, 449--473.

\bibitem{Labourie:2006}
Fran{\c c}ois Labourie, \emph{{Anosov flows, surface groups and curves in
 projective space}}, Inventiones Mathematicae \textbf{165} (2006), no.~1,
 51--114.

\bibitem{Labourie:2005}
\bysame, \emph{{Cross ratios, surface groups, {${\rm PSL}(n,{\bf R})$} and
 diffeomorphisms of the circle}}, Publ. Math. Inst. Hautes {\'E}tudes Sci.
 (2007), no.~106, 139--213.

\bibitem{Labourie:2008tv}
\bysame, \emph{{Cross ratios, Anosov representations and the energy functional
 on Teichm\"uller space}}, Annales Scientifiques de l'Ecole Normale
 Sup{\'e}rieure. Quatri{\`e}me S{\'e}rie \textbf{41} (2008), no.~3, 437--469.

\bibitem{Labourie:2010ev}
\bysame, \emph{{An algebra of observables for cross ratios}}, Comptes Rendus de
 l'Acad{\'e}mie des Sciences. S{\'e}rie I. Math{\'e}matique \textbf{348}
 (2010), no.~9-10, 503--507.

\bibitem{Labourie:2013ka}
\bysame
\emph{{Lectures on representations of surface groups}},
{European Mathematical Society (EMS), Z\"urich},
 (2013), 
 {Z\"urich Lectures in Advanced Mathematics}

\bibitem{Ledrappier:1995}
Fran{\c c}ois Ledrappier, \emph{{Structure au bord des vari{\'e}t{\'e}s {\`a}
 courbure n{\'e}gative}}, vol.~13, Saint, 1995.

\bibitem{Magri:1978gh}
Franco Magri, \emph{{A simple model of the integrable Hamiltonian equation}},
 Journal of Mathematical Physics \textbf{19} (1978), no.~5, 1156--1162.

\bibitem{Niblo:1992uy}
Graham Niblo, \emph{{Separability properties of free groups and surface groups}}, J.
 Pure Appl. Algebra \textbf{78} (1992), no.~1, 77--84.

\bibitem{Otal:1990th}
Jean-Pierre Otal, \emph{{Le spectre marqu{\'e} des longueurs des surfaces {\`a}
 courbure n{\'e}gative}}, Annals of Mathematics \textbf{131} (1990), no.~1,
 151--162.

\bibitem{Otal:1992}
\bysame, \emph{{Sur la g\'eometrie symplectique de l'espace des g\'eod\'esiques
 d'une vari\'et\'e \`a courbure n\'egative}}, Rev. Mat. Iberoamericana
 \textbf{8} (1992), no.~3, 441--456.

\bibitem{Sambarino:2014kc}
Andrés Sambarino, \emph{{Hyperconvex representations and exponential growth}},
 Ergodic Theory Dynam. Systems \textbf{34} (2014), no.~3, 986--1010.

\bibitem{Sambarino:2014jv}
\bysame, \emph{{Quantitative properties of convex representations}},
 Commentarii Mathematici Helvetici \textbf{89} (2014), no.~2, 443--488.

\bibitem{Segal:1991}
Graeme Segal, \emph{{The geometry of the KdV equation}}, Internat. J. Modern
 Phys. A \textbf{6} (1991), no.~16, 2859--2869.

\bibitem{vanMoerbeke:1998wl}
Pierre van Moerbeke, \emph{{Alg{\`e}bres $W_n$ et {\'e}quations
 non--lin{\'e}aires}}, Ast{\'e}risque (1998), no.~252, Exp.\ No.\ 839, 3,
 105--129.

\bibitem{Witten:to}
Edward Witten, \emph{{Surprises with topological field theories}}, cdsweb.cern.ch.

\bibitem{Wolpert:1982eo}
Scott Wolpert, \emph{{The Fenchel-Nielsen deformation}}, Annals of Mathematics
 \textbf{115} (1982), no.~3, 501--528.

\bibitem{Wolpert:1983td}
\bysame, \emph{{On the Symplectic Geometry of Deformations of a
 Hyperbolic Surface}}, Annals of Mathematics \textbf{117} (1983), no.~2,
 207--234.
\end{thebibliography}

\providecommand{\bysame}{\leavevmode\hbox to3em{\hrulefill}\thinspace}
\providecommand{\MR}{\relax\ifhmode\unskip\space\fi MR }
\providecommand{\MRhref}[2]{%
 \href{http://www.ams.org/mathscinet-getitem?mr=#1}{#2}
}
\providecommand{\href}[2]{#2}

\auteur

\end{document}